\documentclass{amsart}
\usepackage{amssymb,amsmath}
\usepackage{color}
\definecolor{darkgreen}{cmyk}{1,0,1,.2}
\definecolor{m}{rgb}{1,0.1,1}
\definecolor{green}{cmyk}{1,0,1,0}
\definecolor{test}{rgb}{1,0,0}   
\definecolor{cmyk}{cmyk}{0,1,1,0}

\newtheorem{Equation}{}[section]
\newtheorem{example}[Equation]{Example}
\newtheorem{theorem}[Equation]{Theorem}
\newtheorem{proposition}[Equation]{Proposition}
\newtheorem{lemma}[Equation]{Lemma}
\newtheorem{corollary}[Equation]{Corollary}

\newtheorem{definition}[Equation]{Definition}
\newtheorem{conjecture}[Equation]{Conjecture}
\newtheorem{remark}[Equation]{Remark}

\def\ch{\operatorname{ch}}

\def\dia{\operatorname{dia}}

\def\Iso{\operatorname{Iso}}

\def\vol{\operatorname{vol}}

\def\End{\operatorname{End}}

\def\Hom{\operatorname{Hom}}

\def\oH{\operatorname{H}}

\def\Im{\operatorname{Im}}

\def\Ind{\operatorname{Ind}}
\def\oK{\operatorname{K}}
\def\Ker{\operatorname{Ker}}

\def\Tr{\operatorname{Tr}}
\def\tr{\operatorname{tr}}

\def\C{\mathbb C}
\def\D{\mathbb D}

\def\R{\mathbb R}

\def\Z{\mathbb Z}
\def\T{\mathbb T}
\def\L{\mathbb L}
\def\H{\mathbb H}

\def\cA{{\mathcal A}}
\def\cB{{\mathcal B}}

\def\maB{{\mathcal B}}
\def\maC{{\mathcal C}}

\def\maG{{\mathcal G}}
\def\cG{{\mathcal G}}

\def\cH{{\mathcal H}}

\def\maS{{\mathcal S}}

\def\cU{{\mathcal U}}

\def\what{\widehat}
\def\wtit{\widetilde}

\def\te{\tilde e}

\def\wcB{\widetilde{\mathcal B}}

\def\wL{{\widetilde L}}

\def\cf{{\check f}}

\def\wf{\widetilde{f}} 
\def\cg{{\check g}}
\def\wg{\widetilde{g}}

\def\wN{\widetilde{N}}

\def\tvarphi{\widetilde{\varphi}}

\def\hf{\widehat{f}}

\def\dd{\displaystyle}
\def\pa{\partial}

\def\tri{\triangle}

\begin{document}



\title[ Twisted Higher Signatures for Foliations] 
{The Twisted Higher Harmonic Signature for Foliations} 


\author[M.-T. Benameur]{Moulay-Tahar Benameur}
\address{LMAM, UMR 7122 du CNRS, Universit\'{e} Paul Verlaine-Metz}
\email{benameur@univ-metz.fr}
\author[J.  L.  Heitsch]{James L.  Heitsch}
\address{Mathematics, Statistics, and Computer Science, University of Illinois at Chicago} 
\email{heitsch@math.uic.edu}
\address{ Mathematics, Northwestern University}

\begin{abstract}
We prove that the higher harmonic signature  of an even dimensional oriented Riemannian foliation $F$ of a compact Riemannian manifold $M$ with coefficients in a leafwise $U(p,q)$-flat complex bundle is a leafwise homotopy invariant.  We also prove the leafwise homotopy invariance of the twisted higher Betti classes.  Consequences for the Novikov conjecture for foliations and for groups are investigated.
\end{abstract}

\maketitle

\section{Introduction}

In this paper, we prove that the higher harmonic signature, $\sigma(F,E)$, of a $2\ell$ dimensional oriented Riemannian foliation $F$ of a compact  Riemannian manifold $M$, twisted by a leafwise flat complex bundle $E$ over $M$, is a leafwise homotopy invariant.   We also derive important consequences for the Novikov conjecture for foliations and for groups.  We assume that $E$ admits a non-degenerate {\em possibly indefinite} Hermitian metric  which is preserved by the leafwise flat structure.  As explained in \cite{Gromov}, this includes the leafwise $O(p,q)$-flat and the leafwise symplectic-flat cases.   We assume that the projection onto the  twisted leafwise harmonic forms in dimension $\ell$ is transversely smooth.  This is true whenever the leafwise parallel translation on $E$  defined by the flat structure is a bounded map, in particular whenever the preserved metric on $E$ is positive definite.   It is satisfied for important examples, e.g., the examples of Lusztig \cite{Lusztig} which proved the Novikov conjecture for free abelian groups, and it is always true whenever $E$ is a bundle associated to the normal bundle of the foliation.   In particular, the smoothness assumption is fulfilled for the (untwisted) leafwise signature operator.

Any metric on $M$ determines a metric on each leaf $L$ of $F$, so also on all covers of $L$.  The bundle $E \, | \, L$ can be pulled back to a  flat bundle (also denoted $E$) on any cover of $L$.  These leafwise metrics and the leafwise flat bundle $E$ determine leafwise Laplacians  $\Delta^E$ and Hodge $*$ operators on the differential forms on $L$ with coefficients in $E \,|\, L$, as well as on  all covers of $L$.  The Hodge operator determines an involution  which commutes with $\Delta^E$, so $\Delta^E$ splits as a sum $\Delta^E = \Delta ^{E,+} + \Delta^{E,-}$, in particular in dimension $\ell$,  $\Delta^E_{\ell} = \Delta^{E,+}_{\ell} + \Delta^{E,-}_{\ell}$.  To each leaf $L$ of $F$, we associate the formal difference of the (in general, infinite dimensional) spaces $\Ker( \Delta^{E,+}_{\ell} )$ and $\Ker(\Delta^{E,-}_{\ell})$ on $\wL$, the simply connected cover of $L$.  We assume that the Schwartz kernel of the projection onto $\Ker( \Delta^{E}_{\ell}) = \Ker( \Delta^{E,+}_{\ell} ) \oplus \Ker(\Delta^{E,-}_{\ell})$ varies smoothly transversely.   
Roughly speaking, transverse smoothness means that the $ \Ker( \Delta^{E,\pm}_{\ell})$ are ``smooth bundles over the leaf space of $F$".  We define a Chern-Connes character $\ch_a$ for such bundles which takes values in the Haefliger cohomology of $F$. The  higher harmonic signature of $F$ is defined as 
$$
\sigma(F,E) = \ch_a(\Ker(\Delta^{E,+}_{\ell}) )  -  \ch_a 
(\Ker(\Delta^{E,-}_{\ell})).
$$
Our main theorem is the following.

\medskip\noindent
{\bf Theorem \ref{main}. } {\em Suppose that  $M$ is a compact  Riemannian manifold,  with oriented Riemannian foliation $F$ of dimension $2\ell$, and that $E$ is a leafwise 
flat complex bundle over $M$ with a (possibly indefinite) non-degenerate Hermitian metric which is preserved by the leafwise flat structure.  Assume that the projection onto  $\Ker( \Delta^{E}_{\ell} )$ for the associated foliation $F_s$ of the homotopy groupoid of $F$  is transversely smooth.  Then $\sigma(F,E)$ is a leafwise homotopy invariant.}

\medskip

In particular, suppose that $M'$,  $F'$,  and $E'$ satisfy the hypothesis of Theorem  \ref{main},
and that $f:M \to M'$ is a  leafwise homotopy equivalence, which is leafwise oriented.  Set $E = f^*(E')$ with the induced leafwise flat structure and preserved metric. Then $f$ induces an isomorphism $f^*$ from the Haefliger cohomology of $F'$ to that of $F$, and 
$$
f^*(\sigma(F',E')) = \sigma(F,E).
$$
A priori, $\sigma (F,E)$ depends on the metric on $M$. However, it is an immediate corollary of Theorem \ref{main}  that it is independent of this metric since 
the identity map is a leafwise homotopy equivalence between $(M,F;g_0)$ and $(M,F;g_1)$.  
In general, $\sigma (F,E)$ depends  on the flat structure and the metric 
on $E$, in particular on the splitting of $E = E^+ \oplus E^-$ into positive (resp. negative) definite sub bundles.   
 
Our  techniques also give the leafwise homotopy invariance  of the twisted higher Betti classes. When the twisting bundle $E$ is trivial, this extends (in the Riemannian case)  the main theorem of \cite{H-L:1991}.

\medskip\noindent

{\bf Theorem \ref{betti}} 
{\em Suppose that  $M$ is a compact Riemannian manifold,  with oriented Riemannian foliation $F$ of dimension $p$.  Let $E$ be a leafwise flat complex bundle over $M$ with a (possibly indefinite) non-degenerate Hermitian metric which is preserved by the leafwise flat structure.  
Assume that the projection onto  $\Ker( \Delta^{E})$ for the associated foliation $F_s$ of the homotopy groupoid of $F$  is transversely smooth. Then the twisted higher 
Betti classes $\beta_j(F,E)$, $0 \leq j \leq p$, are  leafwise homotopy invariants. }

\medskip

We now give some background to place the results of this paper in context.

Let $M$ and  $M'$ be closed oriented manifolds with oriented foliations $F$ and $F'$.   
Let $\varphi: (M', F')\to (M, F)$ be an oriented, leafwise oriented, leafwise homotopy equivalence.
Denote the homotopy groupoid of $F$ by $\maG$, and let  $f: M \to B\maG$ be a classifying map  for $F$.   The BC Novikov conjecture predicts that for every 
$x\in H^*(B\maG; \R)$,
$$
\int_M \L (TF) \cup f^*x = \int_{M'} \L (TF') \cup (f\circ \varphi)^*x.
$$ 
It is easy to check that this conjecture reduces to the case where the leaves have even dimension. In the case of a foliation with a single closed leaf with fundamental
group $\Gamma$ and denoting 
by $f:M\to B\Gamma$ a classifying map for the universal cover of $M$, the BC Novikov conjecture 
reduces to the classical Novikov conjecture
$$
\int_M \L (TM) \cup f^*x = \int_{M'} \L (TM') \cup (f\circ \varphi)^*x, \quad \forall x\in H^*(B\Gamma; \R).
$$

A powerful approach to the Novikov conjecture was initiated by Kasparov in \cite{Kasparov}.   He actually proves a stronger version of the Novikov conjecture, 
namely the rational
injectivity of the famous Baum-Connes map \cite{KasparovSkandalis, HigsonKasparov, Lafforgue}. See  \cite{Tu99} for a proof of this injectivity for a 
large class of foliations, including  hyperbolic foliations. Note that it is still an open question whether the Baum-Connes map is rationally injective 
for Riemannian foliations. 

A second approach to the Novikov conjecture was initiated  by Connes and his collaborators \cite{ConnesMoscovici} and uses cyclic cohomology and the homotopy 
invariance of the  Miscenko symmetric signature in the $K$-theory of the reduced group $C^*$-algebra \cite{Kasparov, Mischschenko}.   
This method proved successful, \cite{CGM93}, for the largest known class of groups, including Gromov-hyperbolic groups.  For foliations, the homotopy invariance of 
the corresponding Miscenko class in the $K$-theory of the $C^*$-algebra of $\cG$  was explained in \cite{BaumConnes1, BaumConnes2} and proved independently  in 
\cite{KaminkerMiller} and \cite{HilsumSkandalis}.  It reduces  the BC Novikov conjecture to an extension problem in the $K$-theory of foliations, together with 
a cohomological longitudinal index formula. The extension problem was first solved by Connes for certain cocycles in  \cite{ConnesTransverse}, by using a highly non 
trivial analytic breakthrough.  For general cocycles, the extension problem is a serious obstacle and many efforts have been made in this direction
 \cite{Cuntz, CuntzQuillen, LMN, Nistor, Puschnigg, Meyer}.  See also the recent \cite{Carrillo} for an alternative approach. 

The present paper was inspired by a third method mainly due to  Lusztig \cite{Lusztig},  and to ideas of Gromov \cite{Gromov}. It relies on the fact that for
 discrete groups having {\em enough finite dimensional $U(p,q)$ representations}, the even cohomology of the  classifying space $B\Gamma$ is generated by 
$U(p,q)$ flat $K$-theory classes. The main theorem needed in this approach  is  the oriented homotopy invariance of the twisted signature by such $K$-theory classes.
 This approach has been extended in \cite{CGM, CGM93} to cover all the known cases, using the concept of groups having {\em enough almost representations} and almost 
flat $K$-theory classes. 

Recall that in non-commutative geometry, the index of an elliptic operator is usually defined as a certain $C^*$-algebra $K$ theory class constructed out of the operator 
itself, without reference to its kernel or cokernel.   In the special (commutative) sub-case of a fibration, the Chern character of  this operator $K$ theory class 
coincides with the Chern character of the index bundle determined by the operator.  
 In the (non-commutative) case of foliations, this equality is not known in general. 
See \cite{BHII}, where conditions are given for it to hold, as well  as \cite{Nistor} and the recent \cite{AGS}.  For the signature operator, and its twists by leafwise almost flat 
$K$-theory classes, the $C^*$-algebra $K$-theory index is well known to be a leafwise homotopy invariant of the foliation \cite{HilsumSkandalis}. However, 
in order to deduce explicit results on the BC Novikov conjecture for foliations, one needs to define a Chern-Connes character of this $C^*$-algebra $K$-theory class 
and to compute it.  Our approach to this problem is to use the index bundle of the twisted leafwise signature operator, whose Chern-Connes character in Haefliger 
cohomology is well defined as soon as the bundle is. It is therefore a natural problem to prove directly the homotopy invariance of the Chern-Connes character of the 
leafwise signature index bundle and its twists by leafwise (almost) 
flat $K$-theory classes.   

Our program to attack the BC Novikov conjecture for foliations consists of three steps. 
\begin{itemize} 
\item Given a $K$-theory class $y=[E^+] - [E^-]$ over $B\cG$, prove that the characteristic number  
$\dd \int_M \L (TF) \cup f^*\ch(y)$ equals the higher leafwise harmonic signature  twisted by $f^*y$.
\item Prove that the  higher leafwise harmonic signature  twisted by leafwise almost flat $K$-theory classes of the ambiant manifold is a  leafwise oriented, leafwise homotopy invariant. 
\item Prove that  complex bundles $E = E^+ \oplus E^-$, such that $[f^*E^+] - [f^*E^-] $ is a leafwise almost flat $K$-theory class, generate the $K$-theory of $B\cG$.
\end{itemize}
 It is clear that solving these three problems for a class of foliations implies the BC Novikov  conjecture for that class.  
The first step was partially completed in our previous papers \cite{BHI, BHII}, where we proved this equality under certain assumptions, which were subsequently removed in 
 \cite{AGS}, provided the bundle $E^+ \oplus E^-$ is {\em globally} flat.  We conjecture that the result is still true under the far less restrictive assumption that $E^+ \oplus E^-$ is only  {\em leafwise} flat. 
The second step is the goal of the present paper, when the coefficient bundle $E$ has a leafwise  flat structure and  the foliation is Riemannian.    
See  \cite{BHIV} for further results on this question. 

Our results so far on the third step rely on deep but now classical results of Gromov \cite{Gromov}, and allow us, (assuming our conjecture above),  to prove, for instance, the BC Novikov conjecture, 
without extra assumptions, for the subring of $H^*(B\cG;\R)$ generated by $H^1(B\cG;\R)$ and $H^2(B\cG;\R)$.  Again see  the forthcoming paper \cite{BHIV}.

Finally, we conjecture that the Riemannian assumption can be removed, and that 
the only serious obstacle now lies in the third step.

\tableofcontents

\vspace{-0.75cm}

We now briefly describe the contents of each section. Section 2 contains notation and some review. In Section 3, we construct the Chern-Connes character for transversely smooth 
idempotents, which takes values in the Haefliger cohomology of the foliation.   
In Section 4, we define the twisted higher harmonic signature, and prove that if the parallel translation using the flat structure on $E$ is bounded, then the projection to the twisted harmonic forms is transversely smooth.   
Section 5 contains two important concepts essential to the proof of our main theorem, namely the notion of a ``smooth bundle over the space of leaves of $F$", and the extension to such bundles of the classical Chern-Weil theory of characteristic classes.  This allows us to compare the characteristic classes of such bundles on different manifolds. 
Section 6 is concerned with the study of leafwise homotopy equivalences, and their induced maps on Haefliger cohomology and on leafwise Sobolev cohomologies. In general, leafwise homotopy equivalences do not behave well on Sobolev forms and cohomologies. To overcome these difficulties, we use two different constructions. The first, due to Hilsum-Skandalis \cite{HilsumSkandalis}, gives smooth bounded maps between Sobolev forms. The second, which uses 
the Whitney isomorphism between simplicial and smooth cohomology, gives us control of the leafwise cohomologies.  
In Section 7, we prove that the pull-backs under leafwise homotopy equivalences of certain smooth 
bundles over the space of leaves are still smooth bundles. 
Section 8 extends the notion of pulled-back connections. 
Section 9 contains the proof of the main theorem. 
In Section 10, we prove the equality between the twisted higher harmonic signature and the Chern-Connes character of the index bundle of the twisted leafwise signature operator.  We  explain how our methods extend to prove Theorem \ref{betti}.   We also conjecture a cohomological formula for  the  twisted higher harmonic signature,  which is already know to be true in some cases.  See \cite{Heitsch:1995, H-L:1999,BHII} and the forthcoming \cite{AGS}. 
Finally, in Section 11 we show how our results lead to important consequences for the Novikov conjecture for foliations and for groups. 

\medskip

{\em Acknowledgments.} We are indebted to J. Alvarez{-Lopez},  A. Connes, J. Cuntz, Y. Kordyukov, J. Renault, J. Roe, G.  Skandalis, D. Sullivan, and K. Whyte for 
many useful discussions. 
Part of this work was done while the first author was visiting the University of Illinois at Chicago, the second author was visiting the University of Metz, and 
both authors were visiting the Institut Henri Poincar\'{e} in Paris, and the Mathematisches Forschungsinstitut Oberwolfach.  Both authors are most grateful for the 
warm hospitality and generous support of their hosts.

\section{Notation and review}\label{note}

Throughout this paper $M$ denotes a smooth compact Riemannian manifold of
dimension $n$, and $F$ denotes an oriented {Riemannian } foliation of $M$ of
dimension $p = 2{\ell}$ and codimension $q$. So $n=p+q$.  The tangent bundle of $F$
is denoted by $TF$, its normal bundle by $\nu$, and its dual normal bundle by 
$\nu^*$.   {We assume that the metric on $M$, when restricted to $\nu$, is bundle like, so the holonomy maps of $\nu$ and $\nu^*$ are isometries.} A leaf of $F$ is denoted $L$.
We denote by $\cU$ a finite good cover of $M$ by foliation charts as 
defined in \cite{H-L:1990}.

If $V \to N$ is a vector bundle over a manifold $N$, we denote the
space of smooth sections by $C^{\infty}(V)$ or by $C^{\infty}(N;V)$ if
we want to emphasize the base space of the bundle.  The compactly
supported sections are denoted by $C^{\infty}_{c}(V)$ or
$C^{\infty}_{c}(N;V)$.  The space of differential $k$ forms on $N$ is
denoted $\cA^{k}(N)$, and we set $\cA^*(N) = \oplus_{k\geq
0}\cA^{k}(N)$.  The space of compactly supported  $k$ forms  is denoted
$\cA_c^{k}(N)$, and $\cA^*_c(N) = \oplus_{k\geq 0}\cA_c^{k}(N)$.   The de Rham exterior derivative is denoted $d$ or $d_N$.
The tangent and cotangent bundles of $N$ will be denoted $TN$ and $T^*N$.

The (reduced) Haefliger cohomology of $F$, \cite{Hae:1980}, \cite{BHII}, is given as
follows. For each $U_i \in {\cU}$, let $T_i\subset U_i$ be a
transversal and set $T=\bigcup\,T_i$.  We may assume that the closures
of the $T_i$ are disjoint.  Let $\cH$ be the holonomy pseudogroup
induced by $F$ on $T$.   Denote the exterior derivative by $d_T:\cA^k_c(T)\to 
\cA^{k+1}_c(T)$.   The usual Haefliger cohomology is defined using the 
quotient of $\cA^k_c(T)$ by the vector subspace $L^k$ generated by 
elements of the form $\alpha-h^*\alpha$ where $h\in \cH$ and 
$\alpha\in\cA^k_c(T)$ has support contained in the range of $h$.  The (reduced) Haefliger cohomology 
uses  the quotient of $\cA^k_c(T)$ by the closure $\overline{L^k}$ of  
$L^k$.   We take this closure in the following sense. (The reader should note that in previous papers, we said that we used the $C^{\infty}$ topology to take this closure, but in fact we used the one given here.)   $\overline{L^k}$ consists of all elements in $\omega \in \cA^k_c(T)$, so that there are sequences $\{ \omega_n \}, \{ \what{\omega}_n \} \subset L^k$ with $||\omega - \omega_n|| \to 0$ and $||d_T(\omega) -  \what{\omega}_n)|| \to 0$.   The norm $|| \, \cdot \, ||$ is the sup norm, that is $||\omega|| = \sup_{x \in T} ||\omega(x)||_x$,  where $||\, \cdot \,||_x$ is the norm on $(\wedge^k T^*T)_x$.  
Set  $\cA^k_c(M/F)= \cA^k_c(T)/ \overline{L^k}$. 
The exterior derivative $d_T$ induces a continuous differential
$d_H:\cA^k_c(M/F)\to \cA^{k+1}_c(M/F)$.  Note that $\cA^k_c(M/F)$ and
$d_H$ are independent of the choice of cover $\cU$.  In this paper, the
complex $\{\cA^*_c(M/F),d_H\}$ and its cohomology $\oH^*_c(M/F)$ will be 
called, respectively, the Haefliger forms and Haefliger cohomology of $F$.   The reader should note that this cohomology appears as a quotient in the general computation of cyclic homology for foliations carried out in \cite{BrylinskiNistor}.

As the bundle $TF$ is oriented, there is a continuous open surjective 
linear map, called integration over the leaves,
$$
\int_F :\cA^{p+k}(M)\longrightarrow \cA^k_c(M/F)
$$
which commutes with the exterior derivatives $d_{M}$ and $d_{H}$.  Given 
$\omega \in \cA^{p+k}(M)$,  write $\omega = \sum \omega_i$ where 
$\omega_i \in\cA^{p+k}_c(U_i)$.  Integrate $\omega_i$ along the fibers 
of the submersion $\pi_i:U_i\to T_i$ to obtain $\dd \int_{U_i} \, 
\omega_i  \in \cA^k_c(T_i)$.  Define $\dd \int_F \ \omega 
\in\cA^k_c(M/F) $ to be the class of $\dd \sum_i \, \int_{U_i} \, 
\omega_i$.  It is independent of the choice of the $\omega_i$ and of the 
cover $\cU$.  As $\dd \int_{F}$ commutes with $d_{M}$ and $d_{H}$, it 
induces the map $\dd \int_F :\oH^{p+k}(M;\R) \to \oH^k_c(M/F)$.

For convenience we will be working on the homotopy  groupoids (also called the monodromy groupoids) of our foliations, but our results extend to the holonomy groupoid, as well as any groupoids between these two extremes.  

Recall that  the homotopy groupoid ${\mathcal G}$ of $F$ consists of equivalence classes of paths $\gamma:[0,1]\to M$ such that the image of $\gamma$ is contained in a leaf of $F$.  Two such paths $\gamma_{1}$ and $\gamma_{2}$ are equivalent if they are in the same leaf and homotopy equivalent (with endpoints fixed) in that leaf.  Two classes may be composed if one ends where the second begins and the composition is just the juxtaposition of the two paths.  This makes $\cG$ a groupoid.  The space $\cG^{(0)}$ of units of $\cG$ consists of the equivalence classes of the constant paths, and we identify $\cG^{(0)}$ with $M$.

For Riemannian foliations, $\cG$ is a Hausdorff dimension $2p+q$ manifold, in fact a fibration.  The basic open sets defining its manifold structure are given as follows.  Given $U, V \in \cU$ and a leafwise path $\gamma$ starting in $U$ and ending in $V$, define $(U, \gamma, V)$ to
be the set of equivalence classes of leafwise paths starting in $U$
and ending in $V$ which are homotopic to $\gamma$ through a homotopy
of leafwise paths whose end points remain in $U$ and $V$
respectively.
It is easy to see, using the holonomy defined by $\gamma$ from a
transversal in $U$ to a transversal in $V$, that if $U, V \simeq
\R^{p} \times \R^{q}$, then $(U,\gamma,V) \simeq \R^{p} \times
\R^{p} \times\R^{q}$.

The source and range maps of the groupoid $\cG$ are the two natural maps 
$s$, $r:\cG \to M$ given by $s\bigl([\gamma]\bigr)=\gamma(0)$, 
$r\bigl([\gamma]\bigr)=\gamma(1)$.  $\cG$ has two natural transverse 
foliations $F_s$ and $F_r$ whose leaves are respectively 
$\wL_x=s^{-1}(x)$, and $\wL^x=r^{-1}(x)$, for each $x\in M$.  Note that 
$r:\wL_{x} 
\to L$ is the simply connected covering of $L$.  We will work with the foliation 
$F_s$.  Note that the intersection of  any leaf $\wL_x$ and any basic open 
set $(U, \gamma, V)$ consists of at most one placque of the foliation 
$F_s$ in $(U, \gamma, V)$, i.e.\ each $\wL_x$ passes through any $(U, 
\gamma, V)$  at most once.

There is a canonical lift of the normal bundle $\nu$ of $F$ to a bundle 
$\nu_{\cG}\subset T\cG$ so that $T\cG=TF_s\oplus TF_r\oplus\nu_{\cG}$,  and 
$r_*\nu_{\cG}=\nu$ and $s_*\nu_{\cG}=\nu$.  It is given as follows.  Let $[\gamma]\in \cG$ with 
$s\bigl([\gamma]\bigr)=x$, $r\bigl([\gamma]\bigr)=y$.  Denote by 
$\exp:\nu\to M$ the exponential map.  Given $X\in \nu_x$ and $t\in\R$
sufficiently small, there is a unique leafwise path $\gamma_t:[0,1]\to M$ 
so that
\begin{center}
i) $\gamma_t(0)=\exp (tX)$ \hspace{.3in}
ii) $\gamma_t(s)\in \exp (\nu_{\gamma(s)})$.\\
\end{center}
In particular $\gamma_0=\gamma$.  Thus the family $[\gamma_t]$ in
$\cG$ defines a tangent vector $\what X\in T\cG_{[\gamma]}$.  It is
easy to check that $s_*(\what X)=X$ and $r_*(\what X)$ is the
parallel translate of $X$ along $\gamma$ to $\nu_y$.  

The metric $g_0$ on $M$ induces a canonical metric $g_0$ on $\cG$ as
follows.  $T\cG=TF_s\oplus TF_r \oplus \nu_{\cG}$ and these bundles are 
mutually orthogonal.  So the normal bundle $\nu_s$ of $TF_s$ is $\nu_s = TF_r \oplus \nu_{\cG}$.  On $TF_r $,  $g_0$ is $s^*\bigl(g_0|TF \bigr)$, on $TF_s$ it is $r^*\bigl(g_0|TF 
\bigr)$, and on  $\nu_{\cG}$ it is $r^*\bigl(g_0|\nu 
\bigr)$, which, since $F$ is Riemannian and the metric on $\nu$ is bundle-like, is the same as $s^*\bigl(g_0|\nu \bigr)$.

We denote by $E$  a leafwise flat complex bundle over $M$.  This means that there is a connection $\nabla_E$ on $E$ over $M$ which, when restricted to any leaf $L$ of $F$, is a flat connection, i.e.\ its curvature $(\nabla_E)^2 \,| \, L = (\nabla_E \,| \, L)^2 =0$.   This is equivalent to the condition that the parallel translation defined by $\nabla_E \,| \, L$, when restricted to contractible loops in $L$, is the identity.  We assume that $E$   admits a (possibly indefinite) non-degererate Hermitian metric, denoted $\{ \cdot, \cdot \}$, which is preserved by the leafwise flat structure.   This means that if $\phi_1$ and $\phi_2$ are local leafwise flat sections of $E$, then their inner product $\{\phi_1,\phi_2\}$ is a locally constant function on each leaf.  More generally, it is characterized by the fact that for general sections 
$\phi_1$ and $\phi_2$, and for any vector field $X$ tangent to $F$, 
 $$
X  \{\phi_1,\phi_2\} \,\, = \,\,  \{\nabla_{E,X}\phi_1,\phi_2\} + \{\phi_1, \nabla_{E,X}\phi_2\}.
$$
We denote also by $E$ its pull back by $r$ to a leafwise (for the foliation $F_s$) flat bundle on $\cG$ along with its invariant metric and leafwise flat connection.  The context should make it clear which bundle we are using.    A splitting of $E$ is a decomposition $E = E^+ \oplus E^-$ (of $E$ on $M$) into an orthogonal sum of two sub-bundles so that the metric is $\pm$ definite on $E^{\pm}$.  Splittings always exist and any two are homotopic.  The splitting   defines an involution $\gamma$ of $E$.  If $\phi$ is a local section of $E$ with $\phi = \phi^+ + \phi^-$ where $\phi^{\pm}$ is a local section of $E^{\pm}$, then $\gamma\phi = \phi^+ - \phi^-$.  If we change the sign of the metric on $E^-$, we obtain a positive definite Hermitian metric  on $E^-$ and so also on $E$ over both $M$ and $\cG$.  In general,  this new metric on $E$, denoted $( \cdot, \cdot)$,  is not preserved by the flat structure.

\begin{example}\label{nateg}  Assume that the codimension of $F$ is even, say $q = 2k$.  Set $E = \wedge^{k} \nu^* \otimes \C$.  The bundles $\nu$ and $\nu^*$ have natural flat structures along the leaves given by the holonomy maps (which define flat local sections). Since the metric on $\nu$  is bundle-like, the induced volume form on $\nu^*$ is invariant under the holonomy of $F$.  Denote by $*_{\nu}$ the Hodge $*$ operator on $\wedge^* \nu^*$, and also its extension to $E$.  Given two elements $\phi_1$ and $\phi_2$ of $E_x$, set 
$$
\{\phi_1, \phi_2\} \,\, = \,\, \sqrt{-1}^{k^2}*_{\nu}(\phi_1 \wedge_{\nu} \overline{\phi_2}),
$$
where $\wedge_{\nu}: E \otimes E \to \wedge^{2k} \nu^*\otimes \C$.   We leave it to the reader to check that $E$ and $\{\cdot, \cdot \}$ satisfy the hypothesis of Theorem \ref{main}.
\end{example}

Denote by $\cA^*_c(F_{s},E)$ the graded algebra of leafwise (for $F_s$) differential forms on $\cG$ with coefficients in $E$ which have compact support when restricted to any leaf of $F_{s}$.   A Riemannian structure on $F$ induces one on $F_{s}$. As usual there is the leafwise Riemannian Hodge operator $*$, which gives 
an inner product on each $\cA^k_c(F_{s},E)$.   In particular, if $\alpha_1$ and $\alpha_2$ are leafwise $\R$ valued $k$ forms and $\phi_1$ and $\phi_2$ are sections of $E$, then 
$$ 
<\alpha_1 \otimes \phi_1,\alpha_2 \otimes \phi_2>(x) = \int_{\wL_x} (\phi_1, \phi_2) \alpha_1 \wedge *\alpha_2 = \int_{\wL_x} \{\phi_1, \gamma \phi_2\} \alpha_1 \wedge *\alpha_2.
$$
We denote by $\cA_{(2)}^*(F_{s},E)$ the field of Hilbert spaces over $M$ which is the leafwise $L^2$ completion of these differential forms under this inner product, i.e.
$$
\cA_{(2)}^* (F_{s},E)_x = L^2 (\wL_x; \wedge T^* F_s \otimes E).
$$
This is a {continuous} field of Hilbert spaces, {see} \cite{Connes:1979}.    Because $M$ is compact, the spaces $L^2 (\wL_x; \wedge^k T^* F_s  \otimes E)$ do not depend 
on our choice of metrics.   However,  the inner products on 
these spaces do depend on the metrics, as do the Hilbert norms, denoted $\| \cdot \|_0$. 

If $E$ is the one dimensional trivial bundle with the trivial flat structure, then $\cA_{(2)}^*(F_{s},E)$ is just the leafwise $L^2$ forms (now with coefficients in $\C$) for the foliation $F_s$ and is denoted $\cA_{(2)}^*(F_{s}, \C)$.

\section{Chern-Connes character for transversely smooth idempotents}\label{cc}

Since we need the ``transverse differential" and graded trace used in \cite{BHI} to define the Chern-Connes character, we now briefly recall that construction.

Consider the algebra  $C_c^{\infty}(\cG; \wedge F_s \otimes E)$ of smooth 
compactly supported sections over $\cG$ of the bundle, here denoted 
$\wedge F_s \otimes E$, whose fiber at $\gamma\in \cG$ is
$$
(\wedge F_s \otimes E)_\gamma = \Hom ((\wedge T^*F \otimes E)_{s(\gamma)}, (\wedge 
T^*F\otimes E)_{r(\gamma)}).
$$
If $\alpha \in C_c^{\infty}(\cG; \wedge F_s \otimes E )$, it defines the leafwise 
operator $A$ which acts on $\phi \in \cA_{(2)}^* (F_{s},E)_x$ by
$$
(A\phi)(\gamma) = \int_{\wL_x} \alpha(\gamma \gamma_1^{-1})\phi(\gamma_1) 
d\, \gamma_1,
$$ 
where $\gamma, \gamma_1 \in \wL_x$, {and we identify $(T^*F_s)_{\gamma}$ with $T^*F_{r(\gamma)}$.}
In  \cite{BHI}, we defined   a Chern-Connes character from the $K-$theory 
of this algebra to the Haefliger cohomology of the foliation, 
$$
\ch_a : K_0(C_c^{\infty}(\cG; \wedge F_s \otimes E)) \longrightarrow \oH_c^*(M/F),
$$
given as follows.
Consider the connection $\nabla$ on $\wedge T^*F_s \otimes E$ given by $\nabla = r^* 
(\nabla_F \otimes \nabla_E)$ where $\nabla_F$ is a connection on $\wedge T^*F$ defined by a 
connection on $T^*F$.    Then $\nabla:C^{\infty}(\wedge T^*F_s \otimes E) \to 
C^{\infty}( {T^{*}\cG} \otimes  \wedge T^*F_s \otimes E)$, and we may extend 
$\nabla$ to an operator of degree one on $C^{\infty}(\wedge{T^{*}\cG} 
\otimes \wedge T^*F_s \otimes E)$, where on decomposable sections $\omega
\otimes \phi $, with $\omega \in C^{\infty}(\wedge^k T^{*}\cG)$, 
$\nabla(\omega\otimes \phi) = d \omega \otimes \phi + (-1)^k \omega \wedge 
\nabla \phi .$   {The foliation $F_s$ has dual normal bundle $\nu_s^* = s^{*}(T^{*}M)$}, and $\nabla$ defines a {\em
quasi-connection} $\nabla^{\nu}$ acting on 
$C^{\infty}(\wedge \nu_s^*\otimes \wedge T^*F_s \otimes E)$ by the composition 
$$
C^{\infty}(\wedge \nu_s^*\otimes \wedge T^*F_s \otimes E) \stackrel{i}{\longrightarrow} 
C^{\infty}(\wedge{T^{*}\cG} \otimes \wedge T^*F_s \otimes E)\stackrel{\nabla}{\longrightarrow} 
C^{\infty}(\wedge{T^{*}\cG} \otimes \wedge T^*F_s \otimes E) \stackrel{p_{\nu}}{\longrightarrow}
C^{\infty}(\wedge \nu_s^*\otimes \wedge T^*F_s \otimes E),
$$
where $i$ is the inclusion and $p_{\nu}$ is induced by the
projection $p_{\nu}:T^{*}{\cG} \to \nu^{*}_{s}$ {determined  by the 
decomposition $T\cG = TF_s  \oplus 
\nu_s$. }  

Denote by $\pa_{\nu}: \End(C^{\infty}(\wedge \nu_s^*\otimes \wedge T^*F_s \otimes E)) \to
\End(C^{\infty}(\wedge \nu_s^*\otimes \wedge T^*F_s \otimes E))$ the linear operator 
given by the graded commutator
$$
\pa_{\nu} (T) = [\nabla^{\nu}, T].
$$
Recall: that 
$(\pa_{\nu})^2$ is given by the commutator with the curvature $\theta^{\nu} = 
(\nabla^{\nu})^2$ of  $\nabla^{\nu}$;  that $\theta^{\nu}$ is a leafwise 
differential operator which is at worst order one;  and that the 
derivatives of all orders of its coefficients are uniformly bounded, with 
the bound possibly depending on the order of the derivative.  See 
\cite{BHII}.

We may consider the algebra 
$$
\cA^*(M) \widehat{\otimes} _{C^{\infty}(M)}C_c^{\infty}(\cG;\wedge F_s \otimes E)
$$
as a subspace of the space of $ \cA^*(M) $-equivariant endomorphisms of
$C^{\infty}(\wedge \nu_s^*\otimes \wedge T^*F_s \otimes E)$ by using the $ 
\cA^*(M) $ module structure of $C^{\infty}(\wedge \nu_s^*\otimes \wedge T^*F_s \otimes E)$,  where for
$\phi \in C^{\infty}(\wedge \nu_s^*\otimes \wedge T^*F_s \otimes E)$, and $\omega \in
 \cA^*(M) $, we set
$$
\omega \cdot\phi = s^{*}(\omega) \wedge\phi.
$$
The operator $\pa_{\nu}$ maps  $ \cA^*(M)  \widehat{\otimes} 
_{C^{\infty}(M)}C_c^{\infty}(\cG;\wedge F_s \otimes E)$  to itself. 

Denote by $C^{\infty}(\cG;\wedge F_s \otimes E)$ the space of \underline{all} smooth 
sections over $\cG$ of $\wedge F_s \otimes E$. For  $T$ an element of $ \cA^*(M) 
\widehat{\otimes}_{C^{\infty}(M)}C^{\infty}(\cG;\wedge F_s \otimes E)$, define the trace of $T$
to be the Haefliger form $\Tr (T)$ given by
$$
\Tr(T) = \int_{F} \tr (T(\overline{x})) dx =  \int_{F} i^* (\tr (T \, | \, 
i(M))) dx,
$$
where $\overline{x}$ is the class of the constant path at $x$, $\tr(T(\overline{x}))$ is the $ \cA^*(M) $-equivariant trace of the Schwartz kernel of $T$ at $\overline{x}$, and so belongs to $\wedge T^* M_x$, and $dx$ is the leafwise volume form associated with the fixed orientation of the foliation $F$.  When restricted to the subspace 
$\cA^*(M) \widehat{\otimes} _{C^{\infty}(M)}C_c^{\infty}(\cG;\wedge F_s \otimes E)$, the map 
$$
\Tr:\cA^*(M) \widehat{\otimes} _{C^{\infty}(M)}C_c^{\infty}(\cG;\wedge F_s \otimes E) \longrightarrow \cA^*_{c}(M/F)
$$ 
is a graded trace which satisfies $\Tr \circ \pa_{\nu} = d_{H} \circ \Tr$, see  \cite{BHI}, and Lemma 6.3 of \cite{BHII}. Moreover, the equality $(\Tr \circ \pa_{\nu}) (T)= (d_{H} \circ \Tr) (T)$ extends to all transversely smooth operators $T$. See Definition \ref{TS} {below} and \cite{BHII}.

Since $\pa_{\nu} ^2$ is not necessarily zero, we used Connes' $X-$trick to 
construct a new graded differential algebra $({\wcB}, \delta)$ out of the 
graded quasi-differential  algebra $\cB=(\cA^*(M) \widehat{\otimes} 
_{C^{\infty}(M)}C_c^{\infty}(\cG;\wedge F_s \otimes E), \pa_{\nu})$.  See 
\cite{ConnesBook}, p.\ 229 for the definition of the grading, the 
extension of $\pa_{\nu}$ to the differential $\delta$, and the product 
structure on $\wcB$.  As a vector space, $\wcB$ is $M_2(\cB)$, the space 
of $2 \times 2$ matrices with coefficients in $\cB$, and $\cB$ embeds as a 
subalgebra of ${\wcB}$ by using the map
$$
T \hookrightarrow \left( \begin{array}{cc} T & 0 \\ 0 & 0
\end{array} \right).
$$
For homogeneous ${\widetilde T}\in{\wcB}$ of degree $k$, Connes defines
$$
\Phi ({\widetilde T}) = \Tr ({T}_{11}) - (-1)^k  \Tr({T}_{22} \theta^{\nu}),
$$
and extends to arbitrary elements by linearity.
The map $\Phi:{\wcB} \to \cA^*_c(M/F)$ is then a  graded trace,
and again we have $\Phi \circ \delta = d_H \circ \Phi$,  see  \cite{BHI}.

The (algebraic) Chern-Connes character in the even case is then the 
morphism 
$$
\ch_a:\oK_{0}(C_c^{\infty}(\cG;\wedge F_s \otimes E))  \longrightarrow 
\oH^*_{c}(M/F) 
$$
defined as follows.   Let $B = [\te_{1}] - [\te_{2}]$ be an element of
$\oK_{0}(C_c^{\infty}(\cG;\wedge F_s \otimes E))$, where $\te_{1}=(e_1,\lambda_1)$
and $\te_{2}=(e_2,\lambda_2)$.  The $\lambda_i$ are $N \times N$ matrices 
of complex numbers, and the $e_i$ are in 
$M_{N}(C_c^{\infty}(\cG;\wedge F_s \otimes E))$, the $N \times N$ matrices over 
$C_c^{\infty}(\cG;\wedge F_s \otimes E)$, which we may consider as elements of 
$M_{N}(\wcB)$.   Denote by $\tr:M_{N}(\wcB) \to \wcB$
the usual trace.
Then the Haefliger {form
$$
(\Phi\circ \tr)\Bigl{(}e_{1}\exp\left({\frac{-(\delta e_{1})^{2}}{2i 
\pi}}\right)\Bigr{)} - 
(\Phi\circ \tr) \Bigl{(}e_{2}\exp \left({\frac{-(\delta e_{2})^{2}}{2 i 
\pi}}\right)\Bigr{)}
$$
is closed and its Haefliger cohomology class} depends 
only on $B$, \cite{BHI}. This Haefliger cohomology class is precisely the 
Chern-Connes character of $B$.
So,
\begin{Equation}\label{ChernDef}
\hspace{1in} $ \dd  \ch_{a}(B) = \left[(\Phi\circ
\tr)\Bigl{(}e_{1}\exp\left({\frac{-(\delta
  e_{1})^{2}}{2i \pi}}\right)\Bigr{)}- (\Phi\circ \tr) 
\Bigl{(}e_{2}
\exp\left({\frac{-(\delta
  e_{2})^{2}}{2i\pi}}\right)\Bigr{)}\right].$
\end{Equation}

We want to consider the Chern-Connes characters of idempotents, such as the 
projection onto the twisted leafwise harmonic forms, which in general do not 
define elements of $\oK_0(C^{\infty}_c(\cG; \wedge F_s \otimes E))$.   The 
idempotents we are interested in are bounded leafwise smoothing
operators on $\wedge T^*F_s \otimes E$.   In order to define the Chern-Connes character of such idempotents, we need the concept of ``transverse smoothness" for $\cA^*(M)$ 
equivariant  bounded leafwise smoothing operators on $\wedge \nu^*_s 
\otimes \wedge T^* F_s \otimes E$.  If $H$ is such an operator, we can write it as 
$$
H = H_{[0]} + H_{[1]} + \cdots + H_{[n]},
$$
where $H_{[k]}$ is homogeneous of degree $k$, that is, for all $j$,
$$
H_{[k]}: C^\infty(\wedge^j \nu^*_s \otimes \wedge T^* F_s \otimes E) \to  
C^\infty(\wedge^{j+k} \nu^*_s \otimes \wedge T^* F_s \otimes E).
$$
Recall that for any $[\gamma] \in \cG$, $s_*:\nu_{s, 
[\gamma]}  \to TM_{s(\gamma)}$ is an isomorphism.   Thus any $X \in C^{\infty}(\wedge^k TM)$ 
defines a section, denoted $\what{X}$, of $\wedge^k \nu_s$. 
For such $X$, $i_{\what{X}} H_{[k]}$ is a  bounded {leafwise} smoothing operator  on $\wedge T^* F_s\otimes E$.  
For any vector field $Y$ on $M$, set
$$
\pa_{\nu}^Y  (i_X  H_{[k]}) = i_{\what{Y}} (\pa_{\nu}(i_{\what{X}} H_{[k]})),
$$
which (if it exists) is an operator  on  $\wedge T^* F_s \otimes E$.   

\begin{definition}\label{TS}
An $\cA^*(M)$ equivariant  bounded leafwise smoothing operator $H$ on 
$\wedge \nu^*_s \otimes \wedge T^* F_s\otimes E$ is transversely smooth provided 
that for
any $X \in C^{\infty}(\wedge^k TM)$, and any vector fields 
$Y_1, ... , Y_{m}$ on $M$,  the operator 
$$
\pa_{\nu}^{Y_1} ... \pa_{\nu}^{Y_{m}}(i_X H_{[k]})
$$
is a bounded leafwise smoothing operator on $ \wedge 
T^* F_s\otimes E$.
\end{definition}

\noindent
Any element of $C_c^{\infty}(\cG;\wedge F_s\otimes E)$ is transversely smooth.  
If the leafwise parallel translation along $E$ is a bounded map, then the projection onto the leafwise harmonic forms with coefficients in $E$ (for the foliation $F_s$)  is transversely smooth.   See Theorem \ref{GRthm} below. Since $\pa_{\nu}$ is a derivation, it is immediate that the composition of transversely smooth operators is transversely smooth. It is also easy to prove that the Schwartz kernel of any transversely smooth operator is a smooth section in all variables, see \cite{BHII}. 

If $K$ {is} a bounded leafwise smoothing operator on $ \wedge T^* F_s \otimes E$, we may extend it to an $\cA^*(M)$ equivariant bounded leafwise smoothing operator on $\wedge \nu^*_s \otimes \wedge T^* F_s \otimes E$ by  using the $\cA^*(M)$ module structure of $C^{\infty}(\wedge \nu^*_s \otimes \wedge T^* F_s \otimes E)$.
More specifically,   given $\phi \in C^{\infty}(\wedge \nu^*_s \otimes \wedge T^* F_s \otimes E)$,  write it as 
$$
\phi = \sum_j  s^*(\omega_j)  \otimes \phi_j ,
$$
where the $\omega_j \in \cA^*(M) $, and the  $\phi_j \in C^{\infty}( \wedge T^* F_s \otimes E)$.
Then 
$$
K(\phi) = \sum_j s^*( \omega_j)\otimes K(\phi_j) .
$$
It is easy to check that this is well defined.

The proof of Lemma 4.5 of \cite{BHII} extends easily to give the following.
\begin{lemma}\label{AK}  Suppose that $A$ is an {$\cA^*(M)$-equivariant}  leafwise differential operator of {finite} order on $\wedge \nu^*_s \otimes \wedge T^* F_s \otimes E$, and that  the derivatives of all orders of its coefficients are uniformly bounded, with the bound possibly depending on the order of the derivative.   Suppose that  $K$ is a bounded leafwise smoothing operator on $\wedge T^* F_s \otimes E$, and extend it to an $\cA^*(M)$-equivariant  bounded leafwise smoothing operator on $\wedge \nu^*_s \otimes \wedge T^* F_s \otimes E$.  Then  $AK$ and $KA$ are  $\cA^*(M)$-equivariant  bounded leafwise smoothing operators on $\wedge \nu^*_s \otimes \wedge T^* F_s \otimes E$.  If $K$ is transversely smooth, so are $AK$ and $KA$.
\end{lemma}

For the convenience of the reader, we recall what the conditions on $A$ mean.  Write 
$A = \sum_0^nA_{[k]}$, where $A_{[k]}$ is homogeneous of degree $k$.
Let  $X \in C^{\infty}(\wedge^k TM)$ be a local section of norm one, and consider $i_{\what{X}}A_{[k]}$
which is a differential operator (say of order $d$) on $\wedge \nu^*_s \otimes  \wedge T^* F_s \otimes E$.  
Let $(U, \gamma, V)$ be a basic open set for $\cG$ where $U, V \in \cU$, the fixed good cover.  Then $U$ and $V$ come with fixed
coordinates $x_1,...,x_p, w_1,...,w_q$ and  $y_1,...,y_p, z_1, ... ,z_q$.  
The $x_i$ and $y_i$ are the leaf coordinates for $F$, and 
the $w_i$ and $z_i$ are the normal coordinates.    The   coordinates for $(U, \gamma, V)$ are then $x_1,...,x_p, y_1,...,y_p, z_1, ... ,z_q$, and the $y_1,...,y_p$ are the leaf coordinates for $F_s$.   
With respect to an orthonormal basis of $\wedge \nu^*_s \otimes  \wedge T^* F_s \otimes E$ on $(U, \gamma, V)$, 
$i_{\what{X}}A_{[k]}$ may be written as a matrix of operators of the form
$$
\sum_{|\alpha|=0}^{d}  a_{\alpha}(x,y,z) \frac{ \pa^{|\alpha|}}{\pa y_1^{\alpha_1} ... \pa y_p^{\alpha_p}},
$$
where the $a_{\alpha}$ are locally defined smooth functions.  Then each derivative of the $a_{\alpha}$ with respect to the variables $x,y,$ and $z$ is assumed to be globally bounded over all basic open sets for $\cG$, and the bound may depend on how many derivatives are taken.

Note that operators $A$ which are the pull backs of operators on $M$, such as $\theta^{\nu}$ and $\tau$, satisfy the hypothesis of Lemma \ref{AK}.  Using Lemma \ref{AK}, it is easy to show that being transversely smooth is independent of the choice of $\nabla^{\nu}$. 

Finally, we need the concept of $\cG$ invariant $\cA^*(M)$-equivariant operators.  
Suppose that $H = H_{[0]} + H_{[1]} + \cdots + H_{[n]}$
is an $\cA^*(M)$-equivariant  bounded leafwise smoothing operator acting on the sections of $\wedge \nu_s^*\otimes \wedge T^*F_s \otimes E$.
Then $H$ is $\cG$ invariant  provided it satisfies two requirements.

$(1)$ For any $X = X_1 \wedge \cdots \wedge X_k \in C^{\infty}(\wedge^k TM)$ with some $X_j \in 
C^{\infty}(TF)$, $i_{\what{X}} H_{[k]} = 0.$

This means that $H_{[k]}$ defines an operator  
$
H_{[k]}: C^\infty(\wedge^j \nu^*_{\cG} \otimes \wedge T^* F_s \otimes E) \to  
C^\infty(\wedge^{j+k} \nu^*_{\cG} \otimes \wedge T^* F_s \otimes E),
$
and that for $k > q$,   $H_{[k]}=0$.

Each $\gamma \in \wL_x^y \equiv \wL_x \cap \wL^y$, defines an action   
$
W_{\gamma}: C^{\infty}(\wL_{x}, \wedge^*  \nu^*_{\cG} \otimes  \wedge T^* F_s \otimes E) \to C^{\infty}(\wL_{y},  \wedge^* \nu^*_{\cG} \otimes \wedge T^* F_s \otimes E),
$
given by 
$$
[W_{\gamma}\xi](\gamma') = \xi (\gamma'\gamma), \quad \gamma'\in
\wL_{y}.
$$
Let $y' =r(\gamma')$, and note  that $[W_{\gamma}\xi](\gamma') \in  (\wedge^*  \nu^*_{\cG} \otimes  \wedge T^* F_s \otimes E)_{\gamma'}$, which we identify with $\wedge^*  \nu^*_y \otimes  (\wedge T^* F \otimes E)_{y'}$, while $\xi (\gamma'\gamma) \in  (\wedge^*  \nu^*_{\cG} \otimes  \wedge T^* F_s \otimes E)_{\gamma'\gamma}$, which we identify with $\wedge^*  \nu^*_x \otimes  (\wedge T^* F \otimes E)_{y'}$.  To effect this action, we identify $\wedge^*  \nu^*_x $ with $\wedge^*  \nu^*_y $ using the holonomy along $\gamma$.  The second requirement  of $H$ is:

$(2)$  For any $\gamma \in \wL_x^y$, 
$$
(\gamma \cdot H)_y \equiv W_{\gamma} \circ H_x \circ W_{\gamma}^{-1} = H_y,
$$
where $H_x$ is the action of $H$ on $ \wedge^*  \nu^*_{\cG} \otimes\wedge T^* F_s \otimes E \, | \, \wL_{x}$.  

Essentially then, $H$ is $\cG$ invariant means that it defines the same operator on each $\wL \subset s^{-1}(L)$ for each leaf $L$ of $F$.  Note that $\pa_{\nu}$ preserves $\cG$ invariant 
$\cA^*(M)$-equivariant transversely smooth operators.

\medskip

In \cite{BHII}, we extended our Chern-Connes character to $\cG$ invariant transversely 
smooth idempotents.   The essential result  needed  was Lemma 4.13 of that paper, which we state for further reference.
\begin{lemma}\label{trace} Suppose that $H$ and $K$ are $\cG$ invariant 
$\cA^*(M)$-equivariant transversely smooth operators acting on the sections of $\wedge \nu_s^*\otimes \wedge T^*F_s \otimes E$.  
Then
$$
 (\Phi \circ \tr) ([H,K]) = 0.
$$
\end{lemma}
\noindent Lemma \ref{trace}  and the proof of Theorem 4.1 of \cite{BHI}
immediately imply
\begin{theorem}\label{family}
Let $e$ be a $\cG$ invariant transversely smooth idempotent acting on the sections of $\wedge T^*F_s \otimes E$.  Then  
$$
(\Phi \circ \tr) \Bigl{(}e\exp\left({\frac{-(\delta e)^{2}}{2i 
\pi}}\right)\Bigr{)}
$$ 
is a closed Haefliger form whose  Haefliger cohomology class,  denoted  
$\ch_a(e)$, depends only on $e$.  In addition, if $e_t$, $0 \leq t \leq 
1$,  is a smooth family of such idempotents, then 
$$
\ch_a(e_0) = \ch_a(e_1).
$$ 
\end{theorem}

The Haefliger class $\ch_a(e)$ is  the Chern-Connes character of $e$.

\begin{lemma}\label{samecc}
Two $\cG$ invariant transversely smooth idempotents which have the same 
image, have the same Chern-Connes character.  
\end{lemma}

\begin{proof}
Suppose that $e^0$ and $e^1$ are two such idempotents.  Then $e^0 \circ 
e^1 = e^1$ and $e^1 \circ e^0 = e^0$, and the family $e^t = te^1 + 
(1-t)e^0$ is a smooth family of $\cG$ invariant transversely smooth 
idempotents connecting $e^0$ to $e^1$.  
Theorem \ref{family} then gives the result.
\end{proof}

\section{The twisted higher harmonic signature}\label{tsig}

We now define the twisted higher harmonic signature $\sigma (F,E)$.  The leafwise de Rham differential on $\cG$ extends to a closed operator on $\cA_{(2)}^*(F_{s}, \C)$ which coincides with the lifted one from the 
foliation $(M,F)$ and it is denoted by $d_s$. The leafwise formal adjoint of $d_s$ with respect to the Hilbert structure is well defined and is denoted by $\delta_s$, and $\delta_s = -*d_s*$. Denote by $\Delta$ the Laplacian given by  $\Delta = (d_s + \delta_s)^2 = d_s\delta_s + \delta_s d_s$, and denote by $\Delta_{k}$ its action on $\cA_{(2)}^k(F_{s}, \C)$.   The leafwise $*$ operator also gives the leafwise involution $\tau$ on $\cA_{(2)}^*(F_{s}, \C)$, where as usual, 
$$
\tau = {\sqrt{-1}} ^{k(k-1) +\ell} \, *
$$
on $\cA_{(2)}^k(F_{s}, \C)$, and it is easy to check that $\delta_s  = - \tau d_s \tau$, so  $\tau(d_s +  \delta_s) =  -(d_s+ \delta_s)\tau$, and $\Delta\tau = \tau \Delta$.

These operators extend to $\cA_{(2)}^*(F_{s},E)$ as follows.    Since the operators are all leafwise, local and linear, we need only define them for local sections of the form $\alpha \otimes \phi$ where 
$\alpha$ is a local $k$ form on $\wL$, and $\phi$ is a local section of $E \,|\,\wL$.   
Then 
$$
d_s(\alpha \otimes \phi) =  d_s\alpha \otimes \phi + (-1)^k \alpha \wedge \nabla^{\wL}_E \phi,  \quad  
\what{*}(\alpha \otimes \phi) =  *\alpha \otimes \gamma\phi,  \quad  
\what{\tau}(\alpha \otimes \phi) =  \tau\alpha \otimes \gamma\phi,  
$$
where $\nabla_E^{\wL} $ is $ \nabla_E$ restricted to $\wL$, so $\nabla^{\wL}_E \phi$ is a local section of  $T^*\wL \otimes E$.  We define the wedge  product $\alpha \wedge \nabla^{\wL}_E \phi$ (as a local section of  $\wedge^{k+1}T^*\wL \otimes E$) in the obvious way.
\begin{lemma}\label{defQ} We have
$$
\delta_s =  -\what{*}\,\, d_s\,\,\what{*}  =  -\what{\tau}\,\, d_s\,\,\what{\tau} .   
$$
\end{lemma}
Note that $d_s^2 = 0$, so also $\delta_s^2 =0$ since $\what{*}\, ^2 = \pm 1$. 
\begin{proof}
Consider two sections $\alpha_1 \otimes \phi_1$ and $\alpha_2 \otimes \phi_2$, and set 
$$ 
Q( \alpha_1 \otimes \phi_1,\alpha_2 \otimes \phi_2 ) (x) \quad = \quad  \int_{\wL_x} \{\phi_1, \phi_2\} \alpha_1 \wedge \alpha_2,   
$$
(and extend to all of $\cA_{(2)}^*(F_{s},E)$ by linearity).  Then
$$ 
<\alpha_1 \otimes \phi_1,\alpha_2 \otimes \phi_2> \quad = \quad  Q(\alpha_1 \otimes \phi_1, \what{*}(\alpha_2 \otimes \phi_2)).
$$
Now suppose that $\alpha_1$ is a local $k-1$ form on $\wL$, $\alpha_2$ is a local $k$ form on $\wL$, and $\phi_1$ is flat.   If $\phi_2$ is an arbitrary section of $E$,  set 
$\{\phi_1, \alpha_2 \otimes  \phi_2\} = \alpha_2 \{\phi_1,  \phi_2\}$.  (Note that $\alpha_2$ is $\C$ valued).
As $\{ \cdot, \cdot \}$ is preserved by the flat structure and $\phi_1$ is flat, it follows that on $\wL$, 
$\nabla_E^{\wL} \{\phi_1,\phi_2\} =  \{\phi_1, \nabla_E^{\wL} \phi_2\}.$
Acting on functions on $\wL$, 
$d_s = \nabla_E^{\wL}$, so
$$
\quad d_s \{\phi_1,\phi_2\} =  \{\phi_1, \nabla_E^{\wL} \phi_2\}.
$$
Then 
$$
<d_s(\alpha_1 \otimes \phi_1),\alpha_2 \otimes \phi_2>  \quad = \quad 
 \int_{\wL_x} \{\phi_1, \gamma \phi_2\} d_s \alpha_1 \wedge *  \alpha_2  \quad = \quad 
(-1)^k \int_{\wL_x} \alpha_1 \wedge d_s( \{\phi_1, \gamma \phi_2\}  *  \alpha_2 ),
$$
while
$$
<\alpha_1 \otimes \phi_1,   -\what{*}\,\,  d_s \, \,\what{*}(\alpha_2 \otimes \phi_2)> \quad = \quad
$$
$$ 
(-1)Q(\alpha_1 \otimes \phi_1,   \what{*}^2d_s\,\, \what{*}(\alpha_2 \otimes \phi_2) ) \quad = \quad 
(-1)^k  Q(\alpha_1 \otimes \phi_1,   d_s\,\, \what{*}(\alpha_2 \otimes \phi_2) ) \quad = \quad 
$$
$$
(-1)^k Q(\alpha_1 \otimes \phi_1,   (d_s *\alpha_2) \otimes \gamma \phi_2  +
(-1)^k *\alpha_2 \wedge \nabla_E^{\wL} \gamma\phi_2 )\quad = \quad 
$$
$$
(-1)^k \int_{\wL_x} \alpha_1 \wedge (d_s *\alpha_2)  \{ \phi_1,  \gamma \phi_2  \}  +
(-1)^k \alpha_1 \wedge *\alpha_2 \wedge  \{\phi_1, \nabla_E^{\wL} \gamma\phi_2 \} \quad = \quad 
$$
$$
(-1)^k \int_{\wL_x} \alpha_1 \wedge d_s  (*  \alpha_2 \{\phi_1, \gamma \phi_2\})  \quad = \quad 
(-1)^k \int_{\wL_x} \alpha_1 \wedge d_s( \{\phi_1, \gamma \phi_2\}  *  \alpha_2 ).
$$
\end{proof}

Denote by $\Delta^E$ the Laplacian given by  $\Delta^E = (d_s + \delta_s )^2 =d_s \,\, \delta_s +  \delta_s \, d_s$, and denote by $\Delta^E_{k}$ its action on $\cA_{(2)}^k (F_{s}, E)$.  Note that  $\what{\tau}$ is still an involution even at the bundle level, and that  
$\what{\tau}(d_s +  \delta_s) =  -(d_s+ \delta_s) \what{\tau}$ and
$\Delta^E\what{\tau}= \what{\tau} \Delta^E$ still hold.

As usual, the space of twisted harmonic forms $\Ker(\Delta^E)$ is related to the leafwise cohomology of the twisted forms. 
The space of closed $L^2$ forms in $\cA_{(2)}^*(F_{s}, E)$ is denoted by $Z^*_{(2)}(F_{s}, E)$ and it is a Hilbert subspace.  The space of exact $L^2$ forms in $\cA_{(2)}^*(F_{s}, E)$ is denoted by  $B^*_{(2)}(F_{s}, E)$, and {we denote }its closure by ${\overline B}^*_{(2)}(F_{s}, E)${. We} denote by $\oH^*_{(2)}(F_{s}, E)$ the leafwise  reduced twisted $L^2$ cohomology  of the foliation, that is
$$ 
\oH^*_{(2)}(F_{s}, E) = Z^*_{(2)}(F_{s}, E) / {\overline B}^*_{(2)}(F_{s}, E).
$$
Here is a well known Hodge result that we state for further use. See the Appendix of \cite{H-L:1990}     
\begin{lemma}\label{hodge}
The field $\Ker (\Delta^E_k)$ is a subfield of 
$Z^k_{(2)}(F_{s}, E)$, and 
$Z^k_{(2)}(F_{s}, E) = \Ker (\Delta^E_k) \oplus {\overline 
B}^k_{(2)}(F_{s}, E)$.  Thus
the natural projection $Z^k_{(2)}(F_{s}, E) \to \oH^k_{(2)}(F_{s}, E)$ 
induces by restriction 
an isomorphism $$\Ker (\Delta^E_k)\simeq \oH^k_{(2)}(F_{s}, E).
$$
In addition
$$
\cA_{(2)}^* (F_{s},E)=  \Ker(d_s + (d_s)^*) \oplus \overline{\Im(d_s)} \oplus \overline{\Im(\delta_s)}.
$$
That is, for each $x \in M$,
$$
L^2(\wL_{x}; \wedge T^*\wL_{x} \otimes (E \, | \, \wL_{x})) = \Ker(d^{x}_s + \delta^x_s) \oplus \overline{\Im(d^{x}_s)} \oplus \overline{\Im(\delta^{x}_s)}.
$$
\end{lemma}

We assume that the projection $P_{\ell}$  onto $\Ker(\Delta^E_{\ell})$ is transversely smooth.  It is a classical result that this projection is a bounded leafwise smoothing operator, so what we are really assuming is a form of smoothness in transverse directions.  This condition holds in many important cases, see the comments below after the statement of Theorem \ref{GRthm}.
Denote by $\cA^*_{\pm}(F_s, E)$ the $\pm 1$ eigenspaces of $\what{\tau}$, and by $\Ker(\Delta^{E \pm}_{\ell})$ the intersections $\cA^*_{\pm}(F_s, E) \cap \Ker(\Delta^E_{\ell})$.  
Denote by $\pi_{\pm}  = \frac{1}{2}(P_{\ell} \pm \what{\tau} \circ P_{\ell})$, and note that these are the projections onto $\Ker (\Delta^{E \pm}_{\ell})$, respectively.  Since the operator $\what{\tau} $ satisfies the hypothesis of Lemma \ref{AK}, both $\pi_{\pm}$  are transversely smooth, and their Chern-Connes characters  $\ch_a (\pi_{\pm})$ are well defined Haefliger cohomology classes.  

\begin{definition}
Suppose that the projection $P_{\ell}$ onto $\Ker (\Delta^E_{\ell})$  is 
transversely smooth.   The higher twisted harmonic signature $\sigma (F,E)$ is the
difference  
$$
\sigma (F,E) = \ch_a (\pi_{+}) - \ch_a (\pi_{-}).
$$
\end{definition}

To justify our claim that our assumption of transverse smoothness for $P_{\ell}$ holds in important cases, we have the following which is an extension of a result due to Gong and Rothenberg, \cite{GR}.
\begin{theorem}\label{GRthm}
If the leafwise parallel translation along $E$ is a bounded map, then the projection $P$ onto $\Ker(\Delta^E)$ is transversely smooth.
\end{theorem}
The conclusion of Gong-Rothenberg is that the Schwartz kernel of $P$ is smooth in all its variables. 

For Riemannian foliations, $P$ is always transversely smooth for the classical signature operator 
(that is, with coefficients in the trivial one dimensional bundle) using either the holonomy or the homotopy groupoid.   $P$ is  transversely smooth whenever the preserved metric on $E$ is positive definite. It is smooth in important examples, e.g. Lusztig \cite{Lusztig}.  If the leafwise parallel translation along $E$ is a bounded map, $P$ is also transversely smooth using the holonomy groupoid, provided that the flat structure on $E$ over each holonomy covering has no holonomy (so using the flat structure to translate a frame of a single fiber of  $E \, | \, \wL$ to all of $\wL$ trivializes $E \, | \, \wL$).

It is an open question whether the projection to the leafwise harmonic forms has transversely smooth Schwartz kernel when $F$ is not Riemannian.  It is  satisfied for all foliations with compact leaves and Hausdorff groupoid \cite{EdwardsMillettSullivan, Epstein}.

Since the paper \cite{GR} has not been published, we give their proof here that $P$ depends smoothly on $x \in M$, and then show how to get transverse smoothness from it.

\begin{proof}
Let $U \subset M$ be a foliation chart and choose $x_0 \in U$.  Then there a diffeomorphism $\varphi_U: U \times \wL_{x_0}  \simeq   s^{-1}(U)$, and a bundle isomorphism $\psi_U:U \times (E \, | \, \wL_{x_0})  \simeq   E \, | \, s^{-1}(U)$, covering $\varphi_U$ and preserving the leafwise flat structure.  They are constructed as follows.
The normal bundle $\nu_s = TF_r \oplus \nu_{\cG} \simeq s^*(TM)$ defines a local transverse
translation for the leaves of the foliation $F_s$.  See \cite{Hurder, Wink}.  We may assume that  $U$ is the diffeomorphic image under $\exp_{x_0}$ of a neighborhood $\what{U}$ of $0 \in TM_{x_0}$.  Then for all $x \in U$, there is a unique $X \in \what{U}$ so that $x = 
\exp_{x_0}(X)$.  Define $\gamma_x:[0,1] \to M$ to be $\gamma_x (t) = 
\exp_{x_0} (tX)$.  {Given $x$ sufficiently close to $x_0$,  for any  $z \in \wL_{x_0}$} 
there is a unique path $\what{\gamma}_z(t)$ in $\cG$ so that 
$\what{\gamma}_z(0) = z$,  $\what{\gamma}_z(t) \in \wtit L_{\gamma_x 
(t)}$, and $\what{\gamma}_z' (t) \in (\nu_{s})_{\what{\gamma}_z(t)}$.  The 
transverse translate $\Phi_x (z)$ of $z$ to $\wtit L_x$ is just 
$\what{\gamma}_z(1)$.  $\Phi_x$ is a smooth diffeomorphism from $\wL_{x_0}$ to $\wL_x$, and we set $\varphi_U (x,z) = \Phi_x (z)$, which is a smooth diffeomorphism from $U \times \wL_{x_0}$  to   $s^{-1}(U)$.

Since we are using the homotopy groupoid, each $\wL$ is simply connected, so $E \, | \, \wL_{x}$ is a trivial bundle for each $x \in M$, and using the flat structure to translate a frame of a single fiber of  $E \, | \, \wL_x$ to all of $\wL_x$ trivializes $E \, | \, \wL_x$.  Choose a local orthonormal framing $e_1,...,e_k$ of $E \, | \, U$ (on $M$).
This framing is also a local framing of $E \, | \, i(U)$ (on $\cG$).
Using the leafwise flat structure of $E$ to translate it along the $\wL$, we get a leafwise flat framing $e^s_1,...,e^s_k$ of $E \, | \, s^{-1}(U)$.  For $(x, \sum_j a_j e^s_j (z)) \in U \times (E \, | \, \wL_{x_0})$, set
$$
\psi_U(x, \sum_j a_j e^s_j (z)) = \sum_j a_j e^s_j (\varphi_U(x,z)).
$$
That is, the image of $\phi \in E_z$ (where $z \in \wL_{x_0}$) under $\psi_U(x, \cdot)$ is obtained by first parallel translating $\phi$ along $\wL_{x_0}$ to $E_{i(x_0)}$, obtaining  $\sum_j a_j e_j (i(x_0))$, and then parallel translating $\sum_j a_j e_j (i(x))$ along $\wL_x$ to $E_{\varphi_U(x, z)}$.  It is clear that $\psi_U$ covers $\varphi_U$ and  preserves the leafwise flat structure.

There is a naturally defined bundle map 
$$
\Psi_U (x): \wedge T^*\wL_{x_0} \otimes (E \, | \, \wL_{x_0}) \to 
\wedge T^*\wL_{x} \otimes (E \, | \, \wL_{x})
$$ 
for each $x \in U$, which on a decomposable element $\alpha \otimes \phi \in (\wedge T^*\wL_{x_0} \otimes (E \, | \, \wL_{x_0}))_z$ is given by 
$$
\Psi_U (x)(\alpha \otimes \phi) = ((\Phi^{-1}_x)^*\alpha) \otimes \psi_U(x,\phi).
$$
We also denote by $\Psi_U$ the induced map 
$$
\Psi_U (x):C^{\infty}_c(\wL_{x_0}; \wedge T^*\wL_{x_0} \otimes (E \, | \, \wL_{x_0})) \to 
 C^{\infty}_c(\wL_x; \wedge T^*\wL_{x} \otimes (E \, | \, \wL_{x})).
$$ 
$\Psi_U (x)$  is invertible, commutes with the extended de Rham operators, and  depends smoothly on $x$.  Note that $\Phi^{-1}_x$ is a diffeomorphism of uniformly
bounded dilation (as is $\Phi_x$).   If the leafwise parallel translation along $E$ is a bounded map, then the map $\psi_U$ is a bounded map, and $\Psi_U$ extends to the following commutative diagram,

\medskip
\begin{picture}(415,80)
\put(55,60){$L^2(\wL_{x_0}; \wedge T^*\wL_{x_0} \otimes (E \, | \, \wL_{x_0})) $}
\put(105,50){ $\vector(0,-1){20}$}
\put(55,10){$L^2(\wL_{x}; \wedge T^*\wL_{x} \otimes (E \, | \, \wL_{x}))$}
\put(75,38){$\Psi_U (x)$}
\put(203,70){$d^{x_0}_s$}
\put(190,64){\vector(1,0){40}}
\put(190,13){\vector(1,0){40}}
\put(203,19){$d^{x}_s$}
\put(240,60){$L^2(\wL_{x_0}; \wedge T^*\wL_{x_0} \otimes (E \, | \, \wL_{x_0})) $}
\put(260,50){ $\vector(0,-1){20}$}
\put(240,10){$L^2(\wL_{x}; \wedge T^*\wL_{x} \otimes (E \, | \, \wL_{x})).$}
\put(270,38){$\Psi_U (x)$}
\end{picture}
\vspace{-2cm}
\begin{Equation}\label{GRCD}
\end{Equation}
\vspace{1.5cm}\noindent
So $\Psi_U (x)(\Ker(d^{x_0}_s)) \subset \Ker(d^{x}_s)$ and
$\Psi_U (x)(\overline{\Im(d^{x_0}_s)}) \subset \overline{\Im(d^{x}_s)}.$  By Lemma \ref{hodge}, we have 
$$
L^2(\wL_{x_0}; \wedge T^*\wL_{x_0} \otimes (E \, | \, \wL_{x_0})) = \Ker(d^{x_0}_s + \delta^{x_0}_s) \oplus \overline{\Im(d^{x_0}_s)} \oplus \overline{\Im(\delta^{x_0}_s)},
$$
and
$$
L^2(\wL_{x}; \wedge T^*\wL_{x} \otimes (E \, | \, \wL_{x})) = \Ker(d^{x}_s + \delta^{x}_s) \oplus \overline{\Im(d^{x}_s)} \oplus \overline{\Im(\delta^{x}_s)}.
$$
With respect to these decompositions, we may write
$$
\Psi_U (x) = \left[\begin{array}{ccc} 
\Psi_{11} (x)   &   0   &   \Psi_{13} (x) \\
&&\\
\Psi_{21} (x)   &  \Psi_{22} (x)  &   \Psi_{23} (x) \\
&&\\
0   &   0   &   \Psi_{33} (x)
\end{array} \right].
$$
It follows immediately that $\Psi_{22} (x): \overline{\Im(d^{x_0}_s)} \to  \overline{\Im(d^{x}_s)} $ is an invertible map which depends smoothly on $x$.  Let $R_{x_0}:L^2(\wL_{x_0}; \wedge T^*\wL_{x_0} \otimes (E \, | \, \wL_{x_0})) \to \overline{\Im(d^{x_0}_s)}$ be the orthogonal projection.   
Define
$$
\wtit{R}_x  \,\, = \,\,  \Psi_{22}(x) R_{x_0} \Psi^{-1}_U(x), \quad  \text{which equals}  \quad\Psi_U(x) R_{x_0} \Psi^{-1}_U(x),
$$
since $\Psi_{22}(x) R_{x_0} = \Psi_U(x) R_{x_0}$.
Then $\wtit{R}_x$ is an idempotent which varies smoothly in $x$, and has image $\overline{\Im(d^{x}_s)}$.  However, it might not be an orthogonal projection.  Set 
$$
Q_x =  I +(\wtit{R}_x - \wtit{R}^*_x)(\wtit{R}^*_x - \wtit{R}_x).
$$
Then $Q_x$ is an invertible self adjoint operator which depends smoothly on $x$, and
the orthogonal projection $R_{x}:L^2(\wL_{x}; \wedge T^*\wL_{x} \otimes (E \, | \, \wL_{x})) \to \overline{\Im(d^{x}_s)}$  is just 
$$
R_{x} \,\, = \,\, \wtit{R}_x \, \wtit{R}^*_x \, Q^{-1}_x,
$$ 
so $R_{x}$ depends smoothly on $x$.

Let $\tau_x$ be the Hodge type operator such that $\delta^{x}_s =  \pm \tau^{-1}_x \what{d}^x_s \tau_x$ , where $\what{d}^x_s$ is the differential associated with the antidual bundle $\overline{E}^*$ of $E$.  The operator $\tau^{-1}_x$ maps $\overline{\Im(\what{d}^{x}_s)}$ onto  $\overline{\Im(\delta^{x}_s)}$.  Set
$\what{S}_x = \tau^{-1}_x   \wtit{S}_x   \tau_x$, where $\wtit{S}_x$ is the operator for $\what{d}^x_s$ corresponding to the operator $\wtit{R}_x$ for $d^x_s$.  The argument above, with $E$ replaced by its antidual $\overline{E}^*$, shows that $\wtit{S}_x$, so also $\what{S}_x$, is an idempotent depending smoothly on $x$.  Note that $\what{S}_x$ has image $\overline{\Im(\delta^{x}_s)}$.  As above, we get that the orthogonal projection $S_x:L^2(\wL_{x}; \wedge T^*\wL_{x} \otimes (E \, | \, \wL_{x})) \to \overline{\Im(\delta^{x}_s)}$  depends smoothly on $x$.  Thus the orthogonal projection $P = I -(R_{x} + S_{x})$ depends smoothly on $x$.

\medskip
We now show that $P$ is transversely smooth.  To do this, we view everything  on $U \times \wL_{x_0}$, using   $\varphi_U$, and $\psi_U$.  Thanks to Diagram \ref{GRCD}, 
we are reduced to considering the operator $d^{x_0}_s:L^2(\wL_{x_0}; \wedge T^*\wL_{x_0} \otimes (E \, | \, \wL_{x_0})) \to L^2(\wL_{x_0}; \wedge T^*\wL_{x_0} \otimes (E \, | \, \wL_{x_0}))$ acting over each point $x \in U$, that is the twisted leafwise de Rham operator on the foliation  $U \times \wL_{x_0}$.   We use $\phi_U$ and $\psi_U$ to pull back the structures on $s^{-1}(U)$ and we use the same notation to denote these pull backs.  In particular, we have the connection $\nabla$ and the normal bundle $\nu_s$ used to define $\pa_{\nu}$.  The leafwise projection $P_x$ onto the twisted leafwise harmonics depends on the leafwise metrics on $\wL_{x_0}$ and  $E \, | \, \wL_{x_0}$, which vary with $x \in U$.   

First we prove that we may assume that the normal bundle $\nu_s$ is the bundle $TU \subset T(U \times \wL_{x_0})$.  Denote the operator corresponding to $\pa_{\nu}$ constructed using $TU$ by $\pa_U$.  Given a (bounded) vector field $Y$ on $U$, we have two lifts, $\what{Y}$ to $\nu_s$ and $\what{Y}_0$ to $TU$.   
The difference $\what{Y} - \what{Y}_0$ is tangent to the fibers $\wL_{x_0}$, so the difference of the operators $\pa_{\nu}^Y - \pa_U^Y = [\nabla_{\what{Y}} - \nabla_{\what{Y}_0}, \cdot]  = [\nabla_{\what{Y}-\what{Y}_0}, \cdot]$ is the commutator with a leafwise differential operator of order one, whose coefficients and all their derivatives are uniformly bounded, with the bound possibly depending on the order of the derivative.  
For $s \in \Z$, we denote by $W_s = W_s^*(\wL_{x_0}, E)$ the usual $s$-th Sobolev space which  is the completion of $C^{\infty}_c(\wL_{x_0}; \wedge T^*\wL_{x_0} \otimes (E \, | \, \wL_{x_0}))$ under the usual $s$-th Sobolev norm. 
Then $\Upsilon(Y) := \nabla_{\what{Y}} - \nabla_{\what{Y}_0}$ defines a bounded leafwise operator from any  $W^*_s$ to $W^*_{s-1}$, and both $\Upsilon(Y) P_x$ and $P_x \Upsilon(Y)$ are  bounded leafwise smoothing operators since $P_x$ is leafwise smoothing.  As 
$$
\pa_{\nu}^Y P_x = \pa_U^Y P_x  +   [\Upsilon(Y), P_x],
$$ 
$\pa_{\nu}^Y P_x$ is a bounded leafwise smoothing operator  if and only  $\pa_U^Y P_x$ is. 

Now assume that for all $Y_1, Y_2$, $\pa^{Y_1}_U P_x$and $\pa^{Y_2}_U\pa^{Y_1}_U P_x $ are bounded leafwise smoothing operators.  Again using the fact that  $\pa_{\nu}^Y  = \pa_U^Y   +   [\Upsilon(Y), \cdot]$, we have
$$
\pa^{Y_2}_{\nu}\pa^{Y_1}_{\nu} P_x  =   \pa_U^{Y_2}\pa_U^{Y_1} P_x   + 
[\pa^{Y_2}_U \Upsilon(Y_1),P_x]     +
[ \Upsilon(Y_1),\pa^{Y_2}_U P_x]  +
[\Upsilon(Y_2), \pa_U^{Y_1} P_x]    +
[\Upsilon(Y_2), [\Upsilon(Y_1), P_x]].
$$
which is a bounded leafwise smoothing operator since  $\pa^{Y_2}_U \Upsilon(Y_1)$ has the same properties as $\Upsilon(Y_1)$, namely it is a leafwise differential operator  of order one, whose coefficients and all their derivatives are uniformly bounded, with the bound possibly depending on the order of the derivative.  As the arguement is symmetric in $\pa_{\nu}$ and $\pa_U$,   $\pa^{Y_1}_{\nu} P_x $and $\pa^{Y_2}_{\nu}\pa^{Y_1}_{\nu} P_x $ are bounded leafwise smoothing operators
if and  only if $\pa^{Y_1}_U P_x $and $\pa^{Y_2}_U\pa^{Y_1}_U P_x$ are.  
Continuing in this manner, we have that 
$\pa^{Y_1}_{\nu} P_x $, $\pa^{Y_2}_{\nu}\pa^{Y_1}_{\nu} P_x$, ..., and $\pa^{Y_{m}}_{\nu}...\pa^{Y_1}_{\nu} P_x$
are bounded leafwise smoothing operators
if and  only if 
$\pa^{Y_1}_U P_x $,  $\pa^{Y_2}_U\pa^{Y_1}_U P_x$, ...,  and $\pa^{Y_{m}}_U...\pa^{Y_1}_U P_x$ are.  Thus we may assume that $\nu_s = TU$.

Next we show that we may use any connection we please, provided it is in the same bounded geometry class as $\nabla$.   Suppose that $\pa_0$ is another derivation constructed from the connection $\nabla^0$ in the  same bounded geometry class as $\nabla$.  Then 
$\pa_{\nu}^Y - \pa_0^Y = [\nabla^{\nu}_Y - \nabla^{0, \nu}_Y, \cdot ]$, and $\nabla^{\nu}_Y - \nabla^{0, \nu}_Y$ is a leafwise differential operator of order zero, whose coefficients and all their derivatives are uniformly bounded, with the bound possibly depending on the order of the derivative.  So $\nabla^{\nu}_Y - \nabla^{0, \nu}_Y$ defines a bounded operator from any Sobolev space $W^s$ to itself.  Proceeding just as we did above, we have that 
$\pa^{Y_1}_{\nu} P_x $, $\pa^{Y_2}_{\nu}\pa^{Y_1}_{\nu} P_x$, ...,  and $\pa^{Y_{m}}_{\nu}...\pa^{Y_1}_{\nu} P_x$ are bounded leafwise smoothing operators if and  only if $\pa^{Y_1}_0 P_x $,  $\pa^{Y_2}_0 \pa^{Y_1}_ 0P_x$, ...,  and $\pa^{Y_{m}}_0...\pa^{Y_1}_0 P_x$ are.  Thus, we are reduced to showing that $\pa_0^{Y_{m}}... \pa_0^{Y_1}(P)$ is a bounded leafwise smoothing operator.

The connection we choose is that pulled back from $L_{x_0}$ under the obvious map 
$U \times \wL_{x_0} \to L_{x_0}$.  We leave it to the reader to show that this is in the same bounded geometry class as $\nabla$.  Now we can choose coordinates on $U$ so we may think of $U = \D^n$ with coordinates, $x_1, ..., x_n$, and $x_0 = (0,...,0)$.   When we do, 
$$
\pa_0^{\pa/\pa x_{i_{m}}} ... \pa_0^{\pa/\pa x_{i_1}}P_x = \pa^{m}P_x/\pa x_{i_{m}} ... \pa x_{i_1}.
$$ 
Thus we are reduced to considering a smooth family of smoothing operators $P_x$ acting on the space of sections of $\wedge T^*\wL_{x_0} \otimes (E \, | \,  \wL_{x_0})$.  The parameter $x$ determines the metric $g_x$ we use on this space, and $P_x$ is the associated projection onto the twisted harmonic sections.  Note that the associated Sobolev spaces $W^*_s$ are the same for all the $g_x$ since these metrics are all in the same bounded geometry class.  The norms on $W^*_s$ do depend on the parameter $x$.   However they are all comparable, so we may assume that we have a single norm $|| \cdot ||_s$ on each $W^*_s$, which is independent of $x$.

Denote $ \pa^{m}/ \pa x_{i_{m}} ... \pa x_{i_1}$  by $ \pa^{m}_{i_{m} ... {i_1}}$.    We need to prove that for all $s$ and $k \geq 0$, $  \pa^{m}_{i_{m} ... {i_1}}P_x $ defines a bounded map from $W^*_s$ to $W^*_{s+k}$.    Given $K: W^*_s \to W^*_{s+k}$, denote the $s, s+k$ norm of $K$ by 
$||K||_{s+k,s}$.  Then 
$$
||K||_{s+k,s} = ||(1 +\Delta)^{(s+k)/2}K(1 +\Delta)^{-s/2}||_{0,0}, 
$$where $\Delta$ is the Laplacian associated to the metric on $\wedge T^*\wL_{x_0} \otimes (E \, | \, \wL_{x_0})$.  Since the norms associated to different metrics are comparable, we may use any metric $g_x$ with associated Laplacian $\Delta_x$ we like.  
Now $P_x   =  (1 +\Delta_x)^{(s+k)/2}P_x  (1 +\Delta_x)^{-s/2}$, so
$$
\pa_i P_x \,\, = \,\, 
\pa_i ((1 +\Delta_x)^{(s+k)/2}P_x  (1 +\Delta_x)^{-s/2}) \,\, = \,\, 
$$
$$  
\pa_i (1 +\Delta_x)^{(s+k)/2}   P_x (1 +\Delta_x)^{-s/2} +   (1 +\Delta_x)^{(s+k)/2} \pa_i P_x (1 +\Delta_x)^{-s/2}  +  (1 +\Delta_x)^{(s+k)/2} P_x \pa_i (1 +\Delta_x)^{-s/2},
$$
which gives
$$
(1 +\Delta_x)^{(s+k)/2} \pa_i P_x (1 +\Delta_x)^{-s/2} \,\, = \,\,  
$$
$$\pa_i P_x   - \pa_i (1 +\Delta_x)^{(s+k)/2}   P_x (1 +\Delta_x)^{-s/2} -
 (1 +\Delta_x)^{(s+k)/2} P_x \pa_i (1 +\Delta_x)^{-s/2}.
$$
So,
$$
|| \pa_i P_x ||_{s+k,s} \,\, = \,\, ||(1 +\Delta_x)^{(s+k)/2} \pa_i P_x (1 +\Delta_x)^{-s/2} ||_{0,0} \,\, = \,\,
$$
$$
||\pa_i P_x  - \pa_i (1 +\Delta_x)^{(s+k)/2}    P_x (1 +\Delta_x)^{-s/2} -
 (1 +\Delta_x)^{(s+k)/2} P_x \pa_i (1 +\Delta_x)^{-s/2}  ||_{0,0}      \,\, \leq \,\,
$$
$$
||\pa_i P_x   ||_{0,0} \,\, +  \,\,  ||   \pa_i (1 +\Delta_x)^{(s+k)/2} P_x (1 +\Delta_x)^{-s/2}  ||_{0,0} \,\, +  \,\,  || (1 +\Delta_x)^{(s+k)/2} P_x \pa_i (1 +\Delta_x)^{-s/2}   ||_{0,0}.
$$
Now for any $r$,  $(1 +\Delta_x)^{r/2}$ and $\pa_i (1 +\Delta_x)^{r/2}$ are leafwise differential operators of order $r$, whose coefficients and all their derivatives are uniformly bounded, with the bound possibly depending on the order of the derivative, {\em  but independent of} $x$.  So they define bounded operators from $W^*_s$ to $W^*_{s-r}$, for any $s$, with bound independent of $x$.
Since $P_x$ is leafswise smoothing, it defines a bounded operator from any $W^*_r$ to any $W^*_s$, whose bound is also independent of $x$, since  
$$
|| P_x ||_{s,r}  \,\, = \,\, ||  (1 +\Delta_x)^{-s/2}P_x   (1 +\Delta_x)^{-r/2}||_{0,0}  \,\, = \,\, || P_x ||_{0,0} \,\, \leq \,\, 1,
$$
Thus we have 
$$
||\pa_i (1 +\Delta_x)^{(s+k)/2}   P_x (1 +\Delta_x)^{-s/2}  ||_{0,0} \,\, \leq \,\,  
||\pa_i (1 +\Delta_x)^{(s+k)/2}  ||_{0,s+k}  \,\, || P_x ||_{s+k,-s}  \,\, ||  (1 +\Delta_x)^{-s/2}  ||_{-s,0}
$$ 
is bounded independently of $x$.  Similarly
$|| (1 +\Delta_x)^{(s+k)/2} P_x \pa_i (1 +\Delta_x)^{-s/2} ||_{0,0}$ is bounded independently of $x$.  Thus $ \pa_i P_x : W^*_s \to W^*_{s+k}$ is bounded if and only if $\pa_i P_x: W^*_0 \to W^*_0$ is.  

Now for any $m$ and and any $r$,  $\pa^{m}_{i_m ... i_1}(1 +\Delta_x)^{r/2}$ is also a  leafwise differential operator of order $r$, whose coefficients and all their derivatives are uniformly bounded, with the bound possibly depending on the order of the derivative, but independent of $x$.  Using this fact, a straightforward induction argument shows that 
$\pa^{m}_{i_m ... i_1}P_x :W^*_s \to W^*_{s+k}$ is bounded if and only if 
$\pa^{m}_{i_m ... i_1}P_x:W^*_0 \to W^*_0$ is.

Now we have (working on $W^*_0 = L^2(\wL_{x_0}; \wedge T^*\wL_{x_0} \otimes (E \, | \, \wL_{x_0})) $) that  $P_x = I -(R_x + S_x)$, where  $R_x$ is the orthogonal projection  $R_x:L^2(\wL_{x_0}; \wedge T^*\wL_{x_0} \otimes (E \, | \, \wL_{x_0}))  \to \overline{\Im(d^{x_0}_s)} \subset L^2(\wL_{x_0}; \wedge T^*\wL_{x_0} \otimes (E \, | \, \wL_{x_0}))$ obtained using the metric $g_x$. 
At the point $x$,  $R_{x_0}$ also has image $\overline{\Im(d^{x_0}_s)}$, but $R_{x_0}$ might not be an orthogonal projection using the metric $g_x$.  As above $R_{x}$ is given by
$$
R_{x} \,\, = \,\, R_{x_0} \, R_{x_0}^{*x} \, Q^{-1}_x,
$$
where
$$
Q_x =  I +(R_{x_0} - R_{x_0}^{*x})(R_{x_0}^{*x} - R_{x_0}).
$$
and $R_{x_0}^{*x}$ is  the adjoint of $R_{x_0}$ constructed using the metric $g_x$.
Since $I = Q_x Q_x^{-1}$, we have that 
$$
0  \,\, = \,\, \pa_i I    \,\, = \,\, \pa_i(Q_x Q_x^{-1})  \,\, = \,\, 
(\pa_i Q_x ) Q_x^{-1}   +   Q_x  (\pa_i Q_x^{-1}).
$$
So
$$
\pa_i Q_x^{-1}  \,\, = \,\, -Q_x^{-1}(\pa_i Q_x ) Q_x^{-1},
$$
and a boot-strapping argument shows that  $\pa^{m}_{i_{m}...{i_1}}(Q_x^{-1})$ is bounded if $\pa^{m}_{i_{m} ... {i_1}}Q_x$ is.  It follows that  $\pa^{m}_{i_{m} ...  x_{i_1}}R_x$ is bounded if  $\pa^{m}_{i_{m} ... {i_1}}R_{x_0}$  and $\pa^{m}_{i_{m} ... {i_1}}R_{x_0}^{*x}$ are bounded.  As $\pa_i R_{x_0}  = 0$ for all $i$, we are reduced to considering $R_{x_0}^{*x}$.

We may write the metric $g_x$ as $g_x(u,v)=g_{x_0}(G_x u, v)$ where $G_x$ is a nonnegative self-adjoint (invertible) operator with respect to $g_{x_0}$, as is its inverse.  Since $g_x$ is the pull back of a family of metrics defined on the compact manifold $M$,  $G_x$ is smooth in all its variables, and it and all its derivatives are uniformly bounded, and the same is true for the inverse $G_x^{-1}$.   Thus for all $m$, both $\pa^{m}_{i_m ... i_1}G_x$ and  $\pa^{m}_{i_m ... i_1}G_x^{-1}$ define bounded operators on $L^2(\wL_{x_0}; \wedge T^*\wL_{x_0} \otimes (E \, | \, \wL_{x_0}))$ (since they are order zero differential operators).  For any bounded operator $A$ on $L^2(\wL_{x_0}; \wedge T^*\wL_{x_0} \otimes (E \, | \, \wL_{x_0}))$, the adjoint of $A$ with respect to $g_x$ is
$$
A^{*x} = G_x^{-1} A^* G_x,
$$
where $A^*$ is the adjoint with respect to $g_{x_0}$.   It follows immediately that for all $m$, 
$\pa^{m}_{i_{m} ... {i_1}}R_{x_0}^{*x} =   \pa^{m}_{i_{m} ... {i_1}}(G_x^{-1}R_{x_0}^{*}G_x)$ is a bounded operator on  
$W^*_0 = L^2(\wL_{x_0}; \wedge T^*\wL_{x_0} \otimes (E \, | \, \wL_{x_0}))$, since $\pa_i  R_{x_0}^{*} = 0$ for all $i$.  Thus for all $m$, 
$\pa^{m}_{i_{m} ... {i_1}}R_{x}$ is a bounded operator on  
$W^*_0$.  

It remains to show that for all $m$, $\pa^{m}_{i_{m} ... {i_1}}S_{x}$ is a bounded operator on  
$W^*_0$.  To do this we may proceed as we did above, using the operators $\wtit{S}_x$ and  $\what{S}_x$.  We need only observe that the Hodge type operator $\tau_x$ has the same properties that $G_x$ does.  Thus for all $m$,
$\pa^{m}_{i_{m} ... {i_1}}P_x =  -(\pa^{m}_{i_{m} ... {i_1}}R_{x} + \pa^{m}_{i_{m} ... {i_1}}S_{x})$  is a bounded operator on  $W^*_0$, and we conclude that
$P_x$ is transversely smooth. 

\end{proof}

\begin{proposition}  If $P$ is transversely smooth, then the
 projections onto $\cA^*_{\pm}(F_s,E) \cap (\Ker(\Delta^E_k) \oplus  
\Ker(\Delta^E_{p-k}))$, $k \neq \ell$, and  $\Ker(\Delta^{E\pm}_{\ell})$  are 
transversely smooth.
\end{proposition}

\begin{proof}
Denote by $P_k$ the projection onto $\Ker( \Delta^E_k)$.  It is immediate that 
$P$ is transversely smooth if and only if all the $P_k$ are transversely 
smooth.  For $k \neq \ell$, the projection onto $\cA^*_{\pm}(F_s,E) \cap 
(\Ker(\Delta^E_k) \oplus  \Ker(\Delta^E_{p-k}))$ is given by $\pi_k^\pm=P_k 
\pm \tau \circ P_k$, (since $P_k \circ \tau \circ P_k = 0$ in those cases), and the projection onto 
$\Ker(\Delta^{E\pm}_{\ell})$ is given by
$\pi_\pm = \frac{1}{2}(P_{\ell} \pm \tau \circ P_{\ell})$.   As the operator $\tau$ 
satisfies the hypothesis of Lemma \ref{AK}, and each $P_k$ is transversely 
smooth, so is each  $\tau \circ P_k$, so all of the projections are also 
transversely smooth.  
\end{proof}

\section{Connections, curvature, and the Chern-Connes character } \label{ccnc} 

We now give an alternate construction of the Chern-Connes characters 
$\ch_a(\pi_+)$ and $\ch_a(\pi_-)$  using ``connections"  and 
``curvatures" defined on  ``smooth sub-bundles" of
$\cA_{(2)}^*(F_{s}, E)$.
\begin{definition}
A smooth subbundle  of $\cA_{(2)}^* (F_{s},E)$ over $M/F$ is a $\cG$ 
invariant transversely smooth idempotent $\pi_0$ acting on $\cA_{(2)}^* 
(F_{s},E)$.  
\end{definition}

\begin{example}
\begin{enumerate}
\item Any idempotent in the algebra of superexpeonentially decaying operators on $\wedge T^*F_s \otimes E$, defined in \cite{BHII}, is a smooth subbundle  of $\cA_{(2)}^* (F_{s},E)$ over $M/F$. So, 
any smooth compactly supported idempotent is a smooth subbundle  of $\cA_{(2)}^* (F_{s},E)$ over $M/F$.

\item   The Wassermann  idempotent of the leafwise signature operator, as defined for instance in \cite{BHII},  is a very important special case of $(1)$ above.  In this case we take $E = M \times \C$.

\item A paradigm for such a smooth subbundle is given by projection onto the 
kernel of a leafwise elliptic operator acting on $\cA_{(2)}^* (F_{s},E)$ 
(induced from a leafwise elliptic operator on $F$). { In particular,} the projections 
$\pi_{+}$ and $\pi_-$. 
\end{enumerate}
\end{example}

\begin{definition}
{The space  $C^\infty_2(\wedge T^*F_s \otimes E)$ consists of all
elements} $\xi \in C^{\infty}(\cG;\wedge T^*F_s \otimes E) \cap 
\cA^*_{(2)}(F_{s},E)$ such that for any quasi-connection $\nabla^{\nu}$, 
and any  vector fields $Y_1, ... , Y_{m}$ on $M$,  \
$$
\nabla^{\nu}_{Y_1} ...\nabla^{\nu}_{Y_{m}}(\xi)
 \in  C^{\infty}(\cG; \wedge T^*F_s  \otimes E) \cap  \cA^*_{(2)}(F_{s}, E),
$$
{where $\nabla^{\nu}_{Y_i} = i_{\what{Y_i}}\nabla^{\nu}$.}
\end{definition}

Note that if $\xi \in C^{\infty}(\cG;\wedge T^*F_s \otimes E) $, $
\nabla^{\nu}_{Y_1} ...\nabla^{\nu}_{Y_{m}}(\xi)$ is automatically in 
$C^{\infty}(\cG;\wedge T^*F_s \otimes E)$, and that if $\xi \in C^\infty_2 (\wedge 
T^*F_s \otimes E)$, then 
$\nabla^{\nu}_{Y_1} ...\nabla^{\nu}_{Y_{m}}(\xi) \in C^\infty_2 (\wedge 
T^*F_s \otimes E)$.   Note also that $C^{\infty}_c(\cG;\wedge T^*F_s \otimes E) \subset 
C^\infty_2 (\wedge T^*F_s \otimes E)$.

\begin{proposition}\label{transxi}
If $H$ is a transversely smooth operator on $\wedge T^*F_s \otimes E$, then $H$ maps 
$C^\infty_2(\wedge T^*F_s \otimes E)$ to itself.
\end{proposition}

\begin{proof}
Let $\xi \in C^\infty_2(\wedge T^*F_s \otimes E)$.   As $H$ is transversely smooth,  
it follows easily that $H \xi \in  C^{\infty}(\cG;\wedge T^*F_s \otimes E) \cap  
\cA^*_{(2)}(F_{s}, E)$.  Fix a quasi-connection $\nabla^{\nu}$, and let 
$Y$ be a vector field on $M$. { Then}
$$
\nabla^{\nu}_ Y(H\xi)  =   \nabla^{\nu}_ YH\xi  - H  \nabla^{\nu}_ Y\xi +  
H  \nabla^{\nu}_ Y\xi  = (\pa_{\nu}^Y H)  \xi +  H  (\nabla^{\nu}_ Y\xi), 
$$
which is in $C^{\infty}(\cG;\wedge T^*F_s \otimes E) \cap  \cA^*_{(2)}(F_{s},E)$, 
since $H$ and $\pa_{\nu}^Y H$ 
are transversely smooth, and $\xi$ and $\nabla^{\nu}_ Y\xi$ are  in 
$C^\infty_2(\wedge T^*F_s \otimes E)$. 
An obvious induction argument now shows that $H\xi \in C^\infty_2(\wedge 
T^*F_s \otimes E)$.
\end{proof}

Let $\pi_0$ be a smooth subbundle  of $\cA_{(2)}^* (F_{s},E)$ over $M/F$. 
\begin{definition}
A smooth section of $\pi_0$ is an element $\xi \in C^\infty_2 (\wedge 
T^*F_s \otimes E)$ which satisfies $\pi_0 \xi = \xi$.  
The set of all smooth sections is denoted $C^{\infty}(\pi_0)$.
\end{definition}

The space $C^{\infty}(\pi_0)$ is a $C^\infty (M)$ module, where $(f\cdot 
\xi)([\gamma]) = f(s(\gamma)) \xi([\gamma])$.  In addition, 
$C^{\infty}(\pi_0) = \pi_0( C^\infty_2 (\wedge T^*F_s \otimes E)) \, \supset \,
\pi_0(C^{\infty}_c(\cG;\wedge T^*F_s \otimes E))$.  

\begin{definition}  Denote by $C^{\infty}(\wedge T^*M;\pi_0)$ the 
collection of all smooth sections of $\wedge T^*M$ with coefficients in 
$C^{\infty}(\pi_0)$, and by 
$C^{\infty}_c(\wedge T^*M;\wedge T^*F_s  \otimes E)$ the collection of all smooth 
sections of $\wedge T^*M$ with coefficients in $C^{\infty}_c(\cG;\wedge T^*F_s \otimes E)$.
\end{definition}

There are natural actions of $\cA^*(M)$  on $C^{\infty}(\wedge T^*M;\pi_0)$ and 
$C^{\infty}_c(\wedge T^*M;\wedge T^*F_s \otimes E)$, and under these actions
$$
C^{\infty}(\wedge T^*M;\pi_0) \simeq \cA^*(M) {{\hat{\otimes}}}_{C^{\infty}(M)} 
C^{\infty}(\pi_0),
$$
and
$$
C^{\infty}_c(\wedge T^*M;\wedge T^*F_s \otimes E) \simeq \cA^*(M) 
{{\hat{\otimes}}}_{C^{\infty}(M)} C^{\infty}_c(\cG;\wedge T^*F_s \otimes E),
$$
with the right completions.
Thus $\pi_0:C^{\infty}_c(\cG;\wedge T^*F_s \otimes E) \to C^{\infty}(\pi_0)$ extends 
to the $\cA^*(M)$ equivariant map
$$
\pi_0:C^{\infty}_c(\wedge T^*M; \wedge T^*F_s \otimes E) \to C^{\infty}(\wedge 
T^*M;\pi_0).
$$

  A {\it local invariant} element is a {local section $\xi$ of $\cA^*_{(2)}(F_{s}, E)$}
defined on an open 
subset $U \subset M$ so that for any leafwise path $\gamma_1$ in $U$,
$\xi([\gamma]) = \xi([\gamma \gamma_1])$ for all $\gamma$ 
with $s(\gamma) =r(\gamma_1)$.
Local invariant elements are common.  In particular, any locally defined 
element  $\xi \in \cA^*_{(2)}(F_{s},E)$ defines local  invariant 
elements.  Suppose that $\xi$ is defined on a foliation chart $U \subset 
M$ for $F$, and let $P_x$ be  the placque in $U$ containing the point 
$x$.   Given $y \in P_x$, let $\gamma_y$ be a path in $P_x$ starting at 
$x$ and ending at $y$.  Define $\widetilde \xi_y \in L^2(\wL_y; \wedge 
T^*F_s \otimes E)$ by $\widetilde \xi_y([\gamma]) = \xi_x([\gamma\gamma_y] )$.   
Then $\widetilde \xi$ is a local invariant element of  $\cA^* 
_{(2)}(F_{s}, E)$ defined along $P_x$.  By restricting $\xi$ to a transversal 
$T$ in a foliation chart $U$ and then extending invariantly to $\widetilde 
\xi$ we obtain local invariant elements of $\cA^*_{(2)}(F_{s}, E)$ defined 
over $U$.   One can of course extend this construction from chart to chart 
as far as one likes, for example along any path $\gamma:[0,1] \to L$ in a 
leaf $L$.   If $\gamma$ is a closed loop, the section at $1$ will 
not agree in general with the section at $0$, so one does not in 
general obtain global invariant sections this way.  

\begin{definition}\label{connection}
A connection $\nabla$ on $\pi_0$ is a linear map 
$$
\nabla:C^{\infty}(\wedge T^*M;\pi_0) \to C^{\infty}(\wedge T^*M;\pi_0)
$$ 
of degree one, so that 
\begin{enumerate}
\item for  $\omega \in \cA^k (M)$ and  $\xi \in C^{\infty}(\pi_0)$,
$\nabla(\omega \otimes \xi) = d_M\omega \otimes \xi + (-1)^k \omega 
\wedge  \nabla \xi;$
\item for local invariant $\xi \in C^{\infty}(\pi_0)$, and $X \in  
C^{\infty}(TF)$, $\nabla_X \xi  =0$, i.\ e.\  $\nabla$ is 
flat along $F$;
\item  $\nabla$ is invariant under the right action of $\cG$;
\item the leafwise operator $\nabla \pi_0 - \pi_0 
\nabla^{\nu}\pi_0:C^{\infty}_c(\wedge T^*M;\wedge T^*F_s \otimes E)  \to 
C^{\infty}(\wedge T^*M;\pi_0)$ is transversely smooth.
\end{enumerate}
\end{definition}

The usual proof shows that since $\nabla$ satisfies (1), it is local in 
the sense that $\nabla \xi (x)$ depends only on $\xi \, | \, U$ where $U$ 
is any open set in $M$ with $x \in U$.   

For $\nabla$ to be invariant under the right action of $\cG$ means the 
following.  Let $\gamma$ be a leafwise path in $M$ from $x = s(\gamma)$ to 
$y = r(\gamma)$.   Let $\xi$ be a local invariant section of $\pi_0$ 
defined on a neighborhood of the path $\gamma$.    For $X\in \nu_x$, we 
may use the natural flat structure on $\nu$ to parallel translate $X$ to 
$\gamma_{*}(X) \in \nu_y$.  Then we require,
$$
\nabla_X \xi =  (R_{\gamma})^{-1} \nabla_{\gamma_{*}(X)} 
\xi=R_{\gamma^{-1}} \nabla_{\gamma_{*}(X)} \xi,
$$
where the isomorphism $R_{\gamma}:L^2 
(\wL_{s(\gamma)}; \wedge T^*F_s \otimes E) \to L^2 (\wL_{r(\gamma)}; \wedge T^*F_s \otimes E)$ 
is given by  $R_{\gamma} (\xi)([\gamma_1]) = \xi[(\gamma_1 \gamma])$.   
Note that this condition does not depend on the choice of normal bundle 
$\nu$ because the ambiguity involves things of the form $\nabla_Y \xi$ 
where $Y \in TF$.  But this is zero because $\nabla$ is flat along $F$.

To see that $\nabla \pi_0 - \pi_0 \nabla^{\nu}\pi_0$ is a leafwise operator, let
$\xi \in C^{\infty}_c(\cG;\wedge T^*F_s \otimes E)$, and $\omega \in  \cA^k (M).$
Then
$$
(\nabla \pi_0 - \pi_0 \nabla^{\nu}\pi_0) (\omega \otimes \xi) =  \pi_0(\nabla  
-  \nabla^{\nu}) \pi_0(\omega \otimes  \xi)  =  
 \pi_0( \nabla  -  \nabla^{\nu})(\omega \otimes  \pi_0( \xi)) =  
$$
$$  
 \pi_0 \Bigl{(} d_{M}\omega  \otimes  \pi_0( \xi) + (-1)^k \omega \wedge 
\nabla \pi_0( \xi) - d_{M}\omega  \otimes \pi_0( \xi) -  (-1)^k \omega
\wedge \nabla^{\nu}\pi_0( \xi)\Bigr{)} = $$
 $$
(-1)^k \pi_0\Bigl{(}\omega \wedge(\nabla -  \nabla^{\nu})\pi_0( 
\xi)\Bigr{)}  = (-1)^k \omega \wedge (\nabla \pi_0 - \pi_0 \nabla^{\nu}\pi_0)  
\xi,
$$
so $\nabla \pi_0 - \pi_0 \nabla^{\nu} \pi_0$ is a leafwise operator.

Next we show that $C^{\infty}_c(\wedge T^*M;\wedge T^*F_s \otimes E)$  is in the 
domain of $ \pi_0 \nabla^{\nu} \pi_0$.  
We identify 
$C^{\infty}_c(\wedge T^* M;\wedge T^*F_s \otimes E)$ with the subspace 
$C^{\infty}_c(\wedge  \nu_s^* \otimes \wedge  T^*F_s \otimes E))$ of 
$C^{\infty}(\wedge  \nu_s^* \otimes \wedge  T^*F_s \otimes E))$.  Now $\pa_{\nu} 
(\pi_0) = [\nabla^{\nu}, \pi_0]$, and (by assumption) it is transversely smooth.  
Thus we have 
$$
\nabla^{\nu} \pi_0 =  \pi_0 \nabla^{\nu} + \pa_{\nu}(\pi_0),
$$
\noindent  so
$$
\pi_0 \nabla^{\nu} \pi_0 =  \pi_0 \nabla^{\nu} + \pi_0\pa_{\nu}(\pi_0).
$$
The domain of the operator on the right contains $C^{\infty}_c(\wedge  
\nu_s^* \otimes \wedge  T^*F_s  \otimes E)$.

\begin{lemma}\label{existconn}
$\pi_0 \nabla^{\nu}$ is a connection on $\pi_0$.
\end{lemma}

\begin{proof} 
Since $\pi_0$ commutes with the action of $\cA^*(M)$ on 
$C^{\infty}(\wedge T^*M;\pi_0)$, to show that $ \pi_0 \nabla^{\nu}$ maps 
the space $C^{\infty}(\wedge T^*M;\pi_0)$ to itself, we need only show that for any 
local section $\xi \in C^{\infty}(\pi_0)$, and any local vector field $X$ 
on $M$,
$$
(\pi_0 \nabla^{\nu} \xi)(X)  \, = \, \pi_0(i_{\what{X}} \nabla^{\nu} \xi)  \, = \,
\pi_0(\nabla^{\nu}_{X} \xi)
$$
is in $C^{\infty}(\pi_0)$, where $\what{X}$ is the lift of $X$ to $\nu_s$.
As $\xi \in C^{\infty}(\pi_0)$, it is in $C^{\infty}_2(\wedge T^*F_s \otimes E)$ and $ \pi_0(\xi) = \xi$, so $\nabla^{\nu}_{X} \pi_0 \xi = \nabla^{\nu}_{X} \xi \in C^{\infty}_2(\wedge T^*F_s \otimes E)$.  
As $\pi_0$ is transversely smooth, $i_{\what{X}} \pa_{\nu}(\pi_0)( \xi) \in C^{\infty}_2(\wedge T^*F_s \otimes E)$.  Since $\pi_0 \nabla^{\nu} =  \nabla^{\nu} \pi_0 + \pa_{\nu}(\pi_0)$,  we have $(\pi_0 \nabla^{\nu} \xi)(X) \in C^{\infty}_2(\wedge T^*F_s \otimes E)$.
Finally, as $\pi_0^2 = \pi_0$, $\pi_0(\pi_0(\nabla^{\nu}_{X}  \xi)) = \pi_0(\nabla^{\nu}_{X}  \xi)$.  Thus 
$( \pi_0 \nabla^{\nu} \xi)(X) \in C^{\infty}(\pi_0)$, and  $ \pi_0 \nabla^{\nu}$ maps $C^{\infty}(\wedge T^*M;\pi_0)$ to itself.

The operator $ \pi_0 \nabla^{\nu}$ satisfies $(1)$ because $\pi_0$ commutes with the 
action of $\cA^*(M)$ on $C^{\infty}(\wedge T^*M;\pi_0)$.  In 
particular, for $\omega \in \cA^k(M)$ and $\xi \in C^{\infty}(\pi_0)$, 
we have
$$
\pi_0 \nabla^{\nu} ( s^*\omega \otimes  \xi) \, = \,  
\pi_0\rho_{\nu}(r^* (\nabla_F \otimes \nabla_E) (s^*\omega \otimes \xi))  \, = \, 
$$
$$
\pi_0  \rho_{\nu}\Bigr{(}d_{\cG} (s^*\omega)  \otimes \xi + (-1)^k s^*\omega \wedge r^*( \nabla_F \otimes \nabla_E) \xi\Bigr{)}    \, = \, 
$$
$$
\pi_0  \rho_{\nu}(s^*d_M \omega \otimes  \xi) +  (-1)^k  \pi_0 
\rho_{\nu}(s^* \omega \wedge  r^* (\nabla_F \otimes \nabla_E) \xi)    \, = \,
$$
$$
s^*d_M \omega  \otimes  \pi_0\xi + (-1)^k s^* \omega \wedge  \pi_0 \rho_{\nu} r^* (\nabla_F \otimes \nabla_E) \xi  \, = \,
d_M\omega  \otimes \xi + (-1)^k \omega \wedge  \pi_0 \nabla^{\nu}\xi.
$$

To show that $ \pi_0 \nabla^{\nu}$ satisfies $(2)$, let $X \in TF_x$, $\xi$ be a local 
invariant section of $\pi_0$ defined near $x$, and $[\gamma] \in \wL_x$.  
The fact that $\xi$ is invariant means that there is a section 
$\what{\xi}$ of $\wedge T^* F \otimes E$ defined in a neighborhood of $r(\gamma)$ so 
that $\xi = r^* \what{\xi}$ in a neighborhood of $[\gamma]$.    Recall 
that $\nu_s = TF_r \oplus \nu_{\cG}$.  Since $X \in TF_x$, $\what{X} \in 
TF_r$ and $r_*(\what{X})= 0$.  Now 
$$
 \pi_0 \nabla^{\nu}_X \xi  = \pi_0(  r^* (\nabla_F \otimes \nabla_E)_{\what{X}}(\xi)),
$$
but at $[\gamma]$, 
$$ 
\Bigl{(} r^* (\nabla_F \otimes \nabla_E)_{\what{X}}(\xi)\Bigr{)}[\gamma] =  
(\nabla_F \otimes \nabla_E)_{r_*(\what{X}[\gamma]) }\what{\xi} = 
(\nabla_F \otimes \nabla_E) _{0}\what{\xi} =  0,
$$
so $ \pi_0 \nabla^{\nu}_X \xi  = 0.$

We leave it to the reader to check that $ \pi_0 \nabla^{\nu}$ satisfies $(3)$ of 
Definition \ref{connection}, which is a straight forward computation, 
using the fact that for $X \in \nu_x$ and $[\gamma] \in \wL_x$, 
$r_*(\what{X}_{[\gamma]}) = \gamma_*(X)$, the parallel translate of $X$ 
along $\gamma$ to $\nu_{r(\gamma)}$.  

$ \pi_0 \nabla^{\nu}$ obviously satisfies $(4)$.  
\end{proof}

\begin{remark}
If $\what{\nabla}^{\nu}$ is another partial connection, then the difference 
$\what{\nabla}^{\nu} - \nabla^{\nu}$ is a leafwise operator which satisfies the hypothesis of Lemma \ref{AK}, so  $\pi_0 \what{\nabla}^{\nu}\pi_0  - \pi_0 \nabla^{\nu}\pi_0  = 
\pi_0 ( \what{\nabla}^{\nu} - \nabla^{\nu})\pi_0$ is transversely smooth and $\pi_0 \what{\nabla}^{\nu}$ is also a connection on $\pi_0$.  So, as in the classical case, the space of connections is an affine space whose linear part is composed of transversely smooth operators.
\end{remark}

Now suppose that $\nabla$ is any connection on $\pi_0$.  Define the curvature $\theta$ of $\nabla$ to be 
$$
\theta = \nabla^2.
$$
The usual computation shows that $\theta$ is a leafwise operator, that is 
\begin{lemma}\label{local}
For any  $\omega \in \cA^*(M)$ and any $\xi \in C^{\infty}(\pi_0)$, {$\nabla ^2(\omega \otimes \xi) = \omega \wedge \nabla ^2(\xi).$}
\end{lemma}

Denote by $C^{\infty}(\wedge T^*M; \cA^{*}_{(2)}(F_s \otimes E))$ the space of 
all smooth sections of $\wedge T^*M$ with coefficients in  $\cA^{*}_{(2)}(F_s \otimes E)$.  Smoothness means that the section is smooth when viewed as a section of $\wedge \nu_s^* \otimes 
\wedge T^*F_s \otimes E$ over $\cG$. 
Extend $\nabla$ to an operator on $C^{\infty}(\wedge T^*M; 
\cA^{*}_{(2)}(F_s \otimes E))$,  by composing it with the obvious extension of 
$\pi_0$ to $C^{\infty}(\wedge T^*M; \cA^{*}_{(2)}(F_s \otimes E))$.   The 
curvature of $\nabla \circ \pi_0$, is given by 
$(\nabla \circ \pi_0)^2 =  \nabla \circ \pi_0 \circ \nabla \circ \pi_0 = 
\nabla \circ \nabla \circ \pi_0 = \theta \circ \pi_0,$ since $\pi_0 \circ 
\nabla = \nabla$.   
We will also denote these new operators by $\nabla$ and $\theta$.   Note 
that although $\nabla$ is an operator which differentiates transversely to 
the foliation $F_s$, the operator $\theta$ is a purely leafwise operator, 
thanks to Lemma \ref{local}.    Also note that 
$$
\theta = \pi_0 \theta = \theta \pi_0.
$$
\begin{lemma}
$\theta$ is  transversely smooth.
\end{lemma}
\begin{proof}
Set  $A =  \pi_0 \nabla \pi_0 - \pi_0 \nabla^{\nu} \pi_0$, a transversely smooth operator.    Then
$$
\theta = ( \pi_0 \nabla \pi_0 )^2 = \pi_0 \nabla^{\nu} \pi_0 \nabla^{\nu} \pi_0 + 
\pi_0 \nabla^{\nu} \pi_0 A \pi_0 + \pi_0 A\pi_0 \nabla^{\nu} \pi_0 + A^2.
$$
As $A$ is transversely smooth, so is $A^2$.   Since $\pi_0 A = A \pi_0 = 
A$, the terms
$$
\pi_0 \nabla^{\nu} \pi_0 A \pi_0 + \pi_0 A\pi_0 \nabla^{\nu} \pi_0 = \pi_0 \nabla^{\nu}A 
\pi_0 + \pi_0 A \nabla^{\nu} \pi_0  =  \pi_0  [\nabla^{\nu},A] \pi_0 =  \pi_0  \pa_{\nu}(A) \pi_0, 
$$
which is transversely smooth.  Now
$\nabla^{\nu} \pi_0  = \pi_0 \nabla^{\nu} + \pa_{\nu}(\pi_0)$, so
$$ 
\pi_0 \nabla^{\nu} \pi_0 \nabla^{\nu} \pi_0 = 
\pi_0 (\nabla^{\nu})^2 \pi_0 +\pi_0 \pa_{\nu}(\pi_0)\nabla^{\nu} \pi_0 =
\pi_0 \theta^{\nu} \pi_0 +\pi_0 \pa_{\nu}(\pi_0) \pi_0 \nabla^{\nu}  +\pi_0 
\pa_{\nu}(\pi_0)  \pa_{\nu}(\pi_0).
$$ 
The curvature $\theta^{\nu} = (\nabla^{\nu})^2$ satisfies the hypothesis of of Lemma \ref{AK}.    As $\pi_0$ is transversely smooth, it follows from Lemma \ref{AK} that $\pi_0 \theta^{\nu} \pi_0$ is  transversely smooth.  Using the facts that $\pa_{\nu}$ is a derivation and $\pi_0$ is an idempotent, it is a simple exercise to show that $\pi_0 \pa_{\nu}(\pi_0) \pi_0 = 0$.   Finally, $ \pi_0 \pa_{\nu}(\pi_0)  \pa_{\nu}(\pi_0)$ is the composition of  transversely smooth operators, so transversely smooth.  Thus  $\theta$ is transversely smooth.
\end{proof}

Set 
$$
\pi_0 e^{-\theta/2i\pi} =  \pi_0 + \sum_{k=1}^{[n/2]}\frac {(-1)^k 
\theta^{k}}{(2i\pi)^k k!},
$$
and consider the Haefliger form $\dd \Tr(\pi_0 e^{-\theta/2i\pi})$.  (Note 
that $2i \pi$ is the complex number.)

\begin{theorem}\label{newchern}
The Haefliger form $\dd  \Tr(\pi_0 e^{-\theta/2i\pi})$ is closed and its 
cohomology class does not depend on the connection used to define it.
\end{theorem}

\begin{proof}
The zero-th order term of  $   \Tr(\pi_0 e^{-\theta/2i\pi})$ is $  
\Tr(\pi_0)$, and since $\pi_0$ is a uniformly bounded  leafwise smoothing operator, we have (see \cite{BHII}),
$$
d_H \Tr (\pi_0) = \Tr (\pa_\nu (\pi_0))  = {\Tr (\pa_\nu (\pi^2_0))} = 2 \Tr (\pi_0 \pa_\nu (\pi_0)) = 2 \Tr (\pi_0 \pa_\nu (\pi_0) \pi_0) = 0.
$$
since $\pi_0$ is a ($\cG$ invariant transversely smooth) idempotent.

\begin{lemma} 
\hspace{0.5cm}   For  $k > 0$, $d_H   \Tr(\theta^{k}) \,\, = \,\  0.$
\end{lemma}

\begin{proof} 

First note that for $k > 0$, 
$$
[\nabla,\theta^k] \,\, = \,\ 
[\nabla, \nabla^{2k} ]\,\, = \,\ \nabla \circ \nabla^{2k} - \nabla^{2k} \circ \nabla \,\, = \,\ 0.
$$
Also note  that $\nabla \,\, = \,\  \pi_0 \nabla^{\nu} \pi_0  + A,$
where $A$ satisfies the hypothesis of Lemma \ref{trace}, as does 
$\theta^k$.  Thus  
$$  
0 \,\, = \,\ \Tr ([\nabla,\theta^k])  \,\, = \,\ 
\Tr ([\pi_0 \nabla^{\nu} \pi_0  + A,\theta^k]) \,\, = \,\
\Tr ([\pi_0 \nabla^{\nu} \pi_0 ,\theta^k])  \,\, = \,\ 
$$
$$
\Tr (\pi_0 \nabla^{\nu} \theta^k - \theta^k  \nabla^{\nu} \pi_0) \,\, = \,\
\Tr ((\pi_0 - 1) \nabla^{\nu} \theta^k +  \nabla^{\nu} \theta^k -  
\theta^k \nabla^{\nu} - \theta^k \nabla^{\nu} (\pi_0 - 1)) \,\, = \,\ 
$$
$$
\Tr ((\pi_0 - 1) \nabla^{\nu} \theta^k ) -
\Tr (\theta^k \nabla^{\nu} (\pi_0 - 1)) +
\Tr([\nabla^{\nu} ,\theta^k]).
$$
Note that the three terms are well defined since the three operators are $\cA^*(M)$-equivariant. 
As 
$\theta = \pi_0 \theta = \theta \pi_0$,  $\theta^k = \pi_0 \theta^k 
\pi_0$, and we have 
$$
\Tr ((\pi_0 - 1) \nabla^{\nu} \theta^k) \,\, = \,\ 
\Tr ((\pi_0 - 1) \nabla^{\nu}\pi_0 \theta^k \pi_0) \,\, = \,\
\Tr ((\pi_0 - 1) \pi_0 \nabla^{\nu} \theta^k \pi_0) +
\Tr ((\pi_0 - 1)  \pa_{\nu}(\pi_0)  \, \theta^k \pi_0) \,\, = \,\ 
0,
$$
since both terms are zero.
The first term is zero because $(\pi_0 - 1) \pi_0 = 0$.  The second term 
is zero because
both $(\pi_0 - 1)  \pa_{\nu}(\pi_0)  \, \theta^k$ and $ \pi_0$ are $\cG$ invariant 
and transversely smooth, so by Lemma \ref{trace},{
$$
\Tr ((\pi_0 - 1)  \pa_{\nu}(\pi_0) \, \theta^k \pi_0)
\,\, = \,\  \Tr (\pi_0(\pi_0 - 1)  \pa_{\nu}(\pi_0)  \, \theta^k ) = 0.
$$
}
Similarly,
$$
\Tr (\theta^k \nabla^{\nu}(\pi_0 - 1)) \,\, = \,\  0.
$$
Thus, 
$$ 
0 \,\, = \,\ \Tr([\nabla^{\nu} ,\theta^k]) \,\, = \,\  \Tr (\pa_{\nu}(\theta^k)).
$$
It follows easily from Lemma 6.3 of \cite{BHII} that $d_H \Tr (\theta^k) = \Tr (\pa_{\nu}(\theta^k))$,
so we have the Lemma. 
\end{proof}

To complete the proof of Theorem \ref{newchern}, we note that a standard argument in the theory of characteristic classes shows that
\begin{lemma} 
The Haefliger class of $\dd \Tr(\pi_0 e^{-\theta/2i\pi})$ does not depend 
on the choice of connection $\nabla$ on $ \pi_0 $. 
\end{lemma}
\end{proof}

\begin{definition}
The Chern-Connes character $\ch_a(\pi_0)$ of the transversely smooth idempotent $\pi_0$ 
is the cohomology class of the Haeflliger form
 $\dd \Tr(\pi_0 e^{-\theta/2i\pi})$, that is
$$\ch_a(\pi_0)=  [ \dd \Tr(\pi_0 e^{-\theta/2i\pi})].$$
\end{definition}

\begin{remark}
In \cite{Heitsch:1995}, \cite{BHI}, and \cite{BHII} we defined Chern-Connes 
characters for various objects.  It is clear from the results of those  
papers that the definition given here is consistent with those 
definitions.   In particular, if $\nabla = \pi_0 \nabla^{\nu}$ is a 
connection on $\pi_0$ constructed from a connection $\nabla_F \otimes \nabla_E$ on $\wedge T^*F \otimes E$, 
then the material in Section 5 of \cite{BHII} (which shows that the 
definitions of \cite{Heitsch:1995} and \cite{BHI} coincide) along with the 
comment after Definition 3.11 of \cite{BHII} shows that the Chern-Connes 
character given here for $\pi_0$ and the Chern-Connes character for $\pi_0$ given 
in \cite{BHII} are the same.
Thus all three constructions of $\ch_a(\pi_0)$ yield the same Haefliger 
class.
\end{remark}
\begin{remark}
Note that in Sections \ref{cc} and \ref{ccnc} we may replace the bundle $\wedge T^*F_s \otimes E$ by any bundle on $\cG$ induced by $r$ from a bundle on $M$, and the results are still valid. 
\end{remark}

Before leaving this section, we record some facts we will need later.
In particular, we show that  any connection $\nabla$ is local in the sense that for 
$X$ transverse to $F$ and any local invariant section $\xi$ of $\pi_0$,
$\nabla_X \xi $ depends only on $\xi$ restricted to any transversal $T$ 
with $X$ tangent to $T$. {See Corollary \ref{nablalocal} below}.

\begin{lemma} \label{spans}  Let $U$ be a coordinate chart for $F$.
There is a countable collection of smooth local invariant sections of 
$\pi_0$ on $U$ which spans $C^{\infty}(\pi_0) \,|\, U$ as a module over 
$C^{\infty}(U)$.
\end{lemma}

\begin{proof}
Let $T$ be a transversal in $U$.
 The set $s^{-1}(T)$ is covered by a countable collection of coordinate 
charts of the form $(U, \gamma,V)$.  In each chart, choose a countable 
collection of smooth sections  $\{\xi^{V, \gamma}_i\}$ of $\wedge T^*F_s \otimes E$  
with support in 
$(U, \gamma,V) \cap s^{-1}(T) $ so that for any section $\xi$ of  
$\cA^*_{(2)}(F_{s}, E)$,  $\xi \,|\,  (U, \gamma,V) \cap s^{-1}(T)$ may be 
written as a linear combination (over the functions on $s(U, \gamma,V)\cap 
T$) of the $\{\xi^{V, \gamma}_i\}$.   Now extend the elements of this set 
to local invariant sections over $U$, also denoted $\{\xi^{V, \gamma}_i\}$.
The collection of sections of $C^{\infty}(\pi_0) \,|\, U$
$$
{ \maS} = \bigcup_{V,\gamma, i} \pi_0 (\xi^{V, \gamma}_i),
$$
then spans $C^{\infty}(\pi_0) \,|\, U$ as a module over $C^{\infty}(U)$,  
and the $\pi_0 (\xi^{V, \gamma}_i),$ are locally invariant sections over 
$U$.

\end{proof}

As a consequence, we deduce the following.

\begin{corollary}
If two connections $\nabla$ 
and $ \widehat \nabla$ on $\pi_0$ agree on local invariant sections, then 
they are the same.   
\end{corollary}

Note that the bundle $E = r^*E$ is flat (in fact trivial) along the leaves of the other foliation $F_r$
of $\cG$, since its leaves are just $r^{-1}(x)$ for $x \in M$.  Denote by $d_r$ the obvious differential associated to $\wedge T^*F_r \otimes E$.
Given a local section $\xi \in C^\infty_2 (\wedge T^*F_s \otimes E)$, we may view $d_r \xi$ as a local element of $C^\infty (\wedge T^*M; \wedge T^*F_s  \otimes E)$.  Note that $d_r^2\xi =0$, and $\xi$ is locally invariant if and only if $d_r \xi = 0$.  
Note that for $\xi \in C^{\infty}(\pi_0)$ and $X \in C^{\infty}(TF)$,  $\nabla_X  \xi = d_r \xi (X)$.  To see this, write  $\xi =  \sum_j g_j \xi_j$, where $\xi_j  \in \cA^k_{(2)}(F_{s},E)$ are local invariant elements, and the $g_j$ are smooth local  functions  on $M$. 
Then Conditions  $(1)$ and $(2)$ of Definition \ref{connection} give
$$
\nabla_X \xi =    \sum_j d_M g_j (X)  \xi_j  =  \sum_j 
d_F g_j (X)  \xi_j  =  \sum_j d_F g_j (X)  \xi_j + g_j d_r  \xi_j (X) = 
d_r \xi (X).
$$

Let $U$ be a foliation chart for $F$ with transversal $T$, and $\nabla$ a 
connection on $\pi_0$. Then on $U$, $\nabla$ is the pull back of  $\nabla$ 
restricted to  $\pi_0 \, | \ T$.  More specifically, for $X$ tangent to 
$T$ and $\xi \in C^{\infty}(\pi_0 \, | \ T)$, with local invariant 
extension $\widetilde{\xi}$ to $C^{\infty}(\pi_0 \,|\, U)$, define
$$
\nabla^T_X \xi \equiv \nabla_X \widetilde{\xi}.
$$ 
We may assume that  $U \simeq \R^p \times T$ with coordinates $(x,t)$ and plaques $\R^p \times t$.    
Denote by  $\rho:U \to T$ the projection.    Let $x \in U$ and $X 
\in TM_x$,  
and set $T_x = x \times T$.   Write $X = X_F + \rho_*(X)$ where $X_F \in 
TF_x$ and $\rho_*(X)$ is tangent to $T_x$.   Let
$\xi \in  C^{\infty}(\pi_0 \,|\, U)$, and define the pull back connection 
$ \rho^*(\nabla^{T} )$ by 
$$
 \rho^*(\nabla^{T} )_X  \xi = d_r \xi(X_F)  + \nabla^{T}_{\rho_*(X)} (\xi 
\,|\,T_x) = 
 d_r \xi(X_F)  + \nabla_{\rho_*(X)} \widetilde{(\xi \,|\,T_x)},
$$
and extend to $C^{\infty}(\wedge T^*U;\pi_0)$ by using $(1)$ of Definition 
\ref{connection} and the fact that 
 $C^{\infty}(\wedge T^*U;\pi_0) \simeq \cA^*(U) \otimes_{C^{\infty}(U)} 
C^{\infty}(\pi_0\, | \, U)$.

Denote the curvature $(\nabla^T)^2$ of $\nabla^T$ by 
$\theta_T.$

\begin{proposition}\label{pullbacks}
$\nabla \,|\, U =  \rho^*(\nabla^{T} )$, and 
$\theta \, | \, U= \rho^*(\theta_{T})$.
\end{proposition}

\begin{proof}   
Let $\xi \in C^{\infty}(\pi_0 \,|\, U)$ and suppose that $X \in TF$, so 
$X_F = X$ and $\rho_*(X)=0$.   Then
$$
\rho^*(\nabla^{T} )_X \xi = d_r\xi(X) =  \nabla_X \xi.
$$
Next suppose that $\xi$ is local invariant, and
$X$ is tangent to $T_x$, so $X_F = 0$ and $\rho_*(X) = 
X$.  Then
$$
\rho^*(\nabla^{T} )_X \xi  =  \nabla_{\rho_*(X)} \widetilde{(\xi 
\,|\,T_x)} = \nabla_X \xi,
$$
since $\widetilde{(\xi \,|\,T_x)}  =   \xi$, as $\xi$ is local invariant.  
Thus $\nabla \,|\, U$ and $\rho^*(\nabla^{T} )$ agree on local invariant sections, 
so they are equal.

For the second equation, writing $ \rho^*(\nabla^{T} ) = d_r  + 
\nabla^{T}$, we have
$$
\theta \xi = d^2_r \xi + 
\nabla^{T} d_r \xi  +  d_r \nabla^{T} \xi +( \nabla^{T})^2\xi =
(\nabla^{T})^2\xi,
$$
since $d^2_r = 0$ and $\nabla^{T}  \circ d_r = - d_r \circ \nabla^{T} $.  
But, with the notation 
$ \rho^*(\nabla^{T} ) = d_r  + \nabla^{T}$, $(\nabla^{T} )^2\xi = 
\rho^*(\theta_{T})\xi$.
\end{proof}

The following is immediate.
\begin{corollary}\label{nablalocal}
$\nabla$ is local in the sense that for $X$ transverse to $F$ and any 
local invariant section $\xi$ of $\pi_0$,
$\nabla_X \xi $ depends only on $\xi \, | \, T$ where $T$ is any 
transversal with $X$ tangent to it.   
\end{corollary}

\section{Leafwise maps} \label{lms}  

Let $M$ and $M'$ be compact Riemannian manifolds with oriented  foliations $F$ and $F'$.  {The results of this section do not require $F$ or $F'$ to be Riemannian.}
Let $f:M \to M'$ be a smooth leafwise homotopy equivalence which preserves the leafwise orientations.  (We need only assume transverse smoothness, and leafwise continuity.  A standard argument then allows $f$ to be approximated by a 
smooth map.)  Suppose that $E' \to M'$ is a leafwise flat complex bundle over $M'$ which satisfies the hypothesis of Theorem \ref{main}, and set $E = f^*(E')$. 
Let $g:M' \to M$ be a leafwise homotopy inverse of $f$.   Then there are 
leafwise homotopies  $h:M \times I \to M$ and $h':M' \times I \to M'$ with $I = [0,1]$, so 
that for all $x \in M, x' \in M'$
$$
h(x,0) =  x, \quad h(x,1) = g\circ f(x),\quad   h'(x',0) = x', \quad  
\text{and} \quad h'(x',1) = f\circ g(x').
$$

We begin by recalling two results on such leafwise maps from \cite{H-L:1991}.

\begin{lemma}[Lemma 3.17 of \cite{H-L:1991}]
Given finite coverings of $M$ and $M'$ by foliation charts, there is a 
number $N$ such that for 
each placque $Q$ of $M'$, there are at most $N$ plaques $P$ of $M$ such 
that $f(\overline{P}) \cap \overline{Q} \neq \emptyset$.
\end{lemma}
Thus $f$ is leafwise uniformly proper and so induces a well defined map
$f^{*}:\oH^{*}_{c}(L'_{f(x)};\R) \to \oH^{*}_{c}(L_{x};\R)$.
In general this map does not extend to the leafwise $L^{2}$ forms, as  shown by simple examples.

\begin{lemma}[Lemma 3.16 of \cite{H-L:1991}] \label{l3.16}
For any finite cover of $M$ by foliation charts there is a number $N$ such 
that for each placque $P$ of $M$, there are at most $N$ plaques $Q$ such 
that {$h(\overline{Q} \times I ) \cap \overline{P} \neq \emptyset$.}
\end{lemma}

Note that this lemma implies that there is a global bound on the leafwise 
distance that $h$ moves points, i.\ e.\  there is a global bound on the 
leafwise lengths of all the curves $\{\gamma_{x} \, | \, x \in M\}$, where 
$\gamma_{x}(t) = h(x,t)$.

\bigskip

We remark that since $f$ is a homotopy equivalence between $M$ and $M'$, the 
dimensions of $M$ and $M'$ are the same.

\begin{theorem}
$f$ induces an isomorphism $f^{*}:\oH^{*}_{c}(M'/F') \to \oH^{*}_{c}(M/F)$ 
on Haefliger cohomology with inverse $g^{*}$.
\end{theorem}

\begin{proof} The map $f$ induces a map $\what{f}$ on transversals.   In 
particular, suppose that $U$, and $U'$ are foliation charts of $M$  and 
$M'$ respectively, and that $f(U) \subset U'$.  
If $T$ and $T'$ are transversals of $U$ and $U'$, then $f$ induces the map 
$\what{f}:T \to T'$.

\begin{lemma} \label{immerse}
$\what{f}:T \to T'$ is an immersion.
\end{lemma} 

\begin{proof} Being an immersion is a local property, so by reducing the 
size of our charts if necessary, we may assume that $g(U') \subset U_{1}$ 
where $U_{1}$ is a foliation chart for $F$, with transversal $T_{1}$.  
Then $\what{g}:T' \to T_{1}$.  The leafwise homotopy $h$  induces a map 
$\what{h}:T \to T_{1}$.  In particular this is the map induced on 
transversals by the map $x \to h(x,1)$.  Since $h$ is continuous and 
leafwise, it is easy to see that $\what{h} = h_{\gamma}$ where 
$h_{\gamma}$ is the holonomy along the leafwise path $\gamma_x(t) = h( 
x,t)$, where $x \in T$.  Thus, $\what{h}$ is locally invertible.  Since $h$ 
is a homotopy of $gf$ to the 
identity, the composition, $\what{h}^{-1} \what{g} \what{f}:T \to T$ is 
the identity, so $\what{f}$ must be an immersion.
\end{proof}

Since $\what{g}$ must also be an immersion, it follows immediately that 
the codimensions of $F$ and $F'$ are the same, and so the dimensions of 
$F$ and $F'$ are also the same.

To construct the map $f^{*}:\oH^{*}_{c}(M'/F') \to \oH^{*}_{c}(M/F)$, we 
proceed as follows.  Let $\cU $ and $\cU' $ be finite good 
covers of $M$ and $M'$ respectively.  We may assume that for each $U \in \cU$, we 
have chosen a $U' \in \cU'$ so that $f(U) \subset U'$ and that the induced map on 
transversals $\what{f}:T \to T'$ is a diffeomorphism onto its image.  Let 
$\alpha' \in \oH^{*}_{c}(M'/F')$.  Since $f$ is onto, we may choose a 
Haefliger form
$\dd \phi' = \sum_{U \in \cU}\phi'_{U}$ in $\alpha'$, so that $\phi'_{U}$ 
has support in $\what{f}(T)$ where $T$ is a transversal in $U$.  We then 
define $\what{f}^{*}(\alpha')$ to be the class of the Haefliger form $\dd 
\sum_{U \in \cU} \what{f}^{*}(\phi'_{U}).$ 

The question of whether $\what{f}^{*}$ is well defined reduces to showing 
the following.

\begin{lemma}
Suppose that $U_{1}$ and $U_{2}$ are foliation charts on $M$ with 
transversals $T_{1}$ and $T_{2}$.  Suppose further that $\phi'$ is a 
Haefliger form on $M'$ with support  contained in $\what{f}(T_{1}) \cap 
\what{f}(T_{2})$.  Then as Haefliger forms on $M$, $[\what{f}\, 
|_{T_{1}}]^{*}(\phi') = 
[\what{f}\, |_{T_{2}}]^{*}(\phi')$.
\end{lemma}

\begin{proof}
Set $\what{f}_{i} = \what{f} \,\, | \,{T_{i}}$.  By writing $\phi'$ as a sum 
of Haefliger forms and 
reducing the size of their supports, we may assume that the support of 
$\phi'$ is contained in a transversal $T'$, that $\what{g}(T')$ is 
contained in a transversal $T$ of $M$ and that the holonomy maps 
$h_{i}:T_{i} \to T$ determined by the paths $\gamma_{i}(t) = h(x_{i},t)$, 
for $x_{i} \in T_{i}$ are defined on the supports of 
$\what{f}_{i}^{*}(\phi')$, respectively.  Further, we may suppose that all 
the maps $\what{f}_{1}, \what{f}_{2}, h_{1}, h_{2}$ and $\what{g}\, |_{T'}$
are diffeomorphisms onto their images.  Since $h$ is a homotopy of $gf$ to 
the identity, 
$\what{f}_{1} = \what{g}^{-1} \circ h_{1}$ and $\what{f}_{2} = 
\what{g}^{-1} \circ h_{2}$, so 
$\what{f}_{1}^{*}(\phi') = h_{1}^{*}\circ (\what{g}^{-1})^{*}(\phi')$ and 
$\what{f}_{2}^{*}(\phi') = h_{2}^{*}\circ (\what{g}^{-1})^{*}(\phi')$.
Thus, $\what{f}_{1}^{*}(\phi') = h_{1}^{*}\circ 
(h_{2}^{-1})^{*}(\what{f}_{2}^{*}(\phi'))$ so as Haefliger 
forms, $\what{f}_{1}^{*}(\phi') = \what{f}_{2}^{*}(\phi')$.
\end{proof}

It now follows easily that the induced map on Haefliger cohomology 
$f^{*}:\oH^{*}_{c}(M'/F') \to \oH^{*}_{c}(M/F)$ is an isomorphism  with 
inverse $g^{*}$.
\end{proof}

\begin{lemma} 
$f$ induces a well defined smooth leafwise map $\cf:\cG \to \cG'$, which is leafwise 
uniformly proper.
\end{lemma}
\begin{proof} {Set $\cf([\gamma]) = [f \circ \gamma]$. 
That $\cf$ is well defined and smooth is clear. Similarly, set $\cg([\gamma']) = [g \circ \gamma']$.} 

 {Let $\cU$ be a finite good cover of $M$.  Since $M$ is compact, there is a bound $m(P)$ on the diameter of any plaque in the cover $\cU$.  Then $m(P)$ is also a bound for any plaque of $F_s$ in the corresponding cover of $\cG$.  Let $\cU'$ be a finite good cover of $M'$, such that for each $U' \in \cU'$ there is $U \in \cU$ so that $g(U') \subset U$.   Given $(U', \gamma', V')$ in the cover of $\cG'$  corresponding to $\cU'$,  choose $U, V \in \cU$ with 
$g(U') \subset U$  and $g(V') \subset V$.   If we set  $\gamma = g \circ \gamma' $, then  $\cg(U', \gamma', V') \subset  (U, \gamma, V)$.
Because $\cU'$ is a good cover, there is  $\epsilon > 0$ so that if $z'_0,z'_1 \in \wL'$ with
$d_{\wL'}(z'_0,z'_1) < \epsilon$, then there is a $(U', \gamma', V')$ with  $z'_0,z'_1 \in  (U', \gamma', V')$, so    $\cg(z'_0), \cg(z'_1) \in(U, \gamma, V)$. Since $\cg(\wL') \cap (U, \gamma, V)$ consists of at most one plaque of $\cg(\wL')$,    it follows that $d_{\wL}(\cg(z'_0), \cg(z'_1)) <  m(P)$.  Thus, if $z'_t$ is a path  in $\wL'$ of length less than $C$, then $\cg \circ z'_t$ is a path in $\cg(\wL')$ of length less than $m(P) C /\epsilon$.}

Suppose that $f(x) = x'$ and let $A' \subset \wL'_{x'}$ have diameter $\dia(A') \leq C$.   Let
$z_0, z_1 \in \wL_x$ with $\cf(z_i) = z'_i \in A'$, and choose a path $z'_t$ in $\wL'_{x'}$ of length less than $C$ between $z'_0$ and $z'_1$.   Then $\cg \circ z'_t$ is a path in $\wL_{gf(x)}$ of length less than $m(P)C /\epsilon$.    Composition on the right by the path $\gamma_x(t) = h(x,t)$ is an isometry from $\wL_{gf(x)}$ to $\wL_x$.  So $(\cg \circ z'_t )\cdot \gamma_x$  is a path in $\wL_x$ of length less than $m(P) C /\epsilon$.   Thus 
$$
d_{\wL_x}([(\cg \circ z'_0 )\cdot \gamma_x],[(\cg \circ z'_1 )\cdot \gamma_x]) \leq m(P) C /\epsilon.
$$
By Lemma \ref{l3.16}, the path $\gamma_y$ has length bounded by say $B$, for all $y \in M$.   Set $y_i = r(z_i)$, and note that $[\gamma^{-1}_{y_i} \cdot (\cg \circ z'_i )\cdot \gamma_x]= z_i$, since $h$ is a leafwise homotopy equivalence between $g \circ f$ and the identity.  As 
$$
d_{\wL_x}(z_i , [(\cg \circ z'_i )\cdot \gamma_x]) =
d_{\wL_x}([\gamma^{-1}_{y_i} \cdot (\cg \circ z'_i )\cdot \gamma_x], [(\cg \circ z'_t )\cdot \gamma_x]) \leq 
\text{ length}(\gamma_{y_i}) \leq B,
$$
we have
$$
d_{\wL_x}(z_0 , z_1) \leq 2B + m(P)C/\epsilon.
$$
Thus $\dia(\cf^{-1}(A')) \leq 2B + m(P)\dia(A') /\epsilon$, and $\cf$ is leafwise uniformly proper.
\end{proof}
Thus $\cf$  induces a well defined map 
$\cf^{*}:\oH^{*}_{c}(\wL'_{f(x)};\R) \to \oH^{*}_{c}(\wL_{x};\R)$.  
As noted above, in general this map  does not induce  a well defined map on  leafwise $L^{2}$ forms.  We will use two different constructions  to deal with this 
problem.    First we adapt the construction  of the $L^2$ pull-back map of 
Hilsum-Skandalis in \cite{HilsumSkandalis} to our setting.  This has the 
advantage that it is transversely smooth.   However, it is not obvious that its action on 
leafwise $L^2$ cohomology respects the wedge product, so we 
will also use the construction in \cite{H-L:1991}, which is based on results of 
Dodziuk, \cite{Dodziuk}.  We assume the reader is familiar with Sobolov theory of spaces of 
sections of a vector bundle over a manifold.

For $s \in \Z$, denote by $W_s^*(F_s, E)$  the field of Hilbert spaces over 
$M$ given by $W_s^*(F_s, E)_x = W_s^*(\wL_x, E)$, the $s$-th Sobolev space of differential forms on $\wL_x$ with coefficients in $E \, | \, \wL_x$.    Just as it does for the leafwise $L^2$ 
forms, the compactness of $M$ implies that these spaces do not depend on 
our choice of Riemannian structure.  Note that $W_s^* \subset W_{s_1}^*$ if $s 
\geq s_1$, and set
$$
W_{\infty}^*(F_s, E) = \bigcap _{s \in \Z} W_s^*(F_s, E) \text{ and } 
W_{-\infty}^*(F_s, E) = \bigcup _{s\in \Z} W_s^*(F_s, E).
$$
Equip $W_{\infty}^* (F_s, E)$  with the induced locally convex topology.  

Let $i:M'\hookrightarrow \R^k$ be an imbedding of the compact manifold 
$M'$ in some Euclidean space $\R^k$, and identify $M'$ with its image.   
We assume for convenience that $k$ is even.  For $x'\in M'$ and $t\in 
\R^k$, define $p(x',t)$ to be the projection of the tangent vector $\dd 
X^t = \frac{d}{ds} \, | \, _{s=0}( x' +st)$ at $x'$ determined by $t$,  to 
the leaf  $L'_{x'}$ in $(M',F') \subset \R^k$.   In particular, first 
project $X^t$ to $TF'_{x'}$ and then exponentiate it to $L'_{x'}$, 
thinking of $L'_{x'}$ as a Riemannian manifold in its own right.  Since 
$M'$ is compact, we may choose a ball $B^k \subset \R^k$ so small that the 
restriction of the smooth map $p_f = p \circ (f, id):M \times B^k \to M'$ 
to any $p_f:L_x \times B^k \to L'_{f(x)}$ is a submersion.  
Lifting this map to the groupoids, we get 
$$
p_f : \maG \times B^k \longrightarrow \maG',
$$
which is a leafwise map if $\maG \times B^k$ is endowed with the foliation 
$F_s\times B^k$.   
Note that $p_f:\wL_x \times B^k \to \wL'_{f(x)}$ is the  map induced on 
the coverings by $p_f:L_x \times B^k \to L'_{f(x)}$.  In particular, 
$p_f([\gamma],t)$ is the composition of leafwise paths $P_f(\gamma,t)$ and 
$f \circ \gamma$, 
$$
p_f([\gamma],t) = [P_f(\gamma,t)\cdot  (f \circ \gamma)],
$$
where $P_f(\gamma,t): [0,1] \to L'_{f(r(\gamma))}$ is 
$$
P_f(\gamma,t)(s) = p_f(r(\gamma),st).
$$
To see that this is a smooth map,  let  $(U,\gamma,V) \times B^k$ and 
$(U', f \circ \gamma, V')$ be local coordinate charts on  $\cG \times B^k$ 
and $\cG'$, respectively, with coordinates $(w,y,z,t)$ and $(w',y',z')$.  
Then in these coordinates, 
$$
p_f(w,y,z,t) = (w'(f(w,y)),y'(f(w,y)),z'(p_f(y,z,t))),
$$
where the second $p_f$ is the map $p_f:V \times B^k \to V'$.  

The crucial fact about $p_f$ is that it has all the same essential 
properties of the projection $\pi_1:\cG \times B^k \to \cG$.  First note 
that, because $f$ {and $\cf$ are} leafwise uniformly proper and $M \times B^k$ is 
compact, both the maps denoted $p_f$ are also leafwise uniformly proper.  Second, we 
may assume that the metric on each $L_x \times B^k$ (respectively  $\wL_x 
\times B^k$) is the product of a fiberwise metric for the submersion $p_f$ 
and the pull-back under $p_f$ of the metric on $L'_{f(x)}$ (respectively 
$\wL'_{f(x)}$).  To see this, give $L \times B^k$  the product metric, 
using the standard metric on $B^k$.   The induced metric on $\wL \times 
B^k$ is then the product metric.  The  fibers of both submersions 
$p_f$ inherit a Riemannian metric, and we denote by  $dvol_{vert}$ the 
canonical $k$ form on both $L \times B^k$ and $\wL \times B^k$ whose 
restriction to  the oriented fibers of $p_f$ is the volume form.  
 Denote by $*$ the Hodge operator on both $L 
\times B^k$ and $\wL \times B^k$, and similarly for   $* '$ on $L'$ and 
$\wL'$. 
Consider the sub-bundle $p_f^* T^*F'  \subset T^*(F \times B^k)$, and its 
orthogonal complement ${p_f^* T^*F'  }^\perp$.  Define a new metric  on 
$T^*(F \times B^k) = p_f^* T^*F' \oplus {p_f^* T^*F'  }^\perp$ (and so 
also on $T^*(F_s \times B^k)$) by declaring that these sub-bundles are 
still orthogonal, and the new metric on ${p_f^* T^*F'  }^\perp$ is the 
same as the original, while the new metric on $p_f^* T^*F' $ is the 
pullback of the metric on $T^*F' $.  Denote the leafwise Hodge operator of 
the new metric by $\what{*}$.  As remarked above, this change of metric 
does not alter any of our Sobolev spaces.  In particular, note that 
for any non-zero $\alpha \in  \wedge^{\ell} T^*(F  \times B^k)$ and any  
$c \in \R^*_+$, 
$$
0 \,\, < \,\,  \frac{ c\alpha \wedge \what{*} c\alpha }{c\alpha \wedge *  
c\alpha} \,\, = \,\,   
 \frac{ \alpha \wedge \what{*} \alpha }{\alpha \wedge *  \alpha},
$$
so the compactness of the sphere bundle $(\wedge^{\ell} T^*(F \times B^k) 
-  \{0\})/ \R^*_+$ implies  that there  are $0 < C_1 <C_2$, so that for all 
$\alpha \in  \wedge^{\ell} T^*(F \times B^k)$,
$$   
C_1 \, \alpha \wedge * \alpha \,\, \leq \,\,
\alpha \wedge \what{*} \alpha  \,\, \leq \,\,   
C_2 \, \alpha \wedge * \alpha,
$$
where we identify the oriented volume elements of $L \times B^k$ at a 
point  with $\R^*_+$.
This property is inherited by the two induced metrics on $T^*(F_s \times 
B^k)$,  so the two norms used to define the Sobolev spaces $W^{ 
\ell}_s (F_s, E)$ are comparable.   Thus, we can substitute the second metric 
for the first, or what is more notationally convenient, assume that the 
first metric satisfies the same pull back property as the second.  

Simple computations give two immediate consequences of this assumption.  
Namely, for any $\alpha_1, \alpha_2 \in  \wedge^{\ell} T^* F'_s$,   
\begin{Equation} \label{eqstar1}
\hspace{1.0in}
 $p_f^*\alpha_1 \wedge * p_f^* \alpha_2  \,\, = \,\,   
 dvol_{vert}  \wedge p_f^* ( \alpha_1 \wedge * ' \alpha_2),$
\end{Equation} 
\noindent
and
\begin{Equation} \label{eqstar2}
\hspace{1.0in}
$dvol_{vert} \wedge p_f^*\alpha_1 \wedge * (dvol_{vert} \wedge 
p_f^*\alpha_2)\,\, = \,\,
dvol_{vert} \wedge p_f^*(\alpha_1 \wedge *' \alpha_2).$
\end{Equation} 
\noindent

Denote by $\pi_2:\cG \times B^k \to B^k$ the projection, and choose a 
smooth compactly supported $k$-form $\omega$ on $B^k$  whose integral is 
$1$.  We shall refer to such a form as a Bott form on $B^k$.  Denote by 
$e_\omega$ the exterior multiplication by the differential $k-$form 
$\pi_2^*\omega$ on $\cG \times B^k$.
For $\xi \in \cA_c^*(F'_s, E')$, we define $f^{(i,\omega)} (\xi) \in 
\cA_c^*(F_s, E)$ as 
$$
f^{(i,\omega)} (\xi) = (\pi_{1,*} \circ e_\omega \circ p_f^*) (\xi).
$$
The map $p_f : \maG \times B^k \longrightarrow \maG'$ is a leafwise (for $F_s \times B^k$) submersion extending $\cf$, so $p_f^* (\xi)$ is a leafwise form on $\cG \times B^k$ with coefficients in the bundle $p_f^* E'$. The map  $\pi_{1,*}$ is integration over the fiber of the projection $\pi_1:\cG \times B^k \to \cG$ of such forms.  In general,  the fiber of $p_f^* E'$ is not constant on fibers of the fibration $\pi_1:\cG \times B^k \to \cG$.  To correct for this, we use the parallel translation given by the flat structure of $p_f^* E'$ to identify all the fibers of $p_f^* E' \, | \, z \times B^k$ with $(p_f^* E')_{ (z,0)} = (\cf^* E')_z = (f^*E')_{r(z)}$.  This is well defined because the ball $B^k \subset \R^k$ is contractible, so parallel translation is independent of the path taken from $(z,0)$ to $(z,t)$ in $z \times B^k$.
\begin{proposition}\label{Smoothing}
For any $s \in \Z$, $f^{(i,\omega)}$ extends to a bounded operator from 
$W_s^*(F'_s, E')$ to $W_s^*(F_s, E)$.
\end{proposition}

\begin{proof} For this proof only, for $\alpha_1 \otimes \phi_2$ and $\alpha_2 \otimes \phi_2 \in \cA_c^*(F_s, E)$, we set 
$$
(\alpha_1\otimes \phi_1)\wedge  ( \alpha_2\otimes \phi_2) = ( \phi_1, \phi_2)\alpha_1\wedge  \alpha_2
\quad \text{and} \quad 
(\alpha_1\otimes \phi_1)\wedge * ( \alpha_2\otimes \phi_2) = ( \phi_1, \phi_2)\alpha_1\wedge * \alpha_2,
$$
where $(\cdot,\cdot)$ is the positive definite metric on $E$.  Similarly for $\cA_c^*(F'_s, E')$.

Since $p_f$ is leafwise uniformly proper, 
$$
C= \sup_{[\gamma'] \in \cG'}  \int_{p_f^{-1} ([\gamma'])}  \; dvol_{vert} 
< +\infty.
$$
Thanks to \ref{eqstar1}, we then have  for any $\alpha \otimes \phi \in 
\cA^{\ell}_c(\wL'_{f(x)}, E') = C^{\infty}_c(\wL'; \wedge^{\ell}T^*\wL'_{f(x)} \otimes E')$,
$$
\| p^*_f((\alpha \otimes \phi)_{f(x)})\|_0^2 \,\, = \,\, 
\int_{\wL_x \times B^k} (p_f^*\phi,p_f^*\phi ) p_f^*\alpha \wedge * p_f^* \alpha  \,\, = \,\,
\int_{\wL_x \times B^k} (p_f^*\phi,p_f^*\phi )dvol_{vert} \wedge  p_f^* ( \alpha \wedge * ' 
\alpha) \,\, $$
$$
\int_{ \wL'_{f(x)}} \Big{[} \int_{p_f^{-1} ([\gamma'])} dvol_{vert} 
\Big{]} (\phi,\phi ) \alpha \wedge * ' \alpha 
\,\, \leq \,\,  
C  \int_{\wL'_{f(x)}} (\phi,\phi )\alpha \wedge * ' \alpha \,\, = \,\,
C\|\alpha \otimes \phi\|_0^2.
$$
This inequality extends to all $\xi \in 
\cA^{\ell}_{(2)}(\wL'_{f(x)},E') = W^{\ell}_0(\wL'_{f(x)},E')$, so 
$p_f^*$ extends to a uniformly bounded (i.e.\ independent of $x$) operator 
from  $W^{\ell}_0(\wL'_{f(x)},E')$ to $W^{\ell}_0(\wL_x\times B^k, p_f^* E' )$, that is 
$p_f^*$ defines a bounded operator from $W^{\ell}_0(F'_s,E')$ to 
$W^{\ell}_0(F_s \times B^k, p_f^* E' )$.

Choose a sub-bundle $\what{H} \subset TF \oplus TB^k$ so that for each 
$L_x$,  it is
a horizontal distribution  for the submersion $p_f:L_x \times B^k \to 
L'_{f(x)}$.   The map
 $(r \times id)_*:TF_s \oplus TB^k \to TF \oplus  TB^k$ is an isomorphism 
on each fiber, so $\what{H}$ determines a sub-bundle $H$ of $TF_s  \oplus 
TB^k$, and  $H \, | \,  \wL_x \times B^k$ is a horizontal distribution for 
the submersion $p_f:\wL_x \times B^k \to \wL'_{f(x)}$.   Choose a finite 
collection of leafwise  vector fields $\what{Y}_1, \ldots, \what{Y}_N$ on 
$M'$ which generate $C^{\infty}(TF')$ over $C^{\infty}(M')$.  Lift these 
to leafwise (for $F'_s$) vector fields  $Y_1, \ldots, Y_N$ on $\cG'$, and 
lift these latter  to sections of $H$, denoted $X_1, \ldots , X_N$.  If  
$X^{vert}$ is a vertical vector field on $\wL \times B^k$ with respect to 
$p_f$, then $i_{X^{vert}} \circ p_f^* = 0$.  Modulo such vector fields, 
the $X_i$ generate $T\wL \oplus TB^k$ over $C^{\infty}(\wL \times B^k)$.  
In addition,   $ i_{X_{j}}\circ p_f^* = p_f^*\circ  i_{Y_{j}} $.  Thus, 
for   any $\xi \in \cA^{\ell}_c(\wL'_{f(x)},E')$, any  $Y_K 
= Y_{k_1} \wedge \cdots \wedge Y_{k_{\ell}}$, and  any $j_1, \ldots, j_m$, 
with $j_i \in \{1, \ldots, N\}$, 
\begin{eqnarray*}
\| i_{X_{j_1}}  d   \cdots   i_{X_{j_m}} d (p_f^*(\xi) (Y_K))\|_0  & \,\,  =    \,\, &
\| p_f^*(i_{Y_{j_1}} d  \cdots i_{Y_{j_m}} d(\xi(Y_K))\|_0 \\    
& \,\,    \leq   \,\, &
\sqrt{C}    \|i_{Y_{j_1}} d  \cdots  i_{Y_{j_m}} 
d(\xi(Y_K))\|_0.
\end{eqnarray*}
A classical argument then shows that for any $s \geq 1$, $p_f^*$ extends to 
a uniformly bounded operator from  $W^{\ell}_s (\wL'_{f(x)},E')$ to 
$W^{\ell}_s (\wL_x\times B^k,  p_f^*E')$, that is  a bounded operator from 
$W^{\ell}_s (F'_s, E')$ to $W^{\ell}_s (F_s \times B^k, p_f^*E')$.

The operator $e_{\omega}$ maps $W^{\ell}_s (\wL_x\times B^k,p_f^*E')$ to 
$W^{k+\ell}_s (\wL_x\times B^k,p_f^*E')$ and is uniformly bounded, since $\omega$ 
and all its derivatives are bounded.  Thus for $s \geq 0$, $e_\omega \circ p_f^*$ is a  bounded operator from 
$W^{\ell}_s (F'_s, E')$ to $W^{k+\ell}_s (F_s \times B^k, p_f^*E')$. 

For the case of $s < 0$, we dualize the argument above.   
Denote by $p_{f,*}$ integration of fiber compactly supported forms along 
the fibers of the submersion $p_f$.  We claim that for any $\alpha \in 
\cA^{k+\ell}_c(\wL_x \times B^k)$,
\begin{Equation} \label{ineqpf}
\hspace{1.0in}
$\dd p_{f,*} \alpha \wedge * ' p_{f,*} \alpha  \,\, \leq \,\,    C \, 
p_{f,*} (\alpha \wedge * \alpha),$
\end{Equation}
\noindent
where, as above,  we identify the oriented volume elements of 
$\wL'_{f(x)}$ at a point  with $\R^*_+$.   Any such $\alpha$ may be written 
as $\alpha = \alpha_1 + \alpha_2$, where $p_{f,*}(\alpha_2) = 0$, and 
$\alpha_1 =  dvol_{vert} \wedge \alpha_3$, with $\alpha_3  \in 
C^{\infty}_c(p_f^*(\wedge^{\ell}T^*\wL'_{f(x)}))$.  Then 
$$
p_{f,*} (\alpha \wedge * \alpha) \,\,  =   \,\,
p_{f,*} (\alpha_1  \wedge *  \alpha_1) \,\, + \,\,
p_{f,*} (\alpha_2 \wedge  *  \alpha_2)  \,\, + \,\,
p_{f,*} (\alpha_1 \wedge  *  \alpha_2)  \,\,+ \,\,
p_{f,*} (\alpha_2 \wedge *  \alpha_1).
$$
The last two terms are zero, since $\alpha_1 \wedge  *  \alpha_2 = 0$ as 
$dvol_{vert} \wedge * \alpha_2 = 0$, and $p_{f,*} (\alpha_2 \wedge *  
\alpha_1) = 0$ since $\alpha_2 \wedge *  \alpha_1$ does not contain 
$dvol_{vert}$.
Thus
$$
p_{f,*} (\alpha \wedge * \alpha)  \,\,=   \,\,
p_{f,*} (\alpha_1  \wedge *  \alpha_1) \,\, + \,\,
p_{f,*} (\alpha_2 \wedge  *  \alpha_2)  \,\,\geq \,\,
p_{f,*} (\alpha_1  \wedge *  \alpha_1).
$$
But, 
$$
p_{f,*} \alpha_1 \wedge * ' p_{f,*} \alpha_1 \,\, = \,\,
p_{f,*} \alpha \wedge * ' p_{f,*} \alpha,
$$
so we need only prove \ref{ineqpf} for $\alpha  = dvol_{vert} \wedge 
\alpha_3$, with  $\alpha_3  \in C^{\infty}_c(p_f^*(\wedge^{\ell}T^* 
\wL'_{f(x)}))$.

Choose a finite collection of sections $\beta_1, \ldots, \beta_r$ of  
$\wedge^{\ell}T^*F'$, so that $\beta_i \wedge *' \beta_j = 0$ if $i \neq 
j$, and the $\beta_i$ generate $C^{\infty}(\wedge^{\ell}T^*F')$ over 
$C^{\infty}(M')$.  Denote also by $\beta_i$ the lift of these sections to 
sections of $\wedge^{\ell}T^*F'_s$.  Then,  $\alpha = dvol_{vert} \wedge 
\alpha_3$, may be written as 
$$
\alpha =  \sum_i  g_i \; dvol_{vert} \wedge p_f^*\beta_i,
$$
where the $g_i$ are smooth compactly supported functions on 
$\wL_x\times B^k$.  Now,
$$
p_{f,*} \alpha \wedge * ' p_{f,*} \alpha \,\,  =  \,\, 
 \sum_{i}p_{f,*}(g_i \, dvol_{vert}) \beta_i \wedge *'  \sum_{j}  
p_{f,*}(g_j \, dvol_{vert})\beta_j  \,\, =  \,\,
 \sum_{i}[p_{f,*}(g_i \, dvol_{vert})]^2  \beta_i \wedge *'  \beta_i.
$$
Thanks to \ref{eqstar2},  
$$
p_{f,*} (\alpha \wedge * \alpha) \,\, = \,\,
p_{f,*}( \sum_{i} (g_i \;  dvol_{vert} \wedge p_f^*\beta_i) \wedge * 
\sum_{j}  (g_j \;  dvol_{vert} \wedge p_f^*\beta_j))\,\, = \,\,
$$
$$
p_{f,*}( \sum_{i,j} g_i g_j \;  dvol_{vert} \wedge 
p_f^*(\beta_i \wedge *' \beta_j) )   \,\, = \,\,
 \sum_{i} p_{f,*}( g_i^2 \, dvol_{vert}) \beta_i \wedge *' \beta_i  
\,\,  \geq  \,\, 
$$
$$
 \sum_{i} \, \frac{[p_{f,*}( g_i \cdot 1\, dvol_{vert})]^2}{ 
p_{f,*}(1 \, dvol_{vert})}\beta_i \wedge *' \beta_i  \,\, \geq \,\,
\frac{1}{C} \, \sum_{i}  [p_{f,*}( g_i \, dvol_{vert})]^2 \beta_i
\wedge *' \beta_i {\,\,= \,\, \frac{1}{C} \,  p_{f,*} \alpha \wedge * ' p_{f,*} \alpha},
$$
proving  \ref{ineqpf}.  Note that the second to last inequality is just  
Cauchy-Schwartz.  

Thus, for all $\alpha \in \cA^{k+\ell}_c(\wL_x \times B^k)$, 
$$
\|p_{f,*} \alpha \|^2_0 \,\, = \,\, \int _{\wL'_{f(x)}}  p_{f,*} \alpha 
\wedge * ' p_{f,*} \alpha  \,\, \leq \,\, 
C \int _{\wL'_{f(x)}}     p_{f,*} (\alpha \wedge * \alpha)  \,\, = \,\,
C  \int _{\wL_x \times B^k}  \alpha \wedge * \alpha   \,\, = \,\, C \,  
\|\alpha \|^2_0.
$$
Using the facts that $p_{f,*}$ commutes with the de Rham differentials, $p_{f,*} 
\circ  i_{X^{vert}}  = 0$ and $i_{Y_{j}} \circ p_{f,*} =p_{f,*} \circ 
i_{X_{j}} $, it is easy to deduce, just as for $p_f^*$, that for any $s  
\geq 0$, $p_{f,*}\circ e_\omega$ extends to a uniformly bounded operator 
(say with bound $C_s$) from $W^{\ell}_s (\wL_{x} \times B^k, p^*_f E')$ to 
$W^{\ell}_s (\wL'_{f(x)}, E')$.  Now suppose that $\xi' \in 
W^{\ell}_s (\wL'_{f(x)},E')$ for some $s <  0$, and recall that $\|(e_\omega 
\circ p_f^*)(\xi')\|_s$ is given by 
$$
\|(e_\omega \circ p_f^*)(\xi')\|_s  \,\, = \,\, 
\sup_{\xi} \frac{ | <\xi', (p_{f,*}\circ e_\omega)(\xi)> )|}{\| \xi \|_{-s}} \,\, \leq  \,\,  
\sup_{\xi} \frac{ \|\xi'\|_s  \| (p_{f,*}\circ 
e_\omega)(\xi)\|_{-s}}{\| \xi \|_{-s}}  \,\, \leq \,\, 
C_s \|\xi'\|_s,
$$
where the supremums are taken over all $\xi \in W^{\ell}_{-s}(\wL_x\times B^k, p^*_f E')$.
Thus, for any $s < 0$ (and so for all $s \in Z$), $e_\omega \circ p_f^*$ 
is a uniformly bounded operator from $W^{\ell}_s (\wL'_{f(x)},E')$ to 
$W^{k+\ell}_s (\wL_{x} \times B^k, p^*_f E')$, so $e_\omega \circ p_f^*$ is  a 
bounded operator from $W^{\ell}_s (F'_s, E')$ to $W^{\ell}_s (F_s \times B^k, p^*_f E')$.

For all $s \in \Z$, the image of $e_\omega \circ p_f^*$ consists of
$\pi_1$-fiber compactly supported distributional forms. 
The argument above for  $p_{f,*}$ applied to $\pi_{1,*}$ shows that it is 
uniformly bounded as a map from $\Im(e_\omega \circ p_f^*) \subset 
W^{k+\ell}_s (\wL_x\times B^k, p^*_f E')$ to $W^{\ell}_s (\wL_x,E)$.  Thus, for all
$s \in \Z$,  $f^{(i,\omega)}$ extends to a bounded operator from 
$W^{\ell}_s (F'_s,E')$ to $W^{\ell}_s (F_s,E)$.
\end{proof}

 As $\omega$ is closed, $e_\omega$ commutes with de Rham differentials.   
The image of $e_\omega \circ p_f^*$  is contained in the $\pi_1$-fiber 
compactly supported forms, so $f^{(i,\omega)} = \pi_{1,*} \circ  e_\omega 
\circ p_f^*$ commutes with de Rham differentials.  It follows immediately 
that the extension of $f^{(i,\omega)}$ to the $L^2$ forms also commutes 
with the closures of the de Rham differentials, so $f^{(i,\omega)}$ 
induces a well defined map $\wf^*: \oH_{(2)}^*(F'_s,E') \longrightarrow 
\oH_{(2)}^*(F_s,E)$ on leafwise reduced $L^2$ cohomology.    As remarked 
above, the properties of this map (using this definition) are not 
immediately obvious.  To deal with this problem, we now switch our point 
of view to that in \cite{H-L:1991}, and give another construction of the 
map $\wf^*$.

Let $\dd K = \bigcup_{\wL} K_{\wL}$ be a bounded leafwise triangulation of 
$F_s$, (see \cite{H-L:1991}) induced from a bounded leafwise triangulation 
to $F$.    Then $K_{\wL}$ is a bounded triangulation of the leaf $\wL$.  A simplicial $k$-cochain $\varphi$ on $K_{\wL}$ with coefficients in $E$ assigns to each $k$-simplex $\sigma$ of $K_{\wL}$ an element $\varphi(\sigma) \in E_{\sigma}$, the fiber of $E$ over the barycenter of $\sigma$. 
To define the co-boundary map $\delta$, we identify $E_{\sigma}$ with the fibers of $E$ over the barycenters of the simplices in the boundary of $\sigma$ using the flat structure of $E$.  This is well defined since $\sigma$ is contractible.  Denote by $\maC_{(p)}^k (K_{\wL}, E)$ the space of  simplicial 
$k$-cochains $\varphi$ on $K_{\wL}$ with coefficients in $E$  such that 
$$
\sum_{\sigma \text{ $k$-simplex of }K_{\wL} } (\varphi (\sigma),\varphi (\sigma)) ^{p/2} < 
+\infty.
$$
The homology of the complex $(\maC_{(p)}^*(K_{\wL},E) , \delta)$ is the $\ell^p$  cohomology of the 
simplicial complex $K_{\wL}$ with coefficients in $E$.   It is denoted $\oH^*_{\tri,p}(\wL,E)$.  The classical Whitney and de Rham  maps extend to well defined chain morphisms
$$
W: \maC^*_{(p)} (K_{\wL},E) \rightarrow \cA^*_{(p)} (\wL,E) \text{ and }  
\oint: \cA^*_{(p)} (\wL,E) \rightarrow \maC^*_{(p)} (K_{\wL},E),
$$
which induce bounded isomorphisms in cohomology (which are inverses of 
each other),  with bounds independent of $\wL$, for $p = 1, 2$.  See   \cite{H-L:1991} for $p = 2$, and  \cite{Goldstein} for $p = 1$.  As above, to define these maps, we use the classical definitions coupled with the fact that for any point $x \in \sigma$, the flat structure of $E \, | \, \sigma$ gives a natural isomorphism between $E_x$ and $E_{\sigma}$. 

Let $f_{K,K'}:K_{\wL} \to K'_{\wL'}$ be an oriented leafwise simplicial approximation  of $\check{f}$ 
as in \cite{H-L:1991}.  It is uniformly proper, so  it defines a pull-back map $f^*_\tri$ on $\ell^p$ cochains with coefficients in $E'$, which commutes with the coboundaries.   The induced map on  cohomology is also denoted $f^*_\tri$.   Set $f^*_D = W \circ f_\tri^* \circ \oint$
\begin{proposition} \quad 
$
\dd \wf^* = f^*_D:\oH^*_{(2)} (F'_s,E') \longrightarrow \oH^*_{(2)}(F_s,E).
$
\end{proposition}

\begin{proof}
As $B^k$ is a finite CW-complex, the map $p_f$ induces the well defined 
map 
$$
p_{f,\tri}^* : \oH_{\Delta, 2}^*(\wL',E') \rightarrow \oH_{\tri, 2}^*(\wL\times B^k, p_f^*E').
$$
Denote by $\beta$ the  simplicial $k$ cocycle $ \oint \omega$ on $B^k$, 
and  by $\pi_2:\wL\times B^k \to B^k$ a simplicial approximation (after 
suitable subdivisions) of the projection.  We choose the subdivision fine enough so that the cup product by the bounded $k$ cocycle $\pi_2^* \beta$ induces the well defined map
$$
[\beta]\cup : \oH_{\Delta, 2}^*(\wL\times B^k, p_f^*E')\rightarrow \oH_{\tri,2,c}^{*+k}(\wL\times B^k, p_f^*E'),
$$
where $\oH_{\Delta, 2,c}^{*}(\wL\times B^k, p_f^*E')$ denotes the $\ell^2$  
simplicial cohomology of cochains which are zero on any simplex that 
intersects the boundary of $\wL\times B^k$, that is ``fiber compactly 
supported" cocycles.  Cap product with the fundamental cycle $[B^k]$ of 
$B^k$ gives the map
$$
\cap [B^k]: \oH_{\Delta, 2,c}^{*+k}(\wL\times B^k, p_f^*E') \rightarrow \oH_{\tri, 2}^*(\wL,E).
$$
Denote by $\oH^{*}_{(2),c} (\wL \times B^k, p_f^*E')$  the cohomology of $L^2$ 
forms which are zero on some neighborhood of the boundary $\wL \times 
B^k$.  Note that $\oH^*_{\tri,2}(\wL \times B^k, p_f^*E')$ is a module over $\oH^*_{\tri,2}(\wL \times B^k)$,  $\oH^*_{(2)} (\wL \times B^k, p_f^*E')$ is a module over $\oH^*_{(2)} (\wL \times B^k)$, and  
$\cap [B^k]: \oH_{\Delta, 2,c}^{*+k}(\wL\times B^k, p_f^*E') \rightarrow \oH_{\tri, 2}^*(\wL,E)$ is defined.  
Then, the following diagram commutes.

\hspace{0.25in}
\begin{picture}(415,80)
\put(0,60){$\oH^*_{\Delta, 2}(\wL',E')$} 
\put(10,50){$\vector(0,-1){20}$}
\put(0,10){$\oH^*_{(2)}(\wL',E')$} 
\put(15,35){$W$}

\put(65,70){$ p_{f,\tri}^*$}
\put(55,64){\vector(1,0){35}}
\put(55,13){\vector(1,0){35}}
\put(65,19){$p_f^*$}

\put(95,60){$\oH^*_{\Delta, 2} (\wL \times B^k, p_f^*E')$}
\put(125,50){ $\vector(0,-1){20}$}
\put(95,10){$\oH^*_{(2)} (\wL\times B^k, p_f^*E')$}
\put(135,35){$W$}

\put(195,70){$[\beta]\cup $}
\put(185,64){\vector(1,0){40}}
\put(185,13){\vector(1,0){40}}
\put(195,19){$[\omega] \wedge$}

\put(230,60){$\oH^{*+k}_{\Delta, 2,c} (\wL \times B^k, p_f^*E')$}
\put(250,50){ $\vector(0,-1){20}$}
\put(230,10){$\oH^{*+k}_{(2),c} (\wL \times B^k, p_f^*E')$}
\put(260,35){$W$}

\put(330,70){$\cap [B^k]$}
\put(325,64){\vector(1,0){40}}
\put(325,13){\vector(1,0){40}}
\put(335,19){$\pi_{1,*}$}

\put(370,60){$\oH^*_{\Delta, 2}(\wL,E)$}
\put(380,50){ $\vector(0,-1){20}$}
\put(370,10){$\oH^*_{(2)}(\wL,E).$}
\put(390,35){$W$}
\end{picture}\\
Since $p_f$ is a smooth submersion, it defines the bounded operator 
$p_f^*:H^*_{(2)}(\wL',E') \to \oH^*_{(2)} (\wL \times B^k,  p_f^*E')$, and 
$W \circ p_{f,\tri}^* = p_f^* \circ W$ by the naturality of the Whitney 
map.
The square in the middle commutes because $W$ is compatible with cup and 
wedge products in cohomology and $W[\beta] = [\omega]$. Finally the RHS 
square is commutative because $W$ is compatible with cap products, and  
integration over the fibers of $\pi_1$ is exactly cap product by the 
fundamental class  in homology of $B^k$.

The bottom line of this diagram is $\wf^*$, so we need only show that 
$$
W \circ \cap [B^k] \circ [\beta]\cup \circ \, p_{f,\tri}^* \circ W^{-1} \,\, = \,\,
f^*_D  \,\, = \,\,  W \circ f_\tri^* \circ \oint.
$$
As $ W^{-1} = \oint$, this reduces to showing that 
$$
 \cap [B^k] \circ [\beta]\cup \circ  \, p_{f,\tri}^*  \,\, = \,\, f_\tri^*.
$$
The zero section $i:\wL\hookrightarrow \wL\times B^k$ induces 
$$
i_\tri^*: \oH^*_{\Delta, 2} (\wL\times B^k, p_f^*E') \rightarrow \oH^*_{\tri, 2}(\wL,E),
$$
and the projection $\pi_1: \wL \times B^k \to \wL$ induces 
$$
\pi_{1,\tri}^*: \oH^*_{\Delta, 2} (\wL,E) \rightarrow \oH^*_{\Delta, 2} 
(\wL\times B^k, p_f^*E').
$$
These maps satisfy
$$
 \pi_{1,\tri}^*\circ i_\tri^*  \,\, = \,\,  id_{\oH^*_{\Delta, 2} 
(\wL\times B^k, p_f^*E')}.
$$
Thus we have
$$
([\beta]\cup ) \circ p_{f,\tri}^* \,\,  = \,\, ( [\beta]\cup)  \circ \, 
\pi_{1,\tri}^*\circ i_\tri^* \circ p_{f,\tri}^*   \,\,  = \,\,  
([\beta]\cup) \circ  \,   \pi_{1,\tri}^* \circ f_\tri^*.
$$
By the Thom Isomorphism Theorem, $([\beta]\cup)  \circ  \, 
\pi_{1,\tri}^*:\oH_{\Delta, 2}^*(\wL,E) \to\oH^{*+k}_{\Delta, 2,c} (\wL\times 
B^k, p_f^*E')$ is an isomorphism  whose inverse is precisely $\cap [B^k]$.
\end{proof}

\begin{corollary}\label{isos} The map $\wf^*: \oH_{(2)}^*(F'_s,E') 
\longrightarrow \oH_{(2)}^*(F_s,E)$  on leafwise reduced $L^2$ cohomology 
induced by $f^{(i,\omega)}$ does not depend on the choices of $i$ and 
$\omega$.  If $f_1$ and $f_2$ are leafwise homotopy equivalent, then 
$\wf_1^* = \wf_2^*$.   If $g: (M', F')\to (M,F)$ is a leafwise homotopy 
inverse for $f$, then $\wg^*\circ \wf^* = id \text{ and } \wf^* \circ \wg^* = id$,
so $\wf^*$ is an isomorphism, with inverse $\wg^*$.
\end{corollary}

\begin{proof}
For any choice of $i$ and $\omega$, $\wf^* = f^*_D$, so they are all the 
same.  The other properties of $\wf^*$ follow from these same properties 
for $f^*_D$ which are easy to prove using classical arguments.
\end{proof}

The following result will be needed for the proof of the main theorem.  Recall the definition of the pairing $Q$ from the proof of Lemma \ref{defQ}.
\begin{proposition}\label{preserves}
If $\xi'_1$ and $\xi'_2$ are closed $L^2$ sections of 
$\wedge^{\ell} \wL'_{f(x)} \otimes E'$, then
$$
Q_x( \wf^*(\xi'_1), \wf^*(\xi'_2)) = 
Q'_{f(x)}(\xi'_1, \xi'_2).
$$
\end{proposition}

\begin{proof}
For this, we need the cup product for simplicial cochains with coefficients in $E$ (and $E'$). Note that since $E$ has two (possibly) different metrics on it, we have two (possibly) different ways of defining this cup product, depending on which metric we use.  We will use the (possibly indefinite) metric $\{ \cdot, \cdot \}$.  The definition we want to extend is that of \cite{LuckSchick}, Equation (3.30).  For ordinary degree  $\ell$ cochains $\varphi_1$ and $\varphi_2$, this is
$$
(\varphi_1 \cup \varphi_2)(\sigma) = \frac{1}{(2\ell+1)!}\sum_{i,j}  \varphi_1 (\sigma_i) \varphi_2 (\sigma_j),
$$
where $\sigma_i$ and $\sigma_j$ are certain faces of the $2 \ell$ simplex $\sigma$,  and $\varphi_1 (\sigma_i)$ and $ \varphi_2 (\sigma_j)$ are real numbers. If $\varphi_1$ and $\varphi_2$ are cochains with coefficients in $E$, then $\varphi_1 (\sigma_i)$ and $ \varphi_2 (\sigma_j)$ are elements of $E_{\sigma_i}$ and $E_{\sigma_j}$ respectively (which we identify with $E_{\sigma}$), and their cup product is an ordinary ($\C$ valued) cochain which is given by the formula
$$
(\varphi_1 \cup \varphi_2)(\sigma) = \frac{1}{(2\ell+1)!} \sum_{i,j} \{ \varphi_1 (\sigma_i), \varphi_2 (\sigma_j)\}.
$$
In \cite{LuckSchick}[3.30] L\"{u}ck and Schick  show that for their definition of the cup product, the Whitney map satisfies
\begin{proposition}
For any $\ell^2$ simplicial cochains $\varphi_1$ and $\varphi_2$ on $\wL$ with coefficients in $E$,
$$
Q(W(\varphi_1 ), W(\varphi_2)) \quad =  \quad \int_{\wL} W(\varphi_1  \cup \varphi_2).
$$
\end{proposition}
Actually, they prove it when $E$ is the one dimensional trivial bundle.  The proof extends immediately to our case, since it is a local statement, and locally $E$ is trivial with the metric the pull-back from the metric on a single fiber.

The reason we use the metric $\{ \cdot, \cdot \}$ in the cup product, and not the metric $( \cdot, \cdot )$, is so that this result will pass to simplicial $\ell^2$ cohomology classes $\Xi_1$ and $\Xi_2$.  
In particular, we have
$$
Q(W(\Xi_1 ), W(\Xi_2))  \,\,  =  \,\,   \int_{\wL} 
W(\Xi_1 \sqcup \Xi_2)   \,\,  =  \,\,  <[\wL], \Xi_1\sqcup \Xi_2>
$$
where $\sqcup$ is {\underline{the}} cup product of $\ell^2$ cohomology 
classes with coefficients in $E$, which takes values in the usual  $\ell^1$ cohomology (no coefficients), and $[\wL]$ is the fundamental 
class in bounded simplicial homology.  As $\oint$ and $W$ are inverses 
of each other on cohomology, we immediately have for any $L^2$ cohomology 
classes $\Psi'_1$ and $\Psi'_2$ on $\wL'$ with coefficients in $E'$, 
$$
<[\wL'], (\oint \Psi'_1) \sqcup (\oint \Psi'_2)>\,\,  =  \,\,  
\int_{\wL'} \Psi'_1 \wedge \Psi'_2.
$$
It is clear from the definitions of the cup product and of $f_\tri^*$ that, 
for any classes $\Xi'_1, \Xi'_2 \in H^*_{ \Delta, 2}(\wL',E')$,  the 
following equality  holds in $H^*_{\Delta, 1}(\wL,E)$,
$$
f_\tri^* \Xi'_1 \sqcup f_\tri^* \Xi'_2  \,\, = \,\,  f_\tri^* 
(\Xi'_1 \sqcup \Xi'_2).
$$

We need only prove the proposition for $f^*_D$.    Recall  that  if $\xi'_1  = \alpha'_1\otimes \phi'_1$ and $\xi'_2 = \alpha'_2 \otimes \phi'_2$, then $\xi'_1 \wedge \xi'_2 = \{ \phi'_1, \phi'_2\} \alpha'_1  \wedge \alpha'_2$, and we extend to all $\xi'_1$ and $\xi'_2$ by linearity. 
Let $\Psi'_1$ and $\Psi'_2$ be the cohomology classes determined by $\xi'_1$ and $\xi'_2$.   Then
$$
Q_x( \wf^*(\xi'_1), \wf^*(\xi'_2)) \,\,  =  \, \, 
\int_{\wL_x}  \wf^*_{D}(\xi'_1) \wedge \wf^*_{D}(\xi'_2)  \,\,  =  \,\, 
\int_{\wL_x}  \wf^*_{D}(\Psi'_1) \wedge \wf^*_{D}(\Psi'_2)  \,\,  =  \,\, 
$$
$$\int_{\wL_x} ( W \circ f_\tri^* \circ \oint \Psi'_1) \wedge (W \circ 
f_\tri^* \circ \oint\Psi'_2)   \,\,  =  \,\, 
\int_{\wL_x}  W((f_\tri^* \circ \oint\Psi'_1) \sqcup  (f_\tri^* \circ \oint \Psi'_2))  \,\,  =  \,\,  
$$
$$
<[\wL], (f_\tri^* \circ \oint \Psi'_1) \sqcup  (f_\tri^* \circ \oint \Psi'_2)>   \,\,  =  \,\, 
<[\wL], f_\tri^* (\oint \Psi'_1 \sqcup  \oint \Psi'_2)>   \,\,  =  \,\, 
<[f_{\tri, *}\wL],  (\oint \Psi'_1 \sqcup  \oint \Psi'_2)>   \,\,  =  \,\,
$$
$$
 <[\wL'], \oint \Psi'_1 \sqcup  \oint \Psi'_2>   \,\,  =  \,\, 
\int_{\wL'_{f(x)}}  \Psi'_1 \wedge \Psi'_2  \,\,  =  \,\, 
\int_{\wL'_{f(x)}}  \xi'_1 \wedge \xi'_2 \,\,  =  \,\, 
Q'_{f(x)}(\xi'_1, \xi'_2).
$$
\end{proof}

\section{Induced bundles} \label{indbdls} 

We assume again that $F$ and $F'$ are Riemannian foliations, and in this section take $\wf^*$ to be 
$$
\wf^* = f^{(i,\omega)} = \pi_{1,*} \circ e_\omega \circ 
p_f^*:W^*_{-\infty}(F',E') \to W^*_{-\infty}(F,E).
$$
The restriction of $\wf^*$ gives isomorphisms from 
$\Ker ({\Delta}^{E'}_{\ell})$, $\Ker ({\Delta}^{E'+}_{\ell})$, and $\Ker ({\Delta}^{E'-}_{\ell})$ to their images  which we denote by 
$$
\Im {\wf^*} = \wf^*(\Ker ({\Delta}^{E'}_{\ell})), \quad \Im {\wf^*_+} = \wf^*(\Ker ({\Delta}^{E'+}_{\ell})), \quad  \text{and} \quad  \Im {\wf^*_-} = \wf^*(\Ker ({\Delta}^{E'-}_{\ell})), 
$$
respectively.  We use similar notation for the map $\wg^*:W^*_{-\infty}(F,E) \to W^*_{-\infty}(F',E')$.

Note that for $x \in M$, $gf(x) \neq x$ in general, which creates technical problems. 
To deal with this, choose a leafwise homotopy equivalence $h:M \times I \to M$ between the identity map on $M$ and $gf$.   Recall the smooth leafwise path  $\gamma_x$ from $x$ to $gf(x)$, given by $\gamma_x (t) = h(x,t)$. It determines the isometry
$R_{x}:\wL_{gf(x)} \to \wL_x$, given by $R_{x}([\gamma]) = [\gamma \cdot \gamma_x]$. 
For any Sobolev space $W^*_s (\wL_x, E)$,  $R_{x}$ determines the isometry
$$
R^*_x: W^*_s (\wL_x, E)  \to W^*_s (\wL_{gf(x)}, E).
$$
In particular for $s = 0$, it gives the isometry, 
$$
R^*_x: L^2(\wL_x; \wedge T^*F_s \otimes E)  \to L^2(\wL_{gf(x)}; \wedge T^*F_s \otimes E).
$$
We shall also consider the smooth leafwise paths $\gamma'_{x'}$  from $x'\in M'$ to $fg(x')$ given by $\gamma'_{x'} (t) = h'(x',t)$ where $h'$ is a fixed leafwise homotopy between the identity of $M'$ and $fg$.   Given $x \in M$, define the isometry $R'_x:\wL'_{f(x)} \to \wL'_{f(x)}$ to be
$$
R'_x[\gamma'] = [\gamma' \cdot f(\gamma_x)^{-1} \cdot \gamma'_{f(x)}].
$$
This induces the isometry
$$
R'^*:L^2(\wL'_{f(x)}; \wedge T^*F'_s \otimes E')  \to L^2(\wL'_{f(x)}; \wedge T^*F'_s \otimes E').
$$
Note that the composition 
$$
 R'_x   \circ  \cf  \circ  R_x  \circ  \cg:\wL'_{f(x)} \to \wL'_{f(x)} 
$$
is homotopic to the identity map,  since for $[\gamma'] \in \wL'_{f(x)}$, 
$$
R'_x   \circ  \cf  \circ  R_x  \circ  \cg ([\gamma'] ) \,\,= \,\,  
[fg(\gamma') \cdot f(\gamma_x) \cdot f(\gamma_x)^{-1} \cdot \gamma'_{f(x)}]  \,\,= \,\,  
[fg(\gamma') \cdot \gamma'_{f(x)}].
$$
Set 
$$
L^t(\gamma') = ({\gamma'}_{r(\gamma')}^{-1} \, | \, _{[0,t]}) \cdot fg(\gamma') \cdot \gamma'_{f(x)}.
$$
Then  $L^0(\gamma') = fg(\gamma') \cdot \gamma'_{f(x)}$, and  $L^1(\gamma') =  {\gamma'}_{r(\gamma')}^{-1} \cdot  fg(\gamma') \cdot \gamma'_{f(x)}$.
Now $s(L^1(\gamma')) = s(\gamma')$ and $r(L^1(\gamma')) = r(\gamma')$, and $h'$ provides a leafwise homotopy between $L^1(\gamma')$ and $\gamma'$, so they define the same element in 
$\wL'_{f(x)}$.  Thus $L^t$ induces a homotopy from $R'_x   \circ  \cf  \circ  R_x  \circ  \cg$ to the identity map.
For $x \in M$,  consider the composition 
$$
(P'_{\ell}\wg^* R^*_x P_{\ell}\wf^* R'^*_x  P'_{\ell})_{f(x)}: 
L^2(\wL'_{f(x)}; \wedge T^*F'_s \otimes E')  \to L^2(\wL'_{f(x)}; \wedge T^*F'_s \otimes E').
$$
Since $R'_x   \circ  \cf  \circ  R_x  \circ  \cg:\wL'_{f(x)} \to \wL'_{f(x)}$
is homotopic to the identity and $P^*_{\ell}$ is the identity on cohomology, it follows  that
$(P'_{\ell}\wg^* R^*_x P_{\ell}\wf^* R'^*_x  P'_{\ell})_{f(x)}$ induces the 
identity on cohomology, which is naturally isomorphic to 
$\Ker(\Delta^{E'}_{\ell})_{f(x)} = \Im(P'_{\ell})_{f(x)}$.   So its 
restriction 
$$
(P'_{\ell}\wg^* R^*_x P_{\ell}\wf^* R'^*_x  P'_{\ell})_{f(x)}:\Ker(\Delta^{E'}_{\ell})_{f(x)} \to  
\Ker(\Delta^{E'}_{\ell})_{f(x)}
$$
is the identity.

We now show that  $\Im {\wf^*_+} $ determines a smooth subbundle of 
$\cA^{\ell}_{(2)}(F_s,E)$ over $M/F$.  
Set 
$$
\pi^f_+ = \wf^*  R'^*  \pi'_+ \wg^* R^* P_{\ell}.
$$
Then for each $x \in M$, 
$$
(\pi^f_+)_x : L^2(\wL_x; \wedge T^*F_s \otimes E)  \to L^2(\wL_x; \wedge T^*F_s \otimes E)
$$
is bounded and leafwise smoothing since $ \pi'_+$ and  $P_{\ell}$ 
are, and  $R'^*_x$, $R^*_x$,  $\wf^*$ and $\wg^*$ are bounded maps.  We leave it to 
the reader to show that $\pi^f_+ $ is $\cG$ invariant using 
the equality
$$
[gf(\gamma) \cdot \gamma_x] = [\gamma_y\cdot \gamma]
$$
for any $\gamma \in \cG$ with $s(\gamma) = x$ and $r(\gamma)=y$.  
As above, this equality holds since the two paths start and end at the same points and a leafwise homotopy between them can be constructed using the leafwise homotopy equivalence $h$.

We extend $\pi^f_+$ to an $\cA^*(M)$ equivariant operator on  $\wedge \nu^*_s 
\otimes \wedge T^* F_s \otimes E$ in the usual way.

\begin{proposition}\label{piftrs}
$\pi^f_+: \cA^{\ell}_{(2)}(F_s, E) \to \Im {\wf^*_+} $ is a transversely 
smooth idempotent.
\end{proposition}

\begin{proof}
First we have,
$$
(\pi^f_+)^2 =  \wf^*  R'^*  \pi'_+ \wg^* R^* P_{\ell} \wf^*  R'^*  \pi'_+ \wg^* R^* P_{\ell} =
\wf^*  R'^*  \pi'_+  P'_{\ell} \wg^* R^* P_{\ell} \wf^*  R'^*   P'_{\ell} \pi'_+ \wg^* R^* P_{\ell} = 
$$

$$
\wf^*  R'^* ( \pi'_+)^2 \wg^* R^* P_{\ell}  = \wf^*  R'^*  \pi'_+  \wg^*R^* P_{\ell}  =  
\pi^f_+,
$$
since $\pi'_+  =\pi'_+ P'_{\ell} = P'_{\ell}\pi'_+$, and for each $x \in 
M$,
$(P'_{\ell}\wg^* R^* P_{\ell}\wf^* R'^*  P'_{\ell})_{f(x)}: 
\Ker(\Delta^{E'}_{\ell})_{f(x)} \to \Ker(\Delta^{E'}_{\ell})_{f(x)}$
is the identity map, and $\Ker(\Delta^{E'}_{\ell}) \supset  \Im(\pi'_+)$.  

As $P_{\ell}$ is transversely smooth, we need only show that $\wf^* R'^*  \pi'_+ 
\wg^* R^*$ is transversely smooth.  

Let $\nabla_{E}$ and  $\nabla_{E'}$ be  the leafwise flat connections on $E$ and  $E'$ and $\nabla_{F'}$ and $\nabla_F$ be the Riemannian connections on $T^*F'$ and $T^*F$, respectively.
Denote by $\nabla^{\nu}$ and  $\nabla'^{\nu}$ the quasi-connections on 
$C^{\infty}(\wedge \nu^*_s \otimes \wedge T^*F_s \otimes E)$ and 
$C^{\infty}(\wedge \nu'^{*}_s \otimes \wedge  T^*F'_s \otimes E')$ constructed from  $\nabla_F \otimes \nabla_E$, and $\nabla_{F'} \otimes \nabla_{E'}$,  respectively.

Now suppose $H$ is any $\cG$ invariant operator of degree zero on $\wedge T^*F_s \otimes E$, e.g.\ $H = \wf^* R'^* 
\pi'_+ \wg^* R^*$.  If $X \in C^{\infty}(TF)$, then since $H$ and 
$\nabla^{\nu}$ are $\cG$ invariant, $\pa_{\nu}^X(H) = 0$.    A vector 
field $Y$ on $M$ is a $\Gamma$ vector field provided that for any $X \in 
C^{\infty}(TF)$, $[X,Y] \in C^{\infty}(TF)$.   If $Y \in C^{\infty}(\nu)$ 
is a $\Gamma$ vector field, it is invariant under the parallel translation 
defined by $F$, so $\pa_{\nu}^Y(H)$ is $\cG$ invariant.  Globally defined 
$\Gamma$ vector fields rarely exist.  The restriction of a global vector 
field to an open subset will be called a local extendable vector field.  
Such local vector fields have all their derivatives  bounded.    Any local 
$\Gamma$ vector field may, after a slight reduction in its domain of 
definition,  be extended to a global vector field.   Finally, a bounded function (on $M$) times a bounded leafwise smoothing operator yields a bounded leafwise smoothing  operator.
With this in mind, the problem of showing that such an $H$ is transversely smooth may be recast as follows (with the proof left to the reader).
\begin{lemma}\label{lgam} Suppose $H$ is a degree zero  $\cG$ invariant $\cA^*(M)$ equivariant (homogeneous of degree $0$) bounded leafwise smoothing operator on $\wedge \nu^*_s \otimes \wedge T^* F_s \otimes E$.  Then $H$ is transversely smooth if and only if for all local extendable $\Gamma$ vector fields $Y_1,...,Y_{m} \in C^{\infty}(\nu)$, the operator $\pa^{Y_1}_{\nu}  ...\pa^{Y_{m}}_{\nu}(H)$ is a bounded leafwise smoothing  operator on $ \wedge T^* F_s\otimes E$.  
\end{lemma}

Note that  the expression  $\wf^* R'^* \pi'_+ \wg^* R^*\nabla^{\nu}$ makes sense as 
$\wf^*R'^* \pi'_+ \wg^* R^*$ is a well defined  $\cA^*(M)$ equivariant 
operator on  $\wedge \nu^*_s \otimes \wedge T^* F_s \otimes E$.  Note further that the 
expression $R^*\nabla^{\nu}$ does not make sense in general.  However, restricted to any sufficiently small transverse submanifold, $gf$ 
is a diffeomorphism onto its image, so $(gf)^{-1}$ is well defined on this 
image.   This makes it possible to prove the following.

\begin{lemma}\label{Rdnu}  Suppose $Y \in \nu_x$, then $\wf^* R'^* \pi'_+ \wg^* R^*\nabla^{\nu}_Y   = \wf^*R'^*  \pi'_+ \wg^* \nabla^{\nu}_{h_*(Y)} R^*$, where $h_*(Y) \in \nu_{gf(x)}$ is the parallel translate of $Y$ along $\gamma_x$.

If  $Y' \in \nu'_{f(x)}$, then $\wf^* R'^*\nabla'^{\nu}_{Y'} \pi'_+ \wg^* R^*   = \wf^* \nabla'^{\nu}_{h'_*(Y')} R'^*  \pi'_+ \wg^* R^*$, where $h'_*(Y') \in \nu'_{f(x)}$ is the parallel translate of $Y'$ along $f(\gamma_x)^{-1} \cdot \gamma'_{f(x)}$.

\end{lemma}
 
\begin{proof}
Let  $(U_x, \gamma, V)$ be a local chart containing $[\gamma] \in \wL_x$, and 
$(U_{gf(x)}, \gamma \gamma^{-1}_x, V)$ a local chart about $[\gamma 
\gamma^{-1}_x] \in \wL_{gf(x)}$.  
To compute $ \wf^* \pi'_+ \wg^* R^*\nabla^{\nu}_{Y}$, we 
may restrict our attention to $s^{-1}(T)$, where $T$ is any submanifold of 
$M$ which has $Y$ tangent to it.  We may assume that  $T \subset U_x$, and 
$g f$ restricted to $T$ is a diffeomorphism onto its image $gf (T)$,
which is also  a transverse submanifold, with $gf (T)\subset U_{gf(x)}$.  
Now 
$s^{-1}(T) \cap (U_x, \gamma , V) \simeq V$ and 
$s^{-1}(gf(T)) \cap (U_{gf(x)}, \gamma \gamma^{-1}_x, V) \simeq V $, and the diffeomorphisms with $V$ are just given by the restriction of the  target map $r$.   In addition,
$$
\bigl{(} \nabla^{\nu}_Y \, | \, s^{-1}(T)\bigr{)} \circ r^* = 
r^* \circ \bigl{(} (\nabla_F \otimes \nabla_E)^{\nu}_{Y_{\gamma}}\bigr{)}   
\text{\quad and \quad}
\bigl{(} \nabla^{\nu}_{h_*(Y)} \, | \, s^{-1}(gf(T))\bigr{)} \circ r^* = 
r^* \circ \bigl{(} (\nabla_F \otimes \nabla_E)^{\nu}_{{h_*(Y)}_{\gamma \gamma^{-1}_x}}\bigr{)},
$$ 
where $(\nabla_F \otimes \nabla_E)^{\nu}$ is the quasi-connection on $\wedge T^*F \otimes E$ over $M$, constructed using the normal bundle $\nu$ of $F$, $Y_{\gamma}$ is the parallel translation of $Y$ along $\gamma$, and ${h_*(Y)}_{\gamma \gamma^{-1}_x}$  is the parallel translation of $h_*(Y)$ along $\gamma \gamma^{-1}_x$.   So  $Y_{\gamma} = {h_*(Y)}_{\gamma \gamma^{-1}_x}$.
The restriction of $R$,
$$
R_T:  s^{-1}(gf(T)) \to s^{-1}(T)
$$
is  well defined, since $(gf)^{-1}$ is well defined on $gf(T)$.   In 
fact, it is a diffeomorphism which locally is just $r^{-1} \circ r$.
$R_T$ induces the map on leafwise differential forms
$$
R^*_T:C^{\infty}(s^{-1}(T); \wedge T^*F_s \otimes E) \to 
C^{\infty}(s^{-1}(gf(T)); \wedge T^*F_s \otimes E),
$$ 
which extends to the operator 
$$
R^*_T:C^{\infty}(s^{-1}(T); \wedge T^*(s^{-1}(T)) \otimes E) \to 
C^{\infty}(s^{-1}(gf(T)); \wedge T^*(s^{-1}(gf(T))) \otimes E),
$$ 
It is clear that  $R^*_T \nabla^{\nu}_Y$ is a well defined map, and since 
locally $R_T = r^{-1} \circ r$,  we have $R^*_T  \nabla^{\nu}_Y =  \nabla^{\nu}_{h_*(Y)} 
R^*_T$.   But $R^*_T$ is just the restriction of $R^*$ to $s^{-1}(T)$, 
so $R^*  \nabla^{\nu}_Y =  \nabla^{\nu}_{h_*(Y)} 
R^*$.

The second statement is proved in the same way.
\end{proof}

\begin{proposition}\label{fdnu} 
The operators  $\wf^* \nabla'^{\nu} -  \nabla^{\nu}\wf^*$ and $\wg^*  \nabla^{\nu}-  \nabla'^{\nu}\wg^*$ are leafwise differential operators
\footnote{By a leafwise differential operator, it is sometimes meant, here and in the sequel, operators generated locally by $\rho \mapsto f^* \frac{\pa \rho}{\pa x_i}$ where the $x_i$s are leafwise variables.}, whose composition with a bounded leafwise  smoothing operator is again a bounded leafwise  smoothing operator.
\end{proposition}

\begin{proof}  We will only do the proof for $\wf^*$ as the proof for $\wg^*$ is the same.

Let $\omega \otimes \alpha \otimes \phi \in C^{\infty}_c(\wedge \nu'^{*}_s \otimes \wedge  T^*F'_s \otimes E')$, with $\omega \in s^* \cA^k (M')$,  
$\alpha \in C^{\infty}_c(\cG' ;\wedge T^* F'_s)$, and $\phi \in C^{\infty}_c(\cG' ;E')$.   Then
$$
d'_s(\omega \otimes \alpha \otimes \phi) = (-1)^k  \omega \otimes d'_s(\alpha \otimes \phi). 
$$
Now
$$
\wf^* \nabla'^{\nu} (\omega \otimes \alpha \otimes \phi) = 
\wf^* (d_{M'} \omega  \otimes \alpha \otimes \phi + 
(-1)^k \omega  \otimes \nabla^{\nu}_{F'} \alpha\otimes \phi + 
(-1)^k \omega  \otimes \alpha\otimes  \nabla^{\nu}_{E'}\phi)=
$$
$$
d_M f^*\omega   \otimes \wf^*\alpha \otimes \wf^*\phi + 
(-1)^k f^*\omega   \otimes \wf^*  \nabla^{\nu}_{F'} \alpha \otimes \wf^*\phi  + 
(-1)^k f^* \omega  \otimes \wf^* \alpha \otimes  \wf^*\nabla^{\nu}_{E'} \phi.
$$
On the other hand,
$$
\nabla^{\nu}\wf^*(\omega \otimes \alpha \otimes \phi) = 
$$
$$
d_M f^*\omega   \otimes \wf^*\alpha \otimes \wf^*\phi + (-1)^k f^* \omega  \otimes \nabla^{\nu}_F \wf^* \alpha \otimes \wf^*\phi + (-1)^k  f^* \omega  \otimes\wf^* \alpha\otimes  \nabla^{\nu}_{E}\wf^*\phi.
$$
Thus
$$
(\wf^* \nabla^{\nu'} -  \nabla^{\nu}\wf^*)(\omega \otimes \alpha \otimes \phi)=
(-1)^k f^*\omega \otimes\Big{(} (\wf^*  \nabla^{\nu}_{F'} -   \nabla^{\nu}_F\wf^*) \alpha\otimes  \wf^*\phi +   \wf^* \alpha \otimes  (\wf^*\nabla^{\nu}_{E'} - \nabla^{\nu}_E\wf^*)\phi
\Bigr{)},
$$
which contains no differentiation of  $\omega$, so $\wf^* \nabla'^{\nu} -  \nabla^{\nu}\wf^*$ is indeed a  leafwise operator, as are its individual components  $\wf^*  \nabla^{\nu}_{F'} -   \nabla^{\nu}_F \wf^*$ and 
$\wf^*\nabla^{\nu}_{E'} - \nabla^{\nu}_E \wf^*$.  

Next consider the leafwise operator $\wf^*  \nabla^{\nu}_{F'} -   \nabla^{\nu}_F \wf^*$ acting on $C^{\infty}( \wedge  T^*F'_s)$. 
Set
$$
d_{\nu} = p_{\nu}d_{\cG}  \text{\quad and \quad} d'_{\nu} = p_{\nu'}d_{\cG'}.
$$
In local coordinates, we may write  $\nabla^{\nu}_{F'}$ and $ \nabla^{\nu}_F$ as
$p_{\nu'}(d_{\cG'} + \Theta_{F'})$ and $p_{\nu}(d_{\cG} + \Theta_F)$, respectively,  where 
$ \Theta_{F'}$ and $\Theta_F$ are leafwise differential operators (of order zero) with coefficients in $T^*\cG'$ and $T^*\cG$.  Then we have 
$$
\wf^*  \nabla^{\nu}_{F'} -   \nabla^{\nu}_F \wf^*= 
\wf^* p_{\nu'}(d_{\cG'} +  \Theta_{F'}) -  p_{\nu}(d_{\cG} + \Theta_F)\wf^* =
$$
$$
\wf^* d'_{\nu} - d_{\nu}\wf^*   + \wf^* p_{\nu'}  \Theta_{F'} -  p_{\nu}\Theta_F \wf^*.
$$

\begin{lemma}\label{lemfdnu}
$\wf^* d'_{\nu} - d_{\nu}\wf^*$ and $\wg^* d_{\nu} - d'_{\nu}\wg^*$ are 
leafwise operators, with
$$\wf^* d'_{\nu} - d_{\nu}\wf^*  =  -   \wf^* d'_s +  d_s\wf^*, \quad \text{ and } 
\quad \wg^* d_{\nu} - d'_{\nu}\wg^*  =  -   \wg^* d_s + d'_s  \wg^* .$$
\end{lemma}

\begin{proof}  Again we only prove this only for $\wf^* d'_{\nu} - d_{\nu}\wf^*$.  As $\wf^*  \nabla^{\nu}_{F'} -   \nabla^{\nu}_F \wf^*$ and $\wf^* p_{\nu'}  \Theta_{F'} -  p_{\nu}\Theta_F \wf^*$ are leafwise operators, so is $\wf^* d'_{\nu} - d_{\nu}\wf^*$.  

On $\cG \times B^k$ we have the foliation $F_s \times B^k$, with all its 
baggage.  In  particular, we use the product metric on $\cG \times B^k$, 
and we have the transverse derivative $d^B_{\nu}$.  Local charts on $\cG 
\times B^k$ are given by  subsets of the form  $(U, \gamma, V) \times 
B^k$, where $(U, \gamma, V)$ is a local chart for $\cG$.   It is clear that in these local coordinates,  $d_{\nu}$ and $d^B_{\nu}$ have exactly the same form.  It is then obvious from the 
definitions of  $\pi_{1,*}$ and $e_\omega$, that 
$$
d_{\nu}  (\pi_{1,*} \circ e_\omega)  =  (\pi_{1,*} \circ e_\omega) d^B_{\nu}  \quad \text{ and } 
\quad 
d_s  (\pi_{1,*} \circ e_\omega)  =  (\pi_{1,*} \circ e_\omega) d^B_s,
$$
where $d^B_s$ is the leafwise derivative associated to the foliation $F_s \times B^k$.
As $\wf^* = \pi_{1,*} \circ e_\omega \circ p_f^*$,  to prove that $\wf^* 
d'_{\nu} - d_{\nu}\wf^*  =  - \wf^* d'_s +  d_s\wf^*$, we need only prove that 
$$
p^*_f d'_{\nu} - d^B_{\nu}p^*_f  =  -  p^*_f d'_s + d^B_s p^*_f.
$$
This is purely a local question, and the usual proof shows that we need only prove it for compactly supported functions on $\cG'$.

Denote by $p'_s$ the projection $p'_s :T\cG' \to TF'_s$ determined by the splitting $T\cG' = \nu'_s \oplus TF'_s$, and by  $p^B_F:T(\cG \times B^k) \to T(F_s \times B^k)$ and  $p^B_{\nu}:T(\cG \times B^k) \to \nu_B$, the projections determined by the splitting $T(\cG \times B^k) = \nu_B \oplus T(F_s \times B^k)$.  Let $\phi \in C^{\infty}_c(\cG')$.  If $X \in T(F_s \times B^k)$, 
then  $p^B_{\nu}(X)= 0$, and  ${p_f}_*X \in TF'_s$, so $p'_{\nu}{p_f}_*(X) = 0$.   Thus 
$$
(p^*_f d'_{\nu} \phi - d^B_{\nu} p^*_f \phi )(X) =  
p^*_f ((d'_{\nu} \phi){p_f}_*(X)) -(d_{\cG \times B^k} p^*_f \phi )p^B_{\nu}(X)  = 
p^*_f ((d_{\cG'} \phi)p'_{\nu}{p_f}_*(X))= 0.$$
Next, suppose $X \in \nu_B$,  the normal bundle to  $F_s \times B^k$, and 
note 
that ${p_f}_* X$  is not necessarily in  $\nu'_s$.    Then 
$$
(p^*_f d'_{\nu} \phi)(X) = p^*_f ((d'_{\nu} \phi)({p_f}_*X)) = p^*_f 
((d_{\cG'} \phi)(p'_{\nu}{p_f}_*X)) = $$
$$
p^*_f ((d_{\cG'} \phi)(p_{f*}X))  -  p^*_f  ((d_{\cG'} 
\phi)(p'_s{p_f}_*X)) = (d_{\cG \times B^k} p^*_f  \phi)(X)  -  p^*_f  
((d'_s \phi)({p_f}_*X)) =$$
$$
(d_{\cG \times B^k} p^*_f  \phi)(p^B_{\nu} X)  -  p^*_f  ((d'_s 
\phi)({p_f}_*X)) = (d^B_{\nu} p^*_f \phi - p^*_f  d'_s \phi)(X).
$$
So
$$
(p^*_f d'_{\nu} - d^B_{\nu}p^*_f)\phi  =  (-  p^*_f d'_s \phi)p^B_{\nu}  = (-  p^*_f d'_s \phi)(I - p^B_F) = 
-  p^*_f d'_s \phi + (p^*_f d'_s \phi)p^B_F =  -  p^*_f d'_s \phi + d^B_s p^*_f \phi,
$$
since, restricted to $T(F_s \times B^k)$, $p^*_f d'_s \phi = d^B_s p^*_f \phi$.

Thus 
$\wf^* d'_{\nu} - d_{\nu}\wf^*  =  -   \wf^* d'_s +  d_s\wf^*$.
\end{proof}

So 
$$
\wf^*  \nabla^{\nu}_{F'} -   \nabla^{\nu}_F \wf^* =   d_s\wf^* - \wf^* d'_s +  \wf^* p_{\nu'}  \Theta_{F'} -  p_{\nu}\Theta_F \wf^*,
$$
a leafwise differential operator (of order  at most one).

Finally, consider 
$\wf^*\nabla^{\nu}_{E'} - \nabla^{\nu}_E \wf^*$ acting on $C^{\infty}_c( E')$.
In local coordinates, and with respect to local framings of $E'$ and $E$,  we may write $\nabla_{E'} = d_{\cG'} + \Theta_{E'}$ and $\nabla_{E} = d_{\cG} + \Theta_{E}$, where  $\Theta_{E'}$  and $\Theta_{E}$ are  leafwise differential operators (of order zero) with coefficients in $T^*\cG'$ and $T^*\cG$.   Then
$$
\wf^*\nabla^{\nu}_{E'} - \nabla^{\nu}_E \wf^*  = 
\wf^* p_{\nu'} \nabla_{E'}  -    p_{\nu} \nabla_{E}\wf^* = 
\wf^* p_{\nu'}  (d_{\cG'} + \Theta_{E'})   -    p_{\nu} (d_{\cG} + \Theta_{E}) \wf^*  =
$$
$$
\wf^* d'_{\nu} -    d_{\nu} \wf^*  +  \wf^* p_{\nu'}\Theta_{E'} -    p_{\nu} \Theta_{E} \wf^* =
- \wf^* d'_s  +  d_s\wf^*  +  \wf^* p_{\nu'}\Theta_{E'} -    p_{\nu} \Theta_{E} \wf^*,
$$
since the proof of Lemma \ref{lemfdnu} above extends to show that $\wf^* d'_{\nu} -    d_{\nu} \wf^*  = -\wf^* d'_s  +  d_s\wf^*$, with respect to the local framings.   So
$$
\wf^*\nabla^{\nu}_{E'} - \nabla^{\nu}_E \wf^* =    d_s\wf^* - \wf^* d'_s +  \wf^*p_{\nu'}\Theta_{E'} -    p_{\nu} \Theta_{E} \wf^*,
$$
also a leafwise differential operator (of order at most one).  

Now observe that if we use coordinates on $\cG'$ and $\cG$ and framings of $E'$ and $E$ coming from coordiantes on $M'$ and $M$, and framings of $E'$ and $E$ over $M'$ and $M$,  all of whose derivatives are uniformly bounded, then $ d_s\wf^* - \wf^* d'_s +  \wf^* p_{\nu'}  \Theta_{F'} -  p_{\nu}\Theta_F \wf^*$ and $ d_s\wf^* - \wf^* d'_s +  \wf^*p_{\nu'}\Theta_{E'} -    p_{\nu} \Theta_{E} \wf^*$ are (at worst) order one differential operators which have all of their derivatives uniformly bounded.   Thus $\wf^* \nabla^{\nu'} -  \nabla^{\nu}\wf^*$ and all its derivatives define  bounded operators from $W^*_s (F',E')$ to $W^*_{s-1}(F,E)$ for each $s$, and so their compositions with a bounded leafwise  smoothing operator are again bounded leafwise smoothing operators.
\end{proof}

Note that the proof above also proves that the composition of $\Upsilon_f = \wf^* \nabla'^{\nu} -  \nabla^{\nu}\wf^*$ or $\Upsilon_g =  \wg^*  \nabla^{\nu}-  \nabla'^{\nu}\wg^*$ with a transversely smooth operator is again a transversely smooth operator.  By virtue of  Lemma \ref{lgam}, we will be using only  local extendable $\Gamma$ vector fields $Y_1,...,Y_{m}$ in proving that $\wf^* R'^* \pi'_+ \wg^* R^*$ is transversely smooth.  Thus we may rewrite Lemma \ref{Rdnu} as 
$$
\wf^* R'^* \pi'_+ \wg^* R^*\nabla^{\nu}   = \wf^*R'^*  \pi'_+ \wg^* \nabla^{\nu} R^*
\quad \text{  and  } \quad
\wf^* \nabla'^{\nu} R'^*  \pi'_+ \wg^* R^* = \wf^* R'^*\nabla'^{\nu}\pi'_+ \wg^* R^*.
$$ 
Then
$$
\pa_{\nu}(\wf^* R'^* \pi'_+ \wg^*R^*) =  [\nabla^{\nu},\wf^*R'^* \pi'_+ \wg^*R^* ]  =
\nabla^{\nu}\wf^*R'^* \pi'_+ \wg^*R^*   - \wf^* R'^*\pi'_+ \wg^* R^* \nabla^{\nu}  = 
$$
$$
\wf^* R'^* \nabla'^{\nu}\pi'_+ \wg^* R^*  - \wf^* R'^* \pi'_+ \nabla'^{\nu}\wg^*R^*   -  \Upsilon_f  R'^* \pi'_+ \wg^* 
   -    \wf^* R'^* \pi'_+ \Upsilon_g R^*. 
$$
So,
\begin{Equation}\label{nabeq}
\hspace{1.5cm}
$\pa_{\nu}^{Y_1}(\wf^* R'^* \pi'_+ \wg^*R^*) = i_{\what{Y}_1}\wf^* R'^*\pa_{\nu'}(\pi'_+) \wg^* R^*   -  
(i_{\what{Y}_1}\Upsilon_f)  R'^* \pi'_+ \wg^*R^*    -    \wf^* R'^*  \pi'_+ (i_{\what{Y}_1} \Upsilon_g ) R^*.$
\end{Equation}
By assumption, $\pa_{\nu'}(\pi'_+ )$ is a bounded leafwise smoothing operator, so $i_{\what{Y}_1}\wf^* R'^*\pa_{\nu'}(\pi'_+) \wg^* R^*$ is also.   The operators $i_{\what{Y}_1}\Upsilon_f$, and $i_{\what{Y}_1} \Upsilon_g$ are leafwise operators which have all their derivatives bounded, so their  
composition with a bounded leafwise smoothing operator (e.g.\ $ R'^* \pi'_+ \wg^*$) is again a bounded leafwise smoothing operator. Thus for any local extendable $\Gamma$ vector field $Y_1$ on $M$, $\pa^{Y_1}_{\nu}(\wf^* R'^* \pi'_+ \wg^*R^*)$ is a bounded leafwise smoothing operator.

To continue the induction argument, we need the following.
\begin{lemma} \label{pushforward}
Let $Y \in C^{\infty}(\nu)$ be a local  extendable $\Gamma$ vector field, then there is a bounded vector field $Z'$ on $\cG'$ so that 
for any  $([\gamma],t) \in \cG \times B^k$,
$$
i_{\what{Y}([\gamma],t)} p^*_f = p^*_f  i_{Z' (p_f([\gamma],t))}.
$$
\end{lemma}
\noindent
Given this, then at $([\gamma], t)  \in \cG \times B^k$ we have
$$
i_{\what{Y}_1} p^*_f R'^* \pa_{\nu'}( \pi'_+)([\gamma],t) = 
i_{\what{Y}_1([\gamma],t)} p^*_f  R'^*  \pa_{\nu'}( \pi'_+) =  
p^*_f (R'^*   i_{Z'_1(p_f([\gamma],t))} \pa_{\nu'}( \pi'_+) ) = 
p^*_f ( R'^*   i_{Z'_1}  \pa_{\nu'}( \pi'_+) p_f([\gamma],t)).
$$
That is, 
$i_{\what{Y}_1} p^*_f  R'^* \pa_{\nu'}( \pi'_+)   = p^*_f R'^* i_{Z'_1}  
 \pa_{\nu'}( \pi'_+)$
so
$$
i_{\what{Y}_1} \wf^* R'^*  \pa_{\nu'}( \pi'_+) \wg^*R^*  = 
\wf^*R'^*  i_{Z'_1}  \pa_{\nu'}( \pi'_+)\wg^* R^* .
$$

\begin{lemma} \label{gents}
If $\rho$ is a transversely smooth operator on $\cA^*_{(2)}(F'_s,E')$ and $Z'$ is a bounded vector field on $\cG'$, then $i_{Z'} \pa_{\nu'}(\rho)$ is a transversely smooth operator. 
\end{lemma}
  
\begin{proof} 
Since $i_{Z'} \pa_{\nu'}(\rho) = i_{p_{\nu'}(Z')} \pa_{\nu'}(\rho)$, we may assume that  $Z' = \sum_j g_j \what{X}'_j$, where $X'_j$ is a finite local basis for the vector fields on $M'$, and the $g_j$ are smooth functions which are globally bounded along with all their derivatives.  Then 
$i_{Z'} \pa_{\nu'}(\rho) =  \sum_j g_j i_{\what{X}'_j}  \pa_{\nu'}(\rho)   =  
\sum_j g_j   \pa^{X'_j}_{\nu'}(\rho)$, which is clearly transversely smooth since the $g_j$ and all their derivatives are globally bounded.
\end{proof}

Using Equation \ref{nabeq}, we have
$$
\pa_{\nu}^{Y_2}\pa_{\nu}^{Y_1}(\wf^*R'^*  \pi'_+ \wg^*R^* ) =  
\pa_{\nu}^{Y_2}\Bigl{(} \wf^* R'^*  i_{Z'_1}  \pa_{\nu'}( \pi'_+)\wg^*R^*    -  
(i_{\what{Y}_1}\Upsilon_f) R'^*  \pi'_+ \wg^* R^*    -    
\wf^* R'^*  \pi'_+(i_{\what{Y}_1} \Upsilon_g )  R^* \Bigr{)}.
$$
Repeating the argument above  we get 
$$
\pa_{\nu}^{Y_2}( \wf^* R'^*  i_{Z'_1}  \pa_{\nu'}( \pi'_+)\wg^*R^*) = 
$$
$$
i_{\what{Y}_2} \wf^* R'^* \pa_{\nu'}( i_{Z'_1}  \pa_{\nu'}( \pi'_+))\wg^*  R^*  -  
(i_{\what{Y}_2}\Upsilon_f)  R'^* i_{Z'_1}  \pa_{\nu'}( \pi'_+)\wg^* R^*  -     
\wf^* R'^* i_{Z'_1}  \pa_{\nu'}( \pi'_+)(i_{\what{Y}_2} \Upsilon_g ) R^*=
$$
$$
\wf^* R'^* i_{Z'_2} \pa_{\nu'}( i_{Z'_1}  \pa_{\nu'}( \pi'_+))\wg^*  R^*  -  
(i_{\what{Y}_2}\Upsilon_f)  R'^* i_{Z'_1}  \pa_{\nu'}( \pi'_+)\wg^* R^*  -    
 \wf^*  R'^* i_{Z'_1}  \pa_{\nu'}( \pi'_+) (i_{\what{Y}_2} \Upsilon_g ) R^*,
$$
which is bounded and leafwise smoothing since $i_{Z'_1} \pa_{\nu'}( \pi'_+)$ is transversely smooth.

As $\pa_{\nu}^{Y_2}$ is a derivation, we have 
$$
\pa_{\nu}^{Y_2}((i_{\what{Y}_1}\Upsilon_f)  R'^* \pi'_+ \wg^* R^*) = 
\pa_{\nu}^{Y_2}(i_{\what{Y}_1}\Upsilon_f)  (R'^* \pi'_+ \wg^* R^*) +
(i_{\what{Y}_1}\Upsilon_f ) \pa_{\nu}^{Y_2}(R'^* \pi'_+ \wg^* R^*).
$$
The operators $\pa_{\nu}^{Y_2}(i_{\what{Y}_1}\Upsilon_f)$ and $i_{\what{Y}_1}\Upsilon_f $ composed with bounded leafwise smoothing operators produce bounded leafwise smoothing operators.  As $R'^* \pi'_+ \wg^* R^*$ and $ \pa_{\nu}^{Y_2}(R'^* \pi'_+ \wg^* R^*)$ are bounded leafwise smoothing operators, this term is a bounded leafwise smoothing operator. Similarly for the third term.

Now, a straight forward induction argument finishes the proof, modulo the proof of
Lemmas \ref{pushforward}.
\begin{proof} 
To prove Lemma \ref{pushforward}, we ``factor through the graph".
In particular,  consider the  map $p_{f,G}:\cG \times B^k \to \cG
\times B^k  \times \cG'$ given by $p_{f,G}(\gamma,t) = (\gamma,t,p_f
(\gamma,t))$ which is a diffeomorphism onto its image.  Denote by
$F'_{G,s}$ the foliation of $\cG \times B^k  \times \cG'$ whose
leaves are of the form $\wL \times B^k \times \wL'$, and denote by
$E'_G$ the pull back of $E'$ under the projection $  \cG \times
B^k  \times \cG' \to \cG'$.
We want to construct a transversely smooth idempotent $\pi'_{+,G}$
which will play the role of $\pi'_+$.   However, $\pi'_{+,G}$ will
not be acting on $\cA^*_{(2)}(F'_{G,s},E'_G)$ over $M \times M'$,
but rather on the space denoted   $\cA^*_{(2)}(F'_{G,s}, \wedge
T^*F'_s \otimes E'_G)$  over $M \times M'$, which associates to
each $(x,x')$ the Hilbert space $L^2(\wL'_{x'}; \wedge T^*F'_s
\otimes E')$.  Then
$$
(\pi'_{+,G})_{(x,x')} := (\pi'_+)_{x'}:L^2(\wL'_{x'}; \wedge
T^*F'_s \otimes E') \to L^2(\wL'_{x'}; \wedge T^*F'_s \otimes E')
$$
is well defined, and it is obvious that $\pi'_{+,G}$ is a
transversely smooth idempotent, and has image $\Ker(\Delta^{E'+}_
{\ell})$.

To define the action $\wtit{p}_{f,G}^*$ of $\wtit{p}_{f,G}$ on $
\cA^*_{(2)}(F'_{G,s}, \wedge T^*F'_s \otimes E'_G)$, we may
consider this space as a subspace of all the forms on $\cG \times
B^k  \times \cG'$ by using the pull back of the projection $\cG
\times B^k  \times \cG' \to  \cG'$.  When we do so,  $\wtit{p}_{f,G}
^*$ is just the usual induced map, and on each fiber $L^2(\wL'_{f
(x)}; \wedge T^*F'_s \otimes E')$ it equals $p_f^*$.

Next define
$$
\wg^*_G:\cA^*_{(2)}(F_s,E) \to \cA^*_{(2)}(F'_{G,s}, \wedge T^*F'_s
\otimes E'_G)
$$
to be
$$
(\wg^*_G)_{g(x')}:= (\wg^*)_{g(x')}:L^2(\wL_{g(x')}; \wedge T^*F_s
\otimes E) \to L^2(\wL'_{x'}; \wedge T^*F'_s \otimes E'),
$$
for each $x' \in M'$.

Finally, the action of $R'^*$ on $\cA^*_{(2)}(F'_{s},E')$ extends
easily to an action on
 $\cA^*_{(2)}(F'_{G,s}, \wedge T^*F'_s \otimes E'_G)$.

Then $p^*_{f,G} R'^* \pi'_{+,G} \wg^*_G R^* = p^*_f R'^* \pi'_+
\wg^*R^*$, and we may work with $\cG \times B^k  \times \cG'$, $F'_
{G,s}$, $p^*_{f,G}$, $\wg^*_G$, and $\pi'_{+,G}$ in place of  $
\cG'$, $F'$, $p^*_f$, $\wg^*$, and $\pi'_+$, respectively.  As $p_
{f,G}$ is a diffeomorphism onto its image, we may push forward
vector fields such as the $\what{Y}_i$ on $\cG$ {(which are bounded
because $F$ is Riemannian)}  to bounded vector fields $Z'_i$ on $
\cG \times B^k  \times \cG'$.  Note that these vector fields are
only defined along the image of $p_{f,G}$, but this is sufficient
for our purposes, since things of the form
$$
\wf^* R'^* i_{Z'_2} \pa_{\nu'}( i_{Z'_1}  \pa_{\nu'}( \pi'_+))
\wg^*  R^*  -
(i_{\what{Y}_2}\Upsilon_f)  R'^* i_{Z'_1}  \pa_{\nu'}( \pi'_+)\wg^*
R^*  -
 \wf^*  R'^* i_{Z'_1}  \pa_{\nu'}( \pi'_+) (i_{\what{Y}_2}
\Upsilon_g ) R^*,
$$
are still well defined.
\end{proof}

This completes the proof that  $\pi^f_+: \cA^{\ell}_{(2)}(F_s,E) \to \Im\wf^*_+$ is a transversely smooth idempotent.
\end{proof}

The same argument shows that $\Im {\wf^*_-}$, and $\Im {\wf^*}$ determine 
smooth bundles  over $M/F$, denoted   $\pi^f_-$ and $\pi^f$ respectively.  In fact, we may use the  proof above to prove.

\begin{proposition}\label{gentsgen}
If $\rho$ is a transversely smooth operator on $\wedge \nu'^*_s \otimes F'_s \otimes E'$, then $\wf^* R'^* \rho  \, \wg^* R^*$ is a transversely smooth operator on  $\wedge \nu^*_s \otimes F_s \otimes E$. 
\end{proposition}

\section{Induced connections}

Let $\nabla'= \pi'_{+} {\nabla'}^{\nu} \pi'_{+}$ be the connection on the sub-bundle $\pi'_{+} =  \Ker ({\Delta}^{E'+}_{\ell})$,  determined by the quasi-connection ${\nabla'}^{\nu}$ on  $\wedge^{\ell}T^*F'_s \otimes E'$.   We now prove that $\nabla'$ induces a connection $\nabla$ on $\pi^f_+$. 

\begin{lemma} 
If $\xi'$ is a local invariant section of  $\pi'_+$, then  $\wf^*(\xi')$ is a local invariant section of $\pi^f_+$.
\end{lemma}

\begin{proof}
Recall that for $([\gamma],t) \in \cG \times B^k$, $p_f([\gamma],t) = 
[P_f(\gamma,t)  \cdot  (f \circ \gamma)]$,  the composition of the leafwise paths 
$P_f(\gamma,t)$ and $f \circ \gamma$, where $P_f(\gamma,t) : [0,1] \to 
L'_{f(s(\gamma))}$ is the leafwise path given by
$$
P_f(\gamma,t)(s) = p_f(r(\gamma),st).
$$
Then
$$
\wf^*(\xi')([\gamma \gamma_1]) = 
\pi_{1,*} \circ e_\omega (( p_f^*\xi')([\gamma \gamma_1],t)) = \pi_{1,*} 
\circ e_\omega ( p_f^*(\xi'(P_f(\gamma \gamma_1,t)\cdot (f \circ \gamma 
\gamma_1)))) = $$
$$
\pi_{1,*} \circ e_\omega ( p_f^*(\xi'(P_f(\gamma ,t) \cdot (f \circ \gamma)\cdot (f 
\circ \gamma_1)))) = \pi_{1,*} \circ e_\omega ( p_f^*(\xi'(P_f(\gamma ,t)\cdot (f 
\circ \gamma)))),
$$
{since $\xi'$ is local invariant.  But this last equals}
$$
\pi_{1,*} \circ e_\omega ( p_f^*\xi'([\gamma],t)) =
\wf^*(\xi')([\gamma]). 
$$
\end{proof}

\begin{lemma} 
Any local invariant section $\xi$ of $\pi^f_+$ induces a local invariant section $\wf^{-*}\xi$ of $\pi'_+$.
\end{lemma}
\begin{proof}  Let $T$ be a transversal in $M$ on which $\xi$ is defined.  
We may assume that $T$ is so small that $f \, | \, T$ is a diffeomorphism 
onto its image $T'$.    Then $(\wf^*)^{-1}:\Im{\wf^*_+} \to \Ker 
({\Delta}^{E'+}_{\ell}) $ is well defined over $T$, and in fact is given by 
the map $R'^* P'_{\ell} \wg^* R^* \,  | \, T$.   To see this,
note that over $T'$, the map $R'^* P'_{\ell} \wg^* R^* \wf^* :\Ker 
({\Delta}^{E'}_{\ell}) \to \Ker ({\Delta}^{E'}_{\ell})$
is the identity map, since it induces the identity map on cohomology, and 
that $\Ker ({\Delta}^{E'+}_{\ell}) \subset \Ker ({\Delta}^{E'}_{\ell})$.  For 
simplicity, we shall denote $R'^* P'_{\ell} \wg^* R^* \,  | \, T$ by 
$\wf^{-*}$.  For $x' \in T'$, define 
$$
(\wf^{-*}\xi)(x') \equiv     \wf^{-*}( \xi(f^{-1}(x'))).
$$
This gives a well defined smooth section on $T'$.  Extend it to a local 
invariant section on a neighborhood of $T'$.  We leave it to the reader to 
show that this construction is well defined, that is it does not depend on 
the choice of $T$.
\end{proof}

In order to define the induced connection $\nabla$, we need only define it 
on local invariant sections, and then extend it using (1) of Definition 
\ref{connection}.

\begin{definition}
Let $\xi$ be a local invariant section of $\pi^f_+$.    Given $X \in 
TM$, set $X' = f_*(X)$.  Define
$$
\nabla_X(\xi) = \wf^*({{{\nabla'}}}_{X'}(\wf^{-*}\xi)).
$$
Extend to $\xi \in C^{\infty}(\wedge T^*M;\pi^f_+)$ by using  (1) 
of Definition~\ref{connection}.
\end{definition}

\begin{proposition}
$\nabla$ is a connection on  $\pi^f_+$.
\end{proposition}

\begin{proof}
We need to check that the four conditions of Definition \ref{connection} 
are satisfied.

\medskip\noindent
{\bf \ref{connection}(1):}  For differential forms, this is satisfied by 
definition, so we need to check it for functions.  Specifically,  we need 
that for any local  function $\omega$ on $M$ which is constant on plaques 
of $F$ (i.e.\ local invariant functions), and for any $X \in TM$, and any 
local invariant  section $\xi$ of  $\pi^f_+$,
$$
\nabla _X(\omega \xi) = d_M\omega(X)  \xi + \omega  \nabla_X \xi.
$$
If  $X \in TF$, this is trivially true since both sides are zero.  
Now suppose that $X$ is transverse to $F$, with $X' =  f_*(X)$, and let 
$T$ be a transversal of $F$ with $X$ tangent to $T$.  We may assume that 
$T$ is so small that  $f$ restricted to $T$ is a diffeomorphism onto its 
image $T'$, a transversal of $F'$, with inverse $f^{-1}:T' \to T$.  The 
vector $X'$ is tangent to $T'$, and thanks to Corollary \ref{nablalocal}, 
we have
$$
\nabla_X(\omega \xi) =  \wf^*({{{\nabla'}}}_{X'}(\wf^{-*}(\omega \xi))) 
=   \wf^*({{{\nabla'}}}_{X'}((\omega \circ f^{-1}) \wf^{-*}\xi)) =  
\wf^*\Big{[}X'(\omega \circ f^{-1}) \wf^{-*}\xi +(\omega \circ f^{-1}) 
{{{\nabla'}}}_{X'} \wf^{-*}\xi\Big{]}=
 $$
$$
(X'(\omega \circ f^{-1})\circ f) \wf^*\wf^{-*}\xi +\omega 
\wf^*({{{\nabla'}}}_{X'} \wf^{-*}\xi) = ( X \omega)  \xi +  \omega   
\nabla_X \xi     =    d_M\omega(X)  \xi +  \omega   \nabla_X \xi.
$$

\medskip\noindent
{\bf \ref{connection}(2):}  If $X \in TF$, then $X' \in TF'$, and as  
$\wf^{-*}\xi$ is local invariant, 
$ {{{\nabla'}}}_{X'}(\wf^{-*}\xi)=0$, so $\nabla_X(\xi) = 
\wf^*({{{\nabla'}}}_{X'}(\wf^{-*}\xi)) = 0$ and $\nabla$ is flat along $F$.

\medskip\noindent
{\bf \ref{connection}(3):} The fact that  $\nabla$ is $\cG-$invariant is 
a simple exercise which is left to the reader.

\medskip\noindent
{\bf \ref{connection}(4):}  We need to show that 
$
A = \nabla \pi^f_+ - \pi^f_+ \nabla^{\nu}\pi^f_+:C^{\infty}_c(\wedge T^*M; 
\wedge T^*F_s \otimes E) \to C^{\infty}(\wedge T^*M; \pi^f_+)
$ 
is transversely smooth.    Now $\pi^f_+ = \wf^* R'^* \pi'_+ \wg^* R^* P_{\ell}$ and   $\nabla = \wf^* {\nabla'} \wf^{-*}  = \wf^* {\nabla'} R'^* P'_{\ell}\wg^* R^* = \wf^* \pi'_{+} {\nabla'}^{\nu}\pi'_{+}  R'^*  P'_{\ell}\wg^*  R^*$.  Using the proof of Proposition \ref{fdnu}, we have that, modulo transversely smooth operators, 
$$
A = \wf^* \pi'_{+} {\nabla'}^{\nu} \pi'_{+}  R'^* P'_{\ell} \wg^*  R^* \wf^*   R'^* \pi'_+ \wg^* R^* P_{\ell}   - 
\wf^*  R'^*  \pi'_+ \wg^* R^* P_{\ell} \nabla^{\nu} \wf^*   R'^* \pi'_+ \wg^* R^* P_{\ell} =
$$
$$
\wf^* \pi'_{+} {\nabla'}^{\nu} \pi'_{+}   R'^* P'_{\ell} \wg^* R^*\wf^*   R'^*  P'_{\ell}  \pi'_+ \wg^* R^* P_{\ell}   - 
\wf^*  R'^*  \pi'_+ \wg^*  R^*  P_{\ell}\wf^* \nabla'^{\nu}   R'^* \pi'_+ \wg^* R^* P_{\ell} =
$$
$$
\wf^* \pi'_{+} {\nabla'}^{\nu} \pi'_{+}   R'^* \pi'_+ \wg^* R^* P_{\ell}  - 
\wf^*  R'^*  \pi'_+ \wg^*  R^*  P_{\ell}\wf^* \nabla'^{\nu}   R'^* \pi'_+ \wg^* R^* P_{\ell}, 
$$
since $P'_{\ell} \wg^* R^*\wf^*   R'^*  P'_{\ell}$ is the identity on $\Im(P'_{\ell}) =  \Ker({\Delta}^{E'}_{\ell}) \supset \Im(\pi'_+)$.  Now  $ R'^* \pi'_+   =   \pi'_+ R'^*$, and 
$\nabla'^{\nu}\pi'_+  =   (\nabla'^{\nu}\pi'_+)\pi'_+  = \pi'_+  \nabla'^{\nu} \pi'_+   + [\nabla'^{\nu}, \pi'_+] \pi'_+   $, and $ [\nabla'^{\nu}, \pi'_+] $ is transversely smooth since $\pi'_+$ is.  So using Proposition \ref{gentsgen},  we have that modulo transversely smooth operators,  
$$
\wf^* R'^* \pi'_+ \wg^* R^* P_{\ell}\wf^* \nabla'^{\nu}  R'^* \pi'_+ \wg^* R^* P_{\ell}   =
\wf^* R'^* \pi'_+ P'_{\ell} \wg^* R^* P_{\ell}\wf^* \pi'_+    \nabla'^{\nu} \pi'_+   R'^* \wg^* R^* P_{\ell}   =
$$
$$ 
\wf^* \pi'_+ R'^*   P'_{\ell} \wg^* R^* P_{\ell}\wf^* P'_{\ell}\pi'_+\nabla'^{\nu}  \pi'_+  R'^* \wg^* R^* P_{\ell}   =
\wf^*  \pi'_+  \nabla'^{\nu}  \pi'_+ R'^*\wg^* R^* P_{\ell},
$$
since $R'^*  P'_{\ell} \wg^* R^* P_{\ell}\wf^* P'_{\ell}$ is also the identity on $\Im(P'_{\ell})$.  As  
$\pi'_{+}   R'^* =  \pi'_{+}  \pi'_{+}   R'^*  =  \pi'_{+}   R'^* \pi'_+ $, $A = 0$ modulo transversely smooth operators, that is, $A$ is transversely smooth.
\end{proof}

\section{Leafwise homotopy invariance of the twisted higher harmonic signature}

In this section we prove our main theorem that the twisted higher harmonic signature is a leafwise homotopy invariant.

\begin{theorem} \label{main}
Suppose that  $M$ is a compact Riemannian manifold,  with oriented Riemannian foliation $F$ of dimension $2\ell$, and that $E$ is a leafwise flat complex bundle over $M$ with a (possibly indefinite) non-degenerate Hermitian metric which is preserved by the leafwise flat structure.  Assume that the projection onto  $\Ker( \Delta^{E}_{\ell} )$ for the associated foliation $F_s$ of the homotopy groupoid of $F$ is transversely smooth.  Then $\sigma(F,E)$ is a leafwise homotopy invariant.
\end{theorem}

Recall that the projection onto $\Ker( \Delta^{E}_{\ell} )$ is transversely smooth: for the (untwisted) leafwise signature operator; whenever $E$ is a bundle associated to the normal bundle of the foliation; and whenever the leafwise parallel translation on $E$  defined by the flat structure is a bounded map, in particular whenever the preserved metric on $E$ is positive definite.  Note also that these conditions are preserved under pull-back by a leafwise homotopy equivalence.

Suppose that $M'$,  $F'$,  and $E'$ satisfy the hypothesis of Theorem  \ref{main},
and that $f:M \to M'$ is a leafwise homotopy equivalence, which preserves the leafwise orientations.  Set $E = f^*(E')$ with the induced leafwise flat structure and preserved metric. Assume that the projections to $\Ker( \Delta^{E}_{\ell} )$ and $\Ker( \Delta^{E'}_{\ell} )$ are transversely smooth.  Then we need to show that 
$$
 \ch_a(\pi_{\pm}) = f^*(\ch_a(\pi'_{\pm})).
$$
We do this in two stages.  The first is to prove

\begin{theorem}\label{secondeq}
$ \ch_a(\pi_{\pm} )  = \ch_a(\pi^f_{\pm})$.
\end{theorem}
\begin{proof}
Recall that $\pi^f_{\pm}= \wf^* R'^* \pi'_{\pm} \wg^* R^*P_{\ell}$, and set 
$$
\what{\pi}^{f,t}_{\pm}  = t \pi^f_{\pm} + (1-t) P_{\ell} \pi^f_{\pm}.
$$
A simple computation, using the fact that  $\pi^f_{\pm} P_{\ell} =  
\pi^f_{\pm}$, shows that  the $\what{\pi}^{f,t}_{\pm}$ are idempotents, and as  
$P_{\ell}$ and the 
$\pi^f_{\pm}$  are transversely smooth, the $\what{\pi}^{f,t}_{\pm}$ 
are smooth families of transversely smooth idempotents. 
It follows from  Theorem \ref{family} that $\ch_a(\what{\pi}^{f,0}_{\pm}) = \ch_a(\what{\pi}^{f,1}_{\pm})$.
Since $\what{\pi}^{f,1}_{\pm}= \pi^f_{\pm}$,  we need to show that 
$\ch_a(\what{\pi}^{f,0}_{\pm}) = \ch_a (\pi_{\pm})$.
We will do only the $+$ case as the other case is the same.  Set $\what{\pi}^f_{\pm} = \what{\pi}^{f,0}_{\pm}$.

Consider the pairings  $< \, , \, >,$ and $Q$ defined in Section \ref{tsig}.
Note that $Q(d_s\alpha_1, \alpha_2 ) =  (-1)^{\ell+1}Q(\alpha_1, d_s\alpha_2 )$.
Using a partition of unity and linearity, this reduces to considering sections of compact support of the form $\alpha = \omega \otimes \phi$, where $\omega \in C^{\infty}_c(\wL; \wedge T^*\wL)$ and $\phi$ is a flat section of $E$, where it is immediate.
So ${\overline B}^\ell_{(2)} (F_{s}, E)$ is totally isotropic under the pairing $Q$, and  it
is orthogonal to $\Ker(\Delta^E_{\ell})$ under the pairing $<\, ,
\, >$.  In addition, this equation implies that $Q$
induces a well defined pairing
$$
Q:\oH^{\ell}_{(2)}(F_{s}, E) \otimes \oH^{\ell}_{(2)}(F_{s}, E) \to
\maB(M),
$$
where $\maB(M)$ denotes the Borel $\C$ valued functions on $M$.
It further implies that $P_{\ell}$ restricted to the cocycles  $Z^{\ell}_{(2)}(F_{s}, E)$ 
preserves $Q$.
The subspaces  $\Ker (\Delta^{E +}_\ell)$ and $\Ker(\Delta^{E -}_\ell)$ are orthogonal under both
of the pairings, since
$Q(\what{\tau}\alpha_1, \alpha_2 ) =Q(\alpha_1,  \what{\tau}\alpha_2)$.
As
$\Ker(\Delta^E_{\ell}) =  \Ker(\Delta^{E +}_{\ell}) \oplus
\Ker(\Delta^{E -}_{\ell})$,
so also $Z^{\ell}_{(2)}(F_s, E) = \Ker(\Delta^{E +}_{\ell}) \oplus
\Ker(\Delta^{E -}_{\ell}) \oplus \overline{B}^{\ell}_{(2)}(F_s, E).$

The kernels of both $\what{\pi}^f_+$ and $\pi_+$ contain $\Ker(P_{\ell})$, so we 
may restrict our attention to $\Im(P_{\ell}) = \Ker(\Delta^E_{\ell})$.    
The image of $\what{\pi}^f_+$ is $P_{\ell}(\Im(\wf^*_+))$. 

\begin{lemma}
$\pi_+:P_{\ell}(\Im(\wf^*_+)) \to  \Ker(\Delta^{E+}_{\ell})$ is an 
isomorphism with bounded inverse.
\end{lemma}

\begin{proof}   
By Proposition \ref{preserves},  $\wf^*$ restricted to
$\Ker({\Delta}^{E'}_{\ell})$ takes the pairing $Q'$ to the pairing
$Q$.
(Note that $Q$ is $\pm$ definite on the $\Im(\pi_{\pm})$ if $\ell
$ is even, while it is $\sqrt{-1}Q$, which is $\pm$ definite on $\Im
(\pi_{\pm})$ is $\ell$ is odd.  We will finesse this point.)
Since $P_{\ell}$ (restricted to the cocycles) preserves the pairing
$Q$,
$Q$ is positive definite on $P_{\ell}(\Im(\wf^*_+))$.
Given $0 \neq \alpha \in P_{\ell}(\Im(\wf^*_+))$, write it
(uniquely) as
$\alpha = \alpha_+ + \alpha_-$, where $\alpha_{\pm} \in \Ker
(\Delta^{E\pm}_{\ell})$.  Then
$$
0 \,\,\, <  \,\,\,Q(\alpha,\alpha) \,\,\,= \,\,\,<\alpha_+,\alpha_+> -
<\alpha_-,\alpha_->  \quad  \leq    \quad  <\alpha_+,\alpha_+>,
$$
so $\pi_+(\alpha) = \alpha_+ \neq 0$ and $\pi_+:P_{\ell}(\Im(\wf^*_+))
\to  \Ker(\Delta^{E+}_{\ell})$ is one-to-one.

The above inequality also implies that $\pi_+^{-1}$ is bounded,
with bound
$\sqrt{2}$.   The element $\alpha = \pi_+^{-1}(\alpha_+)$ and
$||\alpha||^2  \,\,\, =  \,\,\, <\alpha,\alpha>  \,\,\,  =   \,\,\,
<\alpha_+,\alpha_+> + <\alpha_-,\alpha_->  \,\,\, =  \,\,\, ||
\alpha_+||^2
+ ||\alpha_-||^2$.  Since $0 < Q(\alpha,\alpha) $, $||\alpha_-||^2
\,\,\,  < \,\,\,  ||\alpha_+||^2$, so $||\pi_+^{-1}(\alpha_+)||^2
\,\,\,
=  \,\,\, ||\alpha||^2 \,\,\,  \leq  \,\,\,  2 ||\alpha_+||^2$.

Next we show that  $\pi_+$ is onto.   
Choose  $\alpha \in \Ker(\Delta^{E+}_\ell)$ which is orthogonal to $
\pi_+(P_{\ell}(\Im {\wf^*_+}))$. The subspaces $P_{\ell}(\Im {\wf^*_
+})$ and $P_{\ell}(\Im {\wf^*_-})$ are orthogonal under $Q$.
Their direct sum is the space $\Ker(\Delta_\ell)$ of all harmonic
forms, since $\pi^f_+ +\pi^f_-$ induces the identity on cohomology.
Write $\alpha = \beta_+ + \beta_-$, with $\beta_\pm \in P_{\ell}
(\Im {\wf^*_\pm})$. Then    
$$
\|\alpha\|^2 \,\, = \,\, Q (\alpha, \alpha) \,\, = \,\,  Q(\alpha, \beta_+) + Q
(\alpha, \beta_-)
\,\, = \,\, Q(\alpha, \pi_+\beta_+) + Q(\alpha, \pi_-\beta_+) + Q(\alpha,
\beta_-)
\,\, = \,\,  Q(\alpha, \beta_-).
$$
The last equality is a consequence of the facts that $\alpha$ is $Q
$ orthogonal to $\pi_+(P_{\ell}(\Im {\wf^*_+}) )$ and that $\Ker
(\Delta^{E+}_\ell)$ and $\Ker(\Delta^{E-}_\ell)$ are $Q$
orthogonal. Hence, we have,
$$
0 \leq \|\alpha\|^2 = Q(\beta_+, \beta_-) + Q(\beta_-, \beta_-) = Q
(\beta_-, \beta_-) \leq 0.
$$
So, $\alpha=0$, and $\pi_+$ is onto.
 \end{proof}
  
The map $\pi_+^{-1}$ is defined on 
$\pi_+(P_{\ell}(\Im {\wf^*_+}) )$.  Define $\rho_+$ to be orthogonal 
projection onto $\pi_+(P_{\ell}(\Im {\wf^*_+}) )$ composed with  
$\pi_+^{-1}$, i.e.,  $$
\rho_+ =  \pi_+^{-1} \circ \pi_+:\cA^{\ell}_{(2)}(F_s,E)  \to P_{\ell}(\Im(\wf^*_+)).
$$ 
Then $\rho_+$ is an idempotent and has image $P_{\ell}(\Im 
{\wf^*_+})$, which equals $\what{\pi}^f_+$.   
 We claim that 
$\rho_+$ is transversely smooth.  If so, then $\ch_a(\rho_+)$ is defined 
and $\ch_a(\rho_+) = \ch_a(\what{\pi}^f_+)$, since they have the same image. 
Note that $\rho_+$ is projection to $P_{\ell}(\Im(\wf^*_+))$ along 
$\Ker(\pi_+)$.  With this description, it is immediate that $\rho_+ \circ 
\pi_+ = \rho_+$ and $\pi_+ \circ \rho_+ = \pi_+$ since 
$\pi_+$ is projection to  $\Ker(\Delta^{E+}_{\ell})$ along $\Ker(\pi_+)$.   
As above, we may form the smooth family of transversely smooth idempotents 
$t \rho_+ + (1-t) \pi_+$ which connects $\rho_+$ to $\pi_+$.  Again, it 
follows from Theorem \ref{family} that $\ch_a(\rho_+) = \ch_a(\pi_+)$, and 
we have $ \ch_a(\pi_+) = \ch_a(\pi^f_+)$.  So to finish the proof we need only
show that $\rho_+$ is transversely smooth.   

Now 
$$
\what{\pi}^f_{\pm} \, = \,  P_{\ell}  \pi^f_{\pm}  \, = \,  P_{\ell}  \wf^* R'^* \pi'_{\pm}  
\wg^* R^* P_{\ell},
$$
and recalling that $P'_\ell \wg^* R^* P_\ell \wf^*R'^* P'_\ell = P'_\ell$, and $ \pi'_{\pm}= \pi'_{\pm}P'_\ell =P'_\ell  \pi'_{\pm}$,  we have
$$
(\what{\pi}^f_{\pm})^2 \, = \, \what{\pi}^f_{\pm} \quad \text{ and } \quad \what{\pi}^f_{\pm}\what{\pi}^f_{\mp} \, = \, 0.
$$
These idempotents are transversely smooth, since $P_{\ell}$ and the $\pi^f_{\pm}$ are 
transversely smooth.  They also  satisfy $\what{\pi}^f_{+} + \what{\pi}^f_{-} = 
P_{\ell}$, and their kernels both contain $\Ker(P_{\ell})$.  Finally, note 
that the $\Im(\what{\pi}^f_{\pm}) = P_{\ell} (\Im(\wf^*_{\pm}))$.  Next set
$$
 A =  \pi_+ + \what{\pi}^f_-.
$$

\begin{lemma}
The operator $A$ {and  its adjoint $A^t$} are transversely 
smooth, and $A$ is an isomorphism when restricted to $\Ker(\Delta^E_{\ell})$. 
\end{lemma}

\begin{proof}  
$A$  is transversely smooth because both $\pi_+$ and $\what{\pi}^f_-$ are. 
As $A^t = ( \pi_+ + \what{\pi}^f_-)^t =  \pi_+ +(\what{\pi}^f_-)^t$,  we need only show that
$$
(\what{\pi}^f_-)^t =  P_{\ell} R^{*t} \wg^{*t} \pi'_{-} R'^{*t}  \wf^{*t}  P_{\ell}
$$
is transversely smooth.   The operators $P_{\ell} $ and $\pi'_{-}$ are transversely smooth, and $R^{*t} = (R^*)^{-1}$  and $R'^{*t} =  (R'^{*})^{-1}$, since they are both isometries.  Now consider $\wf^{*t}$ and $\wg^{*t}$, restricted to the harmonic forms.   
Let  $\alpha' \in \Im (P'_{\ell})$ and $\alpha \in \Im( \pi_+)$.  Then 
$$
<\alpha', \wf^{*t} \alpha> \,\, = \,\,  <\wf^{*} \alpha',\alpha>  \,\, = \,\, 
Q(\wf^{*} \alpha' , \what{*} \alpha)  = \,\, Q(\wf^{*} \alpha' , \what{\tau} \alpha) \,\, = \,\, 
Q(\wf^{*} \alpha' ,  \alpha)   \,\, = \,\,
Q(\wf^{*} \alpha' , \wf^{*}  R'^* \wg^{*} R^*  \alpha) \,\, = \,\,  
$$
$$
Q'(\alpha' ,  R'^* \wg^{*} R^*  \alpha)   \,\, = \,\, 
Q'(\alpha' ,  \pi'_+ R'^* \wg^{*} R^*  \alpha +  \pi'_- R'^* \wg^{*} R^*  \alpha)   \,\, = \,\, 
Q'(\alpha' ,  \what{\tau} \pi'_+ R'^* \wg^{*} R^*  \alpha  - \what{\tau} \pi'_- R'^* \wg^{*} R^*  \alpha)    
\,\, = \,\,
$$
$$
Q'(\alpha' ,  \what{*} \pi'_+ R'^* \wg^{*} R^*  \alpha  - \what{*} \pi'_- R'^* \wg^{*} R^*  \alpha)    \,\, = \,\,
<\alpha' , ( \pi'_+ R'^* \wg^{*} R^*   - \pi'_- R'^* \wg^{*} R^*)  \alpha>.
$$
So on $\Im (\pi_+)$, $ \wf^{*t}  = \pi'_+  R'^* \wg^{*} R^*  - \pi'_-  R'^* \wg^{*} R^*$.  Similarly, on  $\Im (\pi_-)$, $ \wf^{*t}  = -\pi'_+  R'^* \wg^{*} R^*  +  \pi'_-  R'^* \wg^{*} R^*$.   Thus on $\Im (P_{\ell})$, 
$$
\wf^{*t} \,\, = \,\, (\pi'_+  R'^* \wg^{*} R^*  - \pi'_-  R'^* \wg^{*} R^*) \pi_+  -(\pi'_+ R'^* \wg^{*} R^*  - \pi'_- R'^* \wg^{*} R^*)\pi_-   \,\, = \,\,  (\pi'_+  - \pi'_-) R'^* \wg^{*} R^*( \pi_+ - \pi_-).
$$
Similarly, $\wg^{*t} \,\, = \,\,  ( \pi_+ - \pi_-)  R^* \wf^* R'^*  (\pi'_+  - \pi'_-)$.
As $(\pi'_+  - \pi'_-) \pi'_-  (\pi'_+  - \pi'_-) = \pi'_-$, $R^*$ commutes with $\pi_{\pm}$,  $R'^*$ commutes with $\pi'_{\pm}$, and $ P_{\ell}  \pi_{\pm} =  \pi_{\pm}$, we  have
$$
(\what{\pi}^f_-)^t =   ( \pi_+ - \pi_-)  \wf^*  R'^*  \pi'_{-}  \wg^* R^{*}( \pi_+ - \pi_-),
$$
which is transversely smooth.

Next, note that $Q$ is positive definite on $\Im(\pi_+)$ and 
$\Im(\what{\pi}^f_{+})$, and is negative  definite on $\Im(\pi_-)$ and 
$\Im(\what{\pi}^f_{-})$.  
So $\Im(\pi_{\pm}) \cap \Im(\what{\pi}^f_{\mp}) = \{0\}$.   Let $\alpha \in 
\Ker(\Delta^E_{\ell})$ with $A(\alpha)=0$.  Then $\pi_+(\alpha) = 
-\what{\pi}^f_{-}(\alpha)$ and $ \pi_+(\alpha), \what{\pi}^f_{-}(\alpha)  \in 
\Im(\pi_{+}) \cap \Im(\what{\pi}^f_{-}) = \{0\}$.   Thus $\alpha \in 
\Ker(\pi_+) \cap \Ker(\what{\pi}^f_{-}) \cap \Ker(\Delta^E_{\ell}) = 
\Im(\pi_{-}) \cap \Im(\what{\pi}^f_{+}) = \{0\}$, so $\alpha = 0$,  and $A$ is 
one-to-one.

Now $A(\Im(\what{\pi}^f_{+})) = {\pi_+(P_{\ell}(\Im(\wf^*_+)))=} \Im(\pi_+)$, so $ 
\Im(\pi_+) \subset \Im(A)$.  Just as $\pi_+$ maps $\Im(\what{\pi}^f_{+})$ 
isomorphically to $\Im(\pi_+)$, 
 $\pi_-$ maps $\Im(\what{\pi}^f_{-})$ isomorphically to $\Im(\pi_-)$.  
Given $\alpha \in \Im(\pi_-)$, let $\beta \in \Im(\what{\pi}^f_{-})$, with 
$\pi_-(\beta) = \alpha$, so $\beta = \pi_-(\beta) + \pi_+(\beta) = \alpha + \pi_+(\beta)$, that is $\alpha = 
\beta - \pi_+(\beta)$.  Now $A(\beta) = \pi_+(\beta) + \what{\pi}^f_{-}(\beta) = \pi_+(\beta) + 
\beta$, since $\beta  \in \Im(\what{\pi}^f_{-})$.   So $\beta \in \Im(A)$, since  
$\pi_+(\beta) \in \Im(\pi_+) \subset \Im(A)$.  Thus $\alpha = \beta - \pi_+(\beta) \in 
\Im(A)$,  and we have $ \Im(\pi_-) \subset \Im(A)$.  As $A$ is linear and 
contains $ \Im(\pi_{\pm})$, it also contains $\Im(\pi_+) \oplus  
\Im(\pi_-)  = \Ker(\Delta^E_{\ell})$, and $A$ is onto.
\end{proof}

\begin{lemma}
$A^{-1}$, the inverse of $A$ restricted to $\Ker(\Delta^E_{\ell})$, is a 
bounded isomorphism of $\Ker(\Delta^E_{\ell})$. 
\end{lemma}

\begin{proof}
 $A^{-1}$ is  bounded  if and only if there is a constant $C > 0$, so 
that  $||A(\alpha)||  \geq C$ for all $x\in M$ and all $\alpha \in \Ker(\Delta^E_{\ell})_x$ with $||\alpha|| = 
1$.  If not, there are sequences $x_j \in M$ and  $\alpha_j \in \Ker(\Delta^E_{\ell})_{x_j}$ with 
$||\alpha_j|| = 1$, and  
 $$
\lim_{j \to \infty} ||A(\alpha_j)|| = \lim_{j \to \infty} ||\pi_+(\alpha_j) + 
\what{\pi}^f_-(\alpha_j)||   = 0,
$$
that is, 
$$
0 = \lim_{j \to \infty} \pi_+(\alpha_j)+\what{\pi}^f_-(\alpha_j) =  
\lim_{j \to \infty} \pi_+(\alpha_j)+\pi_+(\what{\pi}^f_-(\alpha_j)) + \pi_-(\what{\pi}^f_-(\alpha_j)) =  
\lim_{j \to \infty} \pi_+(\alpha_j + \what{\pi}^f_-(\alpha_j)) + \pi_-(\what{\pi}^f_-(\alpha_j)).
$$
This implies that 
$\lim_{j \to \infty}   \pi_-(\what{\pi}^f_-(\alpha_j))= 0$.   Now 
$$
0 \,\, \geq \,\, Q(\what{\pi}^f_-(\alpha_j), \what{\pi}^f_-(\alpha_j)) \,\, = \,\,  ||\pi_+(\what{\pi}^f_-(\alpha_j))||^2  -  
||\pi_-(\what{\pi}^f_-(\alpha_j))||^2,
$$
so $\lim_{j \to \infty}   \pi_+(\what{\pi}^f_-(\alpha_j))= 0$, which gives that 
$\lim_{j \to \infty}  \what{\pi}^f_-(\alpha_j) = 0$, so also $\lim_{j \to 
\infty}  \pi_+(\alpha_j) = 0$.
Since $\alpha_j =  \pi_+(\alpha_j) +  \pi_-(\alpha_j)$, we have   
$$
\lim_{j \to \infty} (  \pi_-(\alpha_j)- \alpha_j ) =  0,
$$
in particular, $\lim_{j \to \infty} || \pi_-(\alpha_j)||  = \lim_{j \to \infty} 
|| \alpha_j || = 1$. 
Now $Q( \pi_-(\alpha_j),  \pi_-(\alpha_j)) = - || \pi_-(\alpha_j)||^2$,  so\\
 $\lim_{j \to \infty}  
Q(\pi_-(\alpha_j), \pi_-(\alpha_j)) = -1$.
Since $Q$ is continuous, $\lim_{j \to \infty}  Q(\alpha_j,\alpha_j) = \lim_{j \to \infty} 
Q( \pi_-(\alpha_j),\pi_-(\alpha_j) ) = -1$.   

The fact that 
$\lim_{j \to \infty}  \what{\pi}^f_-(\alpha_j) = 0$ and $\alpha_j =  
\what{\pi}^f_+(\alpha_j) +  \what{\pi}^f_-(\alpha_j)$  
implies that 
$$
\lim_{j \to \infty} ( \what{ \pi}^f_+(\alpha_j)- \alpha_j ) =  0,
$$
and as above, the fact that $Q(\what{ \pi}^f_+(\alpha_j), \what{ \pi}^f_+(\alpha_j)) \geq 0$ implies that 
$$
\liminf_j Q(\alpha_j,\alpha_j) \geq 0,
$$ which contradicts that fact that  $\lim_{j \to \infty}  Q(\alpha_j,\alpha_j) = -1$. 
\end{proof}

Now consider the map $B = A^tA$,  which is transversely smooth, and is an 
isomorphism when restricted to $\Ker(\Delta^E_{\ell})$.  Denote by $B^{-1}$ 
the composition of maps:
$$
B^{-1}:\cA^{\ell}_{(2)} (F_s,E)   
\stackrel{P_{\ell}}{\longrightarrow}      \Ker(\Delta^E_{\ell})    
\stackrel{B^{-1}_{\ell}}{\longrightarrow}     \Ker(\Delta^E_{\ell}),
$$
where $B^{-1}_{\ell}$ is the inverse of $B$ restricted to 
$\Ker(\Delta^E_{\ell})$.  
Since $\rho_+$ takes values in $P_{\ell}(\Im(\wf^*_+ )) =  \Im(\what{ 
\pi}^f_+)$,
$A  \rho_+ = \pi_+$, so $B  \rho_+ = A^t \pi_+$, and 
$\rho_+ = B^{-1}  A^t \pi_+$.   Thus we are reduced to showing that 
$B^{-1}$ is transversely smooth.

Restricting once again to $\Ker(\Delta^E_{\ell})$, we have that the operator 
$B$ is positive, and $A$ and $A^{-1}$ are bounded operators, so there are 
constants $0 < C_0 < C_1 < \infty$ so that for all $\alpha \in  
\Ker(\Delta^E_{\ell})$,  $\alpha \neq 0$, 
$$
C_0 \quad  \leq \quad    \frac{<B\alpha,\alpha> }{<\alpha,\alpha>}\quad  \leq  \quad  C_1.
$$
Thus the spectrum of $B$ on $\Ker(\Delta^E_{\ell})$, $\sigma(B) \subset 
[C_0,C_1]$, and for $\lambda > 0$,
 $\sigma(\frac{B}{\lambda}) \subset 
[\frac{C_0}{\lambda},\frac{C_1}{\lambda}]$, and 
 $\sigma(P_{\ell}-\frac{B}{\lambda}) \subset 
[1-\frac{C_1}{\lambda},1-\frac{C_0}{\lambda}]$.
In particular, for $\lambda > C_1$ we have 
$$
0 \,\, <  \,\, 1-\frac{C_1}{\lambda} \,\, \leq \,\, ||P_{\ell}-\frac{B}{\lambda}|| \,\, \leq \,\, 
1-\frac{C_0}{\lambda} \,\, < \,\, 1.
$$
Since $B =  P_{\ell}  B  P_{\ell}$, this estimate actually holds on 
$\cA^{\ell}_{(2)} (F_s,E)$, and for all Sobolev norms associated to 
$\cA^{\ell}_{(2)}(F_s,E)$.
This estimate along with the fact that $x^{-1} = \frac{1}{\lambda}\sum_{n 
= o}^{\infty} (1-\frac{x}{\lambda})^n$, provided that 
$|1-\frac{x}{\lambda}| < 1$, implies that for $\lambda > C_1$,
$$
B^{-1}  \,\, = \,\,  \frac{1}{\lambda}  \sum_{n=0}^{\infty}  \Bigl{(}P_{\ell} - 
\frac{B}{\lambda} \Bigr{)}^n,
$$
where we set $ \Bigl{(}P_{\ell} - \frac{B}{\lambda} \Bigr{)}^0 = 
P_{\ell}$. 
For $N \in \Z_+$, set
$$
D_N \,\, = \,\,  \frac{1}{\lambda}  \sum_{n=0}^{N}  \Bigl{(}P_{\ell}- 
\frac{B}{\lambda} \Bigr{)}^n,
$$
where again $ \Bigl{(}P_{\ell} - \frac{B}{\lambda} \Bigr{)}^0 = 
P_{\ell}$.   
Then $D_N$ is a uniformly bounded  (over all $N$)  transversely smooth 
operator, and it converges to $B^{-1}$ in all Sobolev norms.  Thus 
$B^{-1}$ is a bounded leafwise smoothing operator.

Let $Y$ be a vector field on $M$, and  consider $\pa^Y_{\nu}D_N = 
\frac{1}{\lambda}  \sum_{n=0}^{N} \pa^Y_{\nu} \Bigl{(}\Bigl{(}P_{\ell}- 
\frac{B}{\lambda} \Bigr{)}^n\Bigr{)}$.  
For any integers $k_1, k_2$, and for $N > 1$,
$$    
||\frac{1}{\lambda}  \sum_{n=N+1}^{\infty} \pa^Y_{\nu} \Bigl{(} 
\Bigl{(}P_{\ell}- \frac{B}{\lambda} \Bigr{)}^n\Bigr{)}||_{k_1, k_2}
\quad \leq \quad
\frac{1}{\lambda}  \sum_{n=N+1}^{\infty} ||\pa^Y_{\nu} \Bigl{(} 
\Bigl{(}P_{\ell}- \frac{B}{\lambda}\Bigr{)}^n \Bigr{)}||_{k_1, k_2}
\quad \leq \quad
$$
$$
\frac{1}{\lambda}
\sum_{n=N+1}^{\infty}
\sum_{r =0}^{n-1} 
||P_{\ell}- \frac{B}{\lambda} ||_{k_1, k_1}^r  
||  \pa^Y_{\nu}\Bigl{(}P_{\ell}- \frac{B}{\lambda} \Bigr{)}||_{k_1,k_2}  
||P_{\ell}- \frac{B}{\lambda}||_{k_2,k_2}^{n-r-1}
\quad = \quad
$$
$$
\frac{1}{\lambda} ||  \pa^Y_{\nu}\Bigl{(}P_{\ell}- \frac{B}{\lambda} 
\Bigr{)}||_{k_1,k_2}
\sum_{n=N+1}^{\infty}n ||P_{\ell}- \frac{B}{\lambda}||^{n-1}
\quad \leq \quad
\frac{1}{\lambda} ||  \pa^Y_{\nu}\Bigl{(}P_{\ell}- \frac{B}{\lambda} 
\Bigr{)}||_{k_1,k_2}
\sum_{n=N+1}^{\infty} n\Bigl{(}1- \frac{C_0}{\lambda} \Bigr{)}^{n-1}.
$$
This converges to $0$ as $N \to \infty$, as $||  
\pa^Y_{\nu}\Bigl{(}P_{\ell}- \frac{B}{\lambda} \Bigr{)}||_{k_1,k_2}$ is 
finite since $P_{\ell}- \frac{B}{\lambda} $ is transversely smooth.  Thus 
the transverse derivative $\pa^Y_{\nu}D_N $ converges in all Sobolev 
norms, so $\lim_{N \to \infty}  \pa^Y_{\nu}D_N$ exists, and it is bounded 
and leafwise smoothing.    

\begin{proposition} $\pa^Y_{\nu}B^{-1}$ exists, in particular, 
$\pa^Y_{\nu}D_N$ converges in all Sobolev norms to $\pa^Y_{\nu}B^{-1}$, so 
$\pa^Y_{\nu}B^{-1}$ is a bounded leafwise smoothing operator.
\end{proposition}

\begin{proof}
As $\pa^Y_{\nu}D_N$ converges in all Sobolev norms, we only need prove that $\pa^Y_{\nu}B^{-1}$ exists and that it equals
$\lim_{N \to \infty}  \pa^Y_{\nu}D_N.$

Recall the situation in the proof of Theorem \ref{GRthm}.  For $y$ close to $x$ in $M$, we have
the smooth diffeomorphism $\Phi_y:\wL_x \to \wL_y$.
Given $Y \in TM_x$, set $\gamma(t) = \exp_x (tY)$.    For  $z \in \wL_x$ and  $t$ sufficiently small, say $|t| \leq \epsilon$, we have the path  $t \to \what{\gamma}_z(t)$,  which covers $\gamma(t)$ and has tangent vector in $\nu_s$.
So for $|t| \leq \epsilon$,  the diffeomorphism $\Phi_{\gamma(t)}:\wL_x \to \wL_{\gamma(t)}$ exists.
The vector  $Y$ defines the transverse vector field $\what{Y}$ along $\wL_x$, i.\ e.\ a smooth section of $\nu_s \, | \, \wL_x$, by requiring $s_*(\what{Y}) = Y$. 
Then, the  operator $\pa^Y_{\nu}(\cdot)  =  [\nabla^{\nu}_{\what{Y}}, \cdot]$ 
can be realized as  $\pa/\pa t(\cdot)$ as follows.
We may parallel translate  all 
objects on  $\wL_x$ to $\wL_{\gamma(t)}$ (and vice-versa) along the paths
$\what{\gamma}_z(t)$, using the connection $\nabla$.  We will 
denote this parallel translation by $\Phi_t$ (and the reverse by 
$\Phi_{t}^{-1}$).
Thus any section of $\xi \in C^{\infty}_c(\wL_x; \wedge^{\ell}T^*F_s \otimes E)$ defines a section $\Phi_t(\xi) = \xi_{t}$ of $C^{\infty}_c(\wL_{\gamma(t)}; \wedge^{\ell}T^*F_s \otimes E)$ given by 
$$
\xi_{t}(z) \,\, = \,\,  \Phi_{t}( \xi  (\Phi^{-1}_{\gamma(t)}(z))),
$$
and $\xi_{t}$ is smooth in $t$.
Note that for such a local section, $\nabla_{\what{Y}} \xi_t = 
\nabla^{\nu}_{\what{Y}} \xi_t$ (as $\what{Y} \in \nu_s$) is defined and 
equals $0$, since $\xi_t$ is parallel translation along integral curves 
of $\what{Y}$ for the connection $\nabla$.  In fact, if we set $Y(t) = 
\gamma'(t)$,  then,
$\nabla_{\what{Y}(t)} \xi_t = \nabla^{\nu}_{\what{Y}(t)} \xi_t = 0$.  
Further note that $\Phi_{\gamma(t)}$ is a diffeomorphism of 
bounded dilation and the induced action on $E$ is also bounded, so the local operators $\Phi_t$ and $\Phi^{-1}_t$ are 
bounded when acting on sections of $C^{\infty}_c(\wL_x; 
\wedge^{\ell}T^*F_s \otimes E)$, (respectively $C^{\infty}_c(\wL_{\gamma(t)}; 
\wedge^{\ell}T^*F_s \otimes E)$.   We denote these bounds by $||\Phi_t||$ and $||\Phi^{-1}_t ||$ respectively.    The bounds are uniform in $t$ for $|t| \leq \epsilon$.

Similarly, we may  parallel translate operators such as $D_N$ from 
nearby leaves to $\wL_x$ as follows.  Given $\xi_1, \xi_2 \in 
C^{\infty}_c(\wL_x; \wedge^{\ell}T^*F_s \otimes E)$, define the operator 
$D_{N,t}$ on $\wL_x$ by
$$
<D_{N,t} (\xi_1) ,\xi_2>  \,\, = \,\,   <  \Phi^{-1}_t\Bigl{(}D_{N,\gamma(t)} 
(\xi_{1,t})\Bigr{)} ,\xi_2>.
$$ 
This is well defined and smooth in $t$ provided $|t| \leq \epsilon$.   Thus,
the operator $\dd  \frac{\pa(D_{N,t})}{\pa t} \, 
\Big{|}_{t = 0}$ is well defined  as a map from $C^{\infty}_c(\wL_{x}; 
\wedge^{\ell}T^*F_s \otimes E)$ to 
$C^{\infty}(\wL_{x}; \wedge^{\ell}T^*F_s \otimes E)$.   
Likewise, $\nabla_{\what{Y}}( D_{N, \gamma(t)}(\xi_{t}))$ is well defined 
for all $\xi \in C^{\infty}_c(\wL_{x}; \wedge^{\ell}T^*F_s \otimes E)$, and takes 
values in $C^{\infty}(\wL_{x}; \wedge^{\ell}T^*F_s \otimes E)$.
The fundamental relationship between parallel translation and the 
connection $\nabla$ translates to the equation
\begin{Equation}\label{contrans}
\hspace{1in}
 $\dd
\Bigl{(} \frac{\pa(D_{N,t})}{\pa t} \, \Big{|}_{t = 0} \Bigr{)} (\xi)   \,\, = \,\,
\nabla_{\what{Y}}( D_{N, \gamma(t)}(\xi_{t})).$
\end{Equation}
\noindent
In fact, for all $t_0 \in [-\epsilon,  \epsilon]$, 
$$
\Bigl{(} \frac{\pa(D_{N,t})}{\pa t}  \, \Big{|}_{t = t_0}\Bigr{)} (\xi)   \,\, = \,\, 
\Phi^{-1}_{t_0} \Bigl{(}\nabla_{\what{Y}(t_0)}( D_{N, \gamma(t)}(\xi_{t})) \Bigr{)},
$$
since $\Phi^{-1}_{t}  = \Phi^{-1}_{t_0}  \circ \Phi^{-1}_{t, t_0}$, where 
$\Phi^{-1}_{t, t_0}$  is parallel translation from $\wL_{\gamma(t)}$ to 
$\wL_{\gamma(t_0)}$.

For $\xi \in C^{\infty}_c(\wL_{x}; \wedge^{\ell}T^*F_s \otimes E)$ we have 
$$
 \pa^Y_{\nu} D_N \xi   \,\, = \, \, [\nabla^{\nu}_{\what{Y}}, D_N]\xi\,\, 
= \, \,  [\nabla_{\what{Y}}, D_{N}] \xi
\,\, =  \,\, \nabla_{\what{Y}} D_{N, \gamma(t)}(\xi_t)-  D_{N} 
\nabla_{\what{Y}}(\xi_t)\,\, =  \,\, \nabla_{\what{Y}} D_{N, 
\gamma(t)}(\xi_t),
$$
since $ \nabla_{\what{Y}} (\xi_t) = 0$.   So by Equation \ref{contrans} we 
have 
$$
\pa^Y_{\nu}D_N \,\, = \,\, \frac{\pa(D_{N,t})}{\pa t} \, \Big{|}_{t = 0}.
$$ 
As above, this extends to 
\begin{Equation}\label{gencon}
\hspace{1in}
$\dd     \frac{\pa(D_{N,t})}{\pa t}  
\,\, = \,\,   
\Phi^{-1}_{t} \Bigl{(}\pa^{Y(t)}_{\nu} D_{N, \gamma(t)} \Bigr{)}  
$,
\end{Equation}
\noindent 
provided $|t| \leq \epsilon$.

For $t \in [-\epsilon,\epsilon]$, set  $D'_{N,t} =  \Phi^{-1}_t 
\Bigl{(}\pa^{Y(t)}_{\nu}D_N \Bigr{)}  =  {\pa(D_{N,t})}/{\pa t}$, and
$D'_t   =  \Phi^{-1}_t \Bigl{(}\lim_{N \to \infty}\pa^{Y(t)}_{\nu} D_N \Bigr{)}$, and
$B^{-1}_t = \Phi^{-1}_t(B^{-1})$.      Note carefully that the 
following computation takes place on the leaf $\wL_x$.   For
$\xi_1, \xi_2 \in C^{\infty}_c(\wL_x; \wedge^{\ell}T^*F_s \otimes E)$, and   $h 
\in (0,  \epsilon)$, we have that
$$
\Bigl{|}  <B^{-1}_{h}(\xi_{1}),\xi_{2}> - <B^{-1}_0 (\xi_1) ,\xi_2> - 
\int_0^{h} <D'_t (\xi_{1}) ,\xi_{2}> dt \Bigr{|} \,\,  \leq
$$
$$
\Bigl{|}  <B^{-1}_{h}(\xi_{1}),\xi_{2}> -  <D_{N, h}(\xi_{1}),\xi_{2}> 
\Bigr{|} \,\,  +
$$
$$
\Bigl{|}  <D_{N,h}(\xi_{1}),\xi_{2}> - <D_{N,0} (\xi_1) ,\xi_2> - 
\int_0^{h} < D_{N,t}'(\xi_{1}) ,\xi_{2}> dt \Bigr{|} \,\, +
$$
$$
\Bigl{|}  <D_{N,0}(\xi_1),\xi_2> - <B^{-1}_0 (\xi_1) ,\xi_2> \Bigr{|} \,\, + \,\,
\Bigl{|} \int_0^{h} <(D_{N,t}' - D'_t)(\xi_{1}) ,\xi_{2}> dt\Bigr{|} .
$$
The term 
$$
 <D_{N,h}(\xi_{1}),\xi_{2}> - <D_{N,0} (\xi_1) ,\xi_2> - \int_0^{h} < 
D_{N,t}'(\xi_{1}) ,\xi_{2}> dt  \,\, = \,\, 0,
$$
since $D_{N,t}' = \pa(D_{N,t})/\pa t$.
The term
$$
\Bigl{|}  <D_{N, 0}(\xi_1),\xi_2> -  <B^{-1}_{0}(\xi_1),\xi_2> \Bigr{|} 
\,\, \leq \,\, 
||D_{N, 0} - B^{-1}_{0}  || \, ||\xi_{1}|| \, ||\xi_{2}||,
$$
which goes to $0$ as $N \to \infty$, since $D_N \to B^{-1}$ in norm.
Likewise, the term
$$
\Bigl{|}  <B^{-1}_{h}(\xi_{1}),\xi_{2}> -  <D_{N, h}(\xi_{1}),\xi_{2}> 
\Bigr{|}    \,\, = \,\, 
\Bigl{|}  <\Phi^{-1}_h  (B^{-1}_{\gamma(h)}-D_{N, \gamma(h)})\Phi_h 
(\xi_{1}),\xi_{2}> \Bigr{|} \,\,  \leq
$$
$$ 
|| \Phi^{-1}_h || \, ||B^{-1}_{\gamma(h)}-  D_{N, \gamma(h)}|| 
\, || \Phi_h || \, ||\xi_{1}|| \, ||\xi_{2}||,
$$
which goes to $0$ as $N \to \infty$, since $D_N \to B^{-1}$ in norm and $|| \Phi^{-1}_h ||$ and $|| \Phi_h ||$ are bounded.
  
The term 
$$
\Bigl{|} \int_0^{h} <(D_{N,t}' - D'_t)(\xi_{1}) ,\xi_{2}> dt \Bigr{|}  \,\,  \leq 
$$
$$
{ \int_0^{h} || \Phi^{-1}_t || 
\,||\partial^{Y(t)}_{\nu}D_{N,\gamma(t)} - \lim_{\what{N} \to \infty} 
\partial^{Y(t)}_{\nu}D_{\what{N},\gamma(t)}||    \, || \Phi_t || \, 
||\xi_{1}|| \, ||\xi_{2}|| dt}, 
$$
which goes to $0$ as $N \to \infty$, since {$|| \Phi^{-1}_t ||$ 
and $|| \Phi_t ||$}  are uniformly bounded for $t \in [0,h]$, and 
$\partial^Y_{\nu}D_{N}$ converges in norm.

Thus 
$$
\lim_{N \to \infty}\Bigl{|}  <B^{-1}_{h}(\xi_{1}),\xi_{2}> - <B^{-1}_0 (\xi_1) ,\xi_2> - 
\int_0^{h} <D'_t (\xi_{1}) ,\xi_{2}> dt \Bigr{|}  \,\, = \,\, 0,
$$
and as the expression inside the limit is independent of $N$, it actually equals $0$.  This implies that
$$
<\lim_{h \to 0}   \frac{1}{h}\Bigl{(} B^{-1}_{h}- B^{-1}_0 - \int_0^{h} 
D'_t dt\Bigr{)}  (\xi_1) ,\xi_2> \,\, = \,\, 0,
$$
for all  $\xi_1, \xi_2 \in C^{\infty}_c(\wL_x; \wedge^{\ell}T^*F_s \otimes E)$, 
so
$$
\lim_{h \to 0}  \frac{1}{h}\Bigl{(} B^{-1}_{h}- B^{-1}_0 - \int_0^{h} D'_t 
dt\Bigr{)} \,\,= \,\,0
$$
 as a map from $C^{\infty}_c(\wL_{x}; \wedge^{\ell}T^*F_s \otimes E)$ to 
$C^{\infty}(\wL_{x}; \wedge^{\ell}T^*F_s \otimes E)$.

Next we have
$$
\Bigl{|}  < \lim_{t \to 0}  D'_t  (\xi_1) ,\xi_2>  - < D'_0 (\xi_1) 
,\xi_2> \Bigr{|}  \,\,  \leq
$$
$$
 \Bigl{|} < \lim_{t \to 0}(D'_t  - D'_{N, t})(\xi_1) ,\xi_2> \Bigr{|}  \,\, + \,\,  
\Bigl{|} < (\lim_{t \to 0} D'_{N, t} - D'_{N, 0} ) (\xi_1) ,\xi_2> 
\Bigr{|}  \,\, + \,\,  
\Bigl{|} <(D'_{N, 0}  - D'_0 ) (\xi_1) ,\xi_2>\Bigr{|} \,\, \leq \,\,
$$
$$
\lim_{t \to 0}||\Phi^{-1}_t || \, ||\lim_{\what{N} \to 
\infty}\pa^{Y(t)}_{\nu}D_{\what{N},\gamma(t)} - \pa^{Y(t)}_{\nu} 
D_{N,\gamma(t)} ||   \, ||\Phi_t || \, \, ||\xi_1|| \, ||\xi_2||  \,\, + \,\, 
$$
$$ 
\Bigl{|}  \lim_{t \to 0}<D'_{N, t}(\xi_1) ,\xi_2>  - <  D'_{N, 0}  (\xi_1) 
,\xi_2>\Bigr{|}  \,\, + \,\, 
||\pa^{Y(0)}_{\nu}D_{N,\gamma(0)} -\lim_{\what{N} \to 
\infty}\pa^{Y(0)}_{\nu}D_{\what{N},\gamma(0)}||   \, || \xi_1|| \, 
||\xi_2||.
$$
The first and last terms can be made arbitrarily small (for $t \in [0,h]$) 
by choosing $N$ sufficiently large.  The middle term equals zero since 
$<D'_{N, t}(\xi_1) ,\xi_2> $ is  continuous in $t$, which follows 
immediately from Equation \ref{gencon} and the fact that $D_N$ is 
transversely smooth.  Thus,
$$
0 \,\, = \,\,  <\lim_{t \to 0}  D'_t  (\xi_1) ,\xi_2>  - < D'_0 (\xi_1) 
,\xi_2> \,\, = \,\, <(\lim_{t \to 0}  D'_t  -D'_0) (\xi_1) ,\xi_2> ,
$$
which holds for all $\xi_1, \xi_2 \in C^{\infty}_c(\wL_x; 
\wedge^{\ell}T^*F_s \otimes E)$, so  $\lim_{t \to 0}  D'_t  - D'_0 \,\, = \,\,  0,$ that is $D'_t$ is continuous at zero.
The operator $\dd \lim_{h \to 0}  \frac{1}{h}\Bigl{(}  \int_0^{h} D'_t 
dt\Bigr{)}$ 
is also well defined as a map from $C^{\infty}_c(\wL_{x}; 
\wedge^{\ell}T^*F_s \otimes E )$ to $C^{\infty}(\wL_{x}; \wedge^{\ell}T^*F_s \otimes E)$, and 
as $D'_t$ is continuous at zero, we have
$$
\lim_{h \to 0}  \frac{1}{h}\Bigl{(}  \int_0^{h} D'_t dt\Bigr{)}   \,\, = 
\,\, D'_0. 
$$
Again by the fundamental relationship between parallel translation and 
$\nabla$, we have  
$$
 \lim_{h \to 0} \frac{B^{-1}_{h}- B^{-1}_0}{h}  \,\, = \,\,  \pa^Y_{\nu}B^{-1},
$$
so
$$
 \pa^Y_{\nu}B^{-1}  \,\,=\,\,   \lim_{h \to 0} \frac{B^{-1}_{h}- 
B^{-1}_0}{h}   \,\,=\,\, 
 \lim_{h \to 0}  \frac{1}{h}\Bigl{(}  \int_0^{h} D'_t dt\Bigr{)}   \,\, = 
\,\,
 D'_0 \,\,=\,\, \lim_{N \to \infty}  \pa^Y_{\nu}D_{N},
$$
and  $\pa^Y_{\nu}B^{-1}$ is a bounded leafwise smoothing operator.
\end{proof}
 
A boot strapping argument now finishes the proof.  Let $Y_1, Y_2$ be 
vector fields on $M$.   As $B^{-1} B = P_{\ell}$ and the $\pa^{Y_i}_{\nu}$ 
are derivations, we have
$$
(\pa^{Y_2}_{\nu} B^{-1}) B \,\, + \,\,  B^{-1}  (\pa^{Y_2}_{\nu}B) \,\, = \,\, 
\pa^{Y_2}_{\nu}P_{\ell},
$$
so
$$
\pa^{Y_2}_{\nu} B^{-1}  \,\,  = \,\, -  B^{-1}  (\pa^{Y_2}_{\nu}B)B^{-1} \,\, + \,\,
(\pa^{Y_2}_{\nu}P_{\ell})B^{-1},
$$
which is in the domain of $\pa^{Y_1}_{\nu}$.   Applying it, we obtain
$$
\pa^{Y_1}_{\nu} \pa^{Y_2}_{\nu} B^{-1} \,\,  =  \,\,
 -  \Big{(}(\pa^{Y_1}_{\nu}B^{-1})  (\pa^{Y_2}_{\nu}B)B^{-1} \,\, + \,\,
  B^{-1}  (\pa^{Y_1}_{\nu}\pa^{Y_2}_{\nu}B)B^{-1} \,\, + \,\, 
 B^{-1}  (\pa^{Y_2}_{\nu}B)(\pa^{Y_1}_{\nu}B^{-1})\Big{)} \,\, +
$$
$$
(\pa^{Y_1}_{\nu}\pa^{Y_2}_{\nu}P_{\ell})B^{-1} \,\, + \,\, 
(\pa^{Y_2}_{\nu}P_{\ell})(\pa^{Y_1}_{\nu} B^{-1}),
$$
which is a bounded leafwise smoothing map, since $B$ and $P$ are 
transversely smooth and $\pa^{Y_1}_{\nu}B^{-1}$ is bounded and leafwise 
smoothing.  
Proceeding by induction, we have that for all vector fields $Y_1, ..., 
Y_{m}$ on $M$, the operator  $\pa^{Y_1}_{\nu} \cdots 
\pa^{Y_{m}}_{\nu} B^{-1}$ is bounded  and leafwise smoothing, so 
$B^{-1}$ is transversely smooth.

This completes the proof Theorem \ref{secondeq}.
\end{proof}

Finally, we prove Theorem \ref{main}, that is we prove

\begin{theorem}\label{firsteq}
$\ch_a(\pi^f_{\pm})  =  f^*(\ch_a( \pi'_{\pm}))$.
\end{theorem}

\begin{proof}    We will only prove that {$\ch_a( \pi^f_+) = f^*(\ch_a( \pi'_+))$}, as the other proof is the same.  We begin by constructing special covers of $M$ and $M'$.  Let $\{ \what{U}'\}$ be a finite open cover of $M'$ by foliation charts with transversals $\what{T}'$.    Choose the $\what{U}'$  so small  that $g \, | \, \what{T}'$ is a diffeomorphism.  Denote by $ \rho'_{\what{U}'}: \what{U}' \to \what{T}'$ the projection.  Let $\{U\}$ be a finite open cover of $M$ by foliation charts with 
transversals $T$.  Since the collection of open sets $f^{-1}(\what{U}')$ cover $M$, we may 
choose the  $U$ small enough so that for each $U$, there is a $\what{U}'_U$ with $f(U) 
\subset \what{U}'_{U}$.  We may further assume that the $U$ are so small that $f \, | \, T$ is 
a diffeomorphism.  Set 
$$
U' =  (\rho'_{\what{U}'_U})^{ -1}( \rho'_{\what{U}'_U}(f(U))).
$$
Then the set $\{U'\}$ is a finite open cover of $M'$ by foliation charts, $f(U) \subset U'$, and $T' = f(T)$ is a transversal of $U'$.  Denote the projection $ \rho'_{\what{U}'_U} \, | \,  U' \to T'$ by $\rho'$.

Set $V = f^{-1}(U')$, and note that $V$ is not 
necessarily connected.   However, $V \supset U$ whose transversal $T$ is 
taken diffeomorphically  onto $T'$ by $f$.   There is a well defined 
projection  $\rho:V \to T$, given by $\rho = (f \, | \, T)^{-1} \circ \rho 
' \circ f$.  Recall the connection $\nabla$ on $\pi^f_+$, (induced 
from the connection ${{\nabla'}}$ on $\pi'_+$) which 
we will use to construct  $\ch_a(\pi^f_+)$, and set $\nabla^T = \nabla 
\,|\, T$ with curvature $\theta_T$.  Then just as in Proposition  
\ref{pullbacks}, we have 
\begin{lemma}\label{locpullbacks}
$\nabla \, | \, V = \rho^*(\nabla^{T} ) \quad \text{  and  }  \quad \theta 
\, | \, V = \rho^*(\theta_{T}).
$
\end{lemma}

\begin{proof}
The proof is essentially the same.  To effect it, we need to be able to 
define local invariant sections over $V$, and to do this, we need families 
of leafwise paths, such that moving along them gives the projection 
$\rho$.    Given $y \in V$,  choose a leafwise path $\gamma'_y:[1,2] \to 
U'$ from  ${\rho'}(f(y))$  to $f(y)$.   Let $h:M \times I \to M$ be a 
leafwise homotopy between the identity map and $g \circ f$.   In 
particular, $h(x,0) = x$ and $h(x,1) = gf(x)$.  Define the leafwise path 
$\gamma_y$ from $\rho(y)$ to $y$ as follows:  
$$
\gamma_y (t) =  h(\rho(y),t)            \text{   for  } 0 \leq t \leq1; 
\quad
\gamma_y (t) =  g(\gamma'_y(t))  \text{   for  } 1 \leq t \leq 2; \quad 
 \text {and} \quad 
\gamma_y (t) =  h(y,3-t)        \text{   for  }  2  \leq t \leq 3.
$$
Since $f(\rho(y)) = {\rho'}(f(y))$, this does give a path from $\rho(y)$ 
to $y$.  
Using the $\gamma_y$, we may extend any local section defined on $T$ to a 
local invariant section on all of $V$, and then proceed just as in the 
proof of  Proposition  \ref{pullbacks}.
\end{proof}

The connection $\nabla^{T'} $ (which is $\nabla'$ restricted to  $\pi'_+ \, | \ T'$), and 
its curvature $\theta_{T'}$ satisfy
${{\nabla'}}  \, | \, U' = {\rho'}^*(\nabla^{T'} )$   and 
${\theta'} \, | \, U' = {\rho'}^*(\theta_{T'})$.
Set $\hf = f\,|\,T$, and  define $\hf^*(\nabla^{T'})$ and
$\hf^*(\theta_{T'})$ as follows.
Let $\xi \in C^{\infty}(\pi^f_+ \, | \, T)$, and suppose that $X$ and $Y$ are 
tangent to $T$.  
Set $X' =  \hf_*(X) = f_*(X)$, and $Y' =   \hf_*(Y) = f_*(Y)$, both of 
which are tangent to $T'$.  Define
$$
\hf^*(\nabla^{T'})_X\xi = \wf^*(\nabla^{T'}_{X'}(\wf^{-*}\xi \, | \, T'))  
\text{  and  }  
\Big{(} \hf^*(\theta_{T'})(X,Y)\Big{)} \xi  = \wf^* (\theta_{T'}  (X',Y') 
(\wf^{-*}\xi \, | \, T')).
$$

\begin{lemma}\label{induced}
$\hf^*(\nabla^{T'}) = \nabla^{T}$ and  $\hf^*(\theta_{T'}) = \theta_{T}$.
\end{lemma}

\begin{proof}
The element $\xi \in C^{\infty}(\pi^f_+ \, | \, T) $ determines the local  invariant 
sections $\widetilde{\xi}$ of  $ \pi^f_+ $ and  $\wf^{-*} \xi$ of 
$\pi'_+$.    Then
$$
\hf^*(\nabla^{T'})_X \xi   = \wf^*(\nabla^{T'}_{X'}(\wf^{-*}\xi \, | \, 
T')) = \wf^*({{\nabla'}}_{X'}\wf^{-*}\xi)
=  \nabla_{X}\widetilde{\xi} = \nabla^T_{X}\xi.
$$
Next, using  local spanning sets of $\pi^f_+  \,| \, V$, and $\pi'_+ \,| \, U'$
it is not difficult to show that
$$
\theta_{T}(X,Y) = \nabla^T_X \nabla^T_Y - \nabla^T_Y \nabla^T_X -  
\nabla^T_{[X,Y]}.
$$
and similarly for $\theta_{T'}(X',Y')$.    Then
$$
\nabla^T_X \nabla^T_Y \xi =  
\wf^*\nabla^{T'}_{X'}\wf^{-*}   \wf^*\nabla^{T'}_{Y'}\wf^{-*}\xi   = 
\wf^*\nabla^{T'}_{X'}\nabla^{T'}_{Y'}\wf^{-*}\xi 
$$
and $\nabla^T_Y \nabla^T_X \xi   =   
\wf^*\nabla^{T'}_{Y'}\nabla^{T'}_{X'}\wf^{-*}\xi$.
Since $\hf$ is a diffeomorphism,  $\hf_*([X,Y]) = [X',Y']$, so
$$
 \nabla^T_{[X,Y]}\xi =   \wf^*\nabla^{T'}_{[X',Y']}\wf^{-*}\xi.
$$ 
It follows immediately that 
$$
\Big{(}\hf^*(\theta_{T'})(X,Y) \Big{)}\xi =   
\wf^*\theta_{T'}(X',Y')\wf^{-*}\xi  =  \theta_T (X,Y)\xi.  
$$   
\end{proof} 

Now consider the curvature operator ${\theta'}$ of ${{\nabla'}}$ over $U'$.   We may assume that $U' \simeq \R^p \times \R^q$ with coordinates  $x'_1, ...,x'_n$, and that $T' = \{0\} \times \R^q$.  
Choose a local invariant spanning set $\{\xi'_i\}$ of 
$\pi'_+ \, | \, U'$.  Recall that for $\alpha'_1 \otimes \phi'_1, \alpha'_2 \otimes \phi'_2$ sections of $\wedge T^*\wL' \otimes E'$, 
$$ 
Q'( \alpha'_1 \otimes \phi'_1,\alpha'_2 \otimes \phi'_2 )  \,\, = \,\,  
\int_{\wL'} \{\phi'_1, \phi'_2\} \alpha'_1 \wedge \alpha'_2  \,\, = \,\,  
\int_{\wL'}( \alpha'_1 \otimes \phi'_1) \wedge (\alpha'_2 \otimes \phi'_2).  
$$
There are functions $a'_{i,j,k,l}$ on $T'$ (thanks to Proposition 
\ref{pullbacks}) so that the action of ${\theta'}$ on a section $\xi'$ of $\pi'_+$ is given by 
$$
{\theta'}(\xi') \,\, = \,\,\sum_{k,l = p+1}^n \sum_{i,j} a'_{i,j,k,l} \,Q'( \xi'_j,  \xi')  \, \xi'_i  \, dx'_k \wedge dx'_l  \,\, = \,\,
\sum_{i,j,k,l} a'_{i,j,k,l} \, \Big[ \int_{\wL'} \xi'_j  \wedge   \xi' \Big] \xi'_i  \, dx'_k \wedge dx'_l.$$
The reason that we can represent ${\theta'}$ this way is because for any $\xi' \in \Ker(\pi'_+)$ and any $\what{\xi}'  \in \Im(\pi'_+)$,    $Q'( \xi',  \what{\xi}') = 0$. This follows from the facts that $< \xi',  \what{\xi}'> \, =  \,0$,  $Q'( \xi', \what{*} \, \what{\xi}') \, = \, < \xi',  \what{\xi}'>$, and $\what{\xi}' = \what{\tau} \what{\xi}'  =  \sqrt{-1}^{\ell^2}\what{*} \, \what{\xi}'$.

Let  $x' \in U'$ and $y',z' \in 
\wL_{x'}$.  With respect to the spanning set  $\{\xi'_i\}$ and the local 
coordinates on $U'$,  the Schwartz kernel $\Theta_{x'} 
'(y',z')$ of ${\theta'} \, | \, U'$ is given by 
$$
\Theta_{x'} '(y',z') = \sum_{k,l = p+1}^n \sum_{i,j} a'_{i,j,k,l}(\rho'(x')) 
\xi'_i(y') \otimes \xi'_j(z') dx'_k \wedge dx'_l.
$$
We write this more succinctly as 
$$
\Theta' \, | \, U'= \sum_{i,j,k,l} a'_{i,j,k,l}\, \xi'_i \otimes \xi'_j \,   dx'_k \wedge dx'_l.
$$
Recall that  $\overline{x}' \in \wL_{x'}$ is the class of the constant path at 
$x'$,  that we identify $M'$ with its image under $x' \to \overline{x}'$,
and that $\dd \int_{U'} $ is integration over the fibration $U' \to 
T'$.  Let $\{\psi'_{U'}\}$ be a partition of unity  subordinate to the 
special cover  $\{U'\}$  of $M'$.    Then
$$
\Tr({\theta'}) \, | \, T' =  \int_{U'}  
\psi'_{U'}(x')\sum_{i,j,k,l} a'_{i,j,k,l}(\rho'(x')) \xi'_i(\overline{x}')\wedge 
\xi'_j(\overline{x}')  \,dx'_k \wedge dx'_l.
$$
Note that we do not multiply  the integrand by the leafwise volume form $dx'$, since this is already incorporated in it by our use of the leafwise 
differential forms $\xi_i'$ in the Schwartz kernel $\Theta'$ of $\theta'$.  In particular, being very precise,
$$
\Theta'_{x'}(y',z') \,\, = \,\,  \sum_{i,j,k,l}  
a'_{i,j,k,l}(\rho'(x')) \xi'_i(y') \otimes i_{\vol(z')} [\xi'_j(z')\wedge (\cdot)]  
\, dx'_k \wedge dx'_l,
$$
where $\vol(z')$ is the oriented unit length vector in $(\wedge^{2\ell} TF_s)_{z'}$.
Then 
$$
\tr(\Theta'_{x'}(\overline{x}',\overline{x}')) \, dx' \,\, = \,\, \sum_{i,j,k,l} a'_{i,j,k,l}(\rho'(x'))(i_{\vol(\overline{x}')} [ \xi'_i(\overline{x}') \wedge \xi'_j(\overline{x}') ] ) dx' \,  dx'_k \wedge dx'_l  \,\, = \,\, 
$$
$$
\sum_{i,j,k,l} a'_{i,j,k,l}(\rho'(x')) \xi'_i(\overline{x}') \wedge \xi'_j(\overline{x}') \,  dx'_k \wedge dx'_l.
$$
To avoid notational overload, we will not be this precise.

The $\cG'$ invariance  of ${\theta'}$ allows us to compute $\Tr({\theta'})$  as 
follows.   Denote the plaque of $x'$ in $U'$ by 
$P_{x'}$.  Let $j':P_{x'} \to \wL_{x'}$ be the map given by: $j'(w')$ is the 
class of any leafwise path in $P_{x'}$ from $x'$ to $w'$.  Then the value 
of $\Tr({\theta'})$ at $\rho'(x') \in T'$ is given by 
$$
\Tr({\theta'})(\rho'(x')) =    \int_{j'(P_{x'})}  
\psi'_{U'}({j'}^{-1}(y'))\sum_{i,j,k,l} a'_{i,j,k,l}(\rho'(x')) \xi'_i(y')\wedge 
\xi'_j(y') \, | \, j'(P_{x'})\,dx'_k \wedge dx'_l.
$$
Abusing notation once again by identifying $P_{x'}$ with its image under 
$j'$, we have that at $\rho'(x') \in T'$,
$$
\Tr({\theta'})(\rho'(x')) = \int_{P_{x'}} \psi'_{U'}(y') 
\sum_{i,j,k,l} a'_{i,j,k,l}(\rho'(x')) \xi'_i(y')\wedge \xi'_j(y') 
\, dx'_k \wedge dx'_l  = 
$$
$$
\sum_{i,j,k,l} a'_{i,j,k,l}(\rho'(x'))
\Big[\int_{P_{x'}} \psi'_{U'}(y') \xi'_i(y')\wedge \xi'_j(y') \Big]
\, dx'_k \wedge dx'_l.
$$
Similar remarks apply to all powers of ${\theta'}$.

We now return to our analysis on $V = f^{-1}(U')$, where  we have the 
normal coordinates $x_{p+1}, ..., x_n$ given by  $x_i = x'_i \circ f \circ 
\rho$, so $dx_i = f^*(dx'_i)$.
If we set $\xi_i = \wf^*(\xi'_i)$, then the $\xi_i$ are a spanning set  
of  $\pi^f_+  \, | \, V$.  Set $a_{i,j,k,l} = a'_{i,j,k,l}\circ f 
\circ \rho$, where $\rho:V \to T$.   Using Lemma \ref{induced} along with Proposition \ref{preserves}, the Schwartz kernel $\Theta_x(y,z)$ of $\theta \, | \, V$ is given by 
$$
\Theta_x(y,z) = \sum_{i,j,k,l } a_{i,j,k,l}(\rho(x)) \xi_i(y) \otimes \xi_j(z) \, dx_k \wedge dx_l,
$$
and the action $\theta  \, | \, V$ is 
$$
{\theta}(\xi) \,\, = \,\,\sum_{k,l = p+1}^n \sum_{i,j} a_{i,j,k,l} \,Q( \xi_j,  \xi)  \, \xi_i  \, dx_k \wedge dx_l  \,\, = \,\,
\sum_{i,j,k,l} a_{i,j,k,l} \, \Big[ \int_{\wL} \xi_j  \wedge   \xi \Big] \xi_i  \, dx_k \wedge dx_l.$$
That is 
$$
\Theta = \wf^* {\Theta'} .
$$

We are interested in the Schwartz kernels $\Theta'^k$ and $\Theta^k$ of the operators ${\theta'}^k$ and $\theta^k$.    These 
are given by 
$$
{\Theta '} _{x'}^k   (y',z')  =  \int_{\wL_{x'}}  \int_{\wL_{x'}}  \ldots    \int_{\wL_{x'}}  {\Theta '} _{x'}   (y',w'_1) \wedge  {\Theta '} _{x'}  (w'_1,w'_2)  \wedge \ldots \wedge  \Theta ' _{x'}   (w'_{k-1},z') 
$$
and
$$
{\Theta } _{x}^k   (y,z)  = \int_{\wL_{x}}  \int_{\wL_{x}} \ldots   \int_{\wL_{x}} 
{\Theta } _{x}   (y,w_1) \wedge  {\Theta } _{x}   (w_1,w_2)\wedge  
\ldots \wedge  \Theta_{x}   (w_{k-1},z), 
$$
where the integration is done over repeated variables.
Using  Proposition \ref{preserves}  again, 
we have immediately that
$$
\Theta^k =  \wf^*({\Theta'}^k).
$$ 
 
For each $\psi'_{U'}$ in the  partition of unity subordinate to $\{U'\}$, 
set $\psi_V = \psi' _{U'} \circ f$, which gives a partition of unity 
subordinate to the open cover $\{ V \}$ of $M$.    Denote by $\dd \int_{V}$ 
integration over the fibration $\rho:V \to T$.    Recall the map 
$i:M \to \cG$ given by $i(x) = \overline{x}$, the class of the constant 
path at $x$.

\begin{lemma}  \hspace{0.25in}
$
\dd \Tr(\theta^k)  =   \sum_V \int_{V} \psi_V  \, i^* \tr(\Theta^k).
$
\end{lemma}

\begin{proof}
It suffices to show that for any differential form $\omega$ on $M$, $\dd  
\int_{F} \psi_V \omega$ and $\dd \int_{V} \psi_V \omega$ define the same 
Haefliger form.  Let $W_0,...,W_k, W_{k+1}, ..., W_{m}$ be an open cover 
of $M$ by foliation charts, with transversals $S_0, ...,S_{m}$.  
We may assume that $W_0,...,W_k$ are the only elements which intersect the 
support of $\psi_V$ non-trivially, and that these sets are subsets of
$V$.  Let $\what{\psi}_0, ..., \what{\psi}_{m}$ be a partition of unity subordinate to the 
$W_j$.   We require that $W_0 = U$ and $S_0 = T$.  Recall that  
\ ${\rho'}:U' \to T'$ is the projection.  
For $j = 1,...,k$, choose a point $y_j \in S_j$.   Then ${\rho'}(f(y_j)) 
= f( \rho(y_j))$, and  as in the proof of Lemma \ref{locpullbacks}, we 
define the leafwise path $\gamma_j$ from $\rho(y_j)$ to $y_j$.
By construction, the holonomy map $h_j$ induced by the leafwise path 
$\gamma_j$ (which has domain possibly a proper subset of $S_0$) has  range 
{\bf all} of $S_j$.  In addition, for each $S_j$, the map  $h^{-1}_j:S_j 
\to S_0 = T$ is just the restriction to $S_j$ of   the projection 
$\rho:V \to T$.
Then the Haefliger classes 
$$
\int_{F} \psi_V \omega \,\, \equiv \,\,\ 
\sum_{j=0}^k  \int_{W_j} \what{\psi}_j \psi_V \omega  \,\, = \,\, 
\int_{W_0} \what{\psi}_0 \psi_V \omega   \,\, + \,\,  \sum_{j=1}^k  
h_j^* \Big{(}\int_{W_j} \what{\psi}_j \psi_V \omega\Big{)} \,\, = \,\, 
$$
$$
\int_{W_0} \what{\psi}_0 \psi_V \omega   \,\, + \,\,    \sum_{j=1}^k h_j^* 
\Big{(}\int_{W_j} \what{\psi}_j \psi_V \omega\Big{)}.
$$
\noindent
The Haefliger form $\dd \int_{W_0} \what{\psi}_0 \psi_V \omega   + \sum_{j=1}^k 
h_j^* \Big{(}\int_{W_j} \what{\psi}_j \psi_V \omega\Big{)}$ 
is supported on $S_0 = T$, and it follows immediately from the fact 
that  $h^{-1}_j:S_j \to S_0$ is just $\rho:S_j \to T$, that it equals
$ \dd\int_{V} \psi_V \omega$.
\end{proof}

Now  
$\dd \ch_a(\pi^f_+ ) = \Big[\Tr\Bigl{(}\pi^f_+ + \sum_{k=1}^{[n/2]}\frac {(-1)^k 
\theta^{k}}{(2i\pi)^k k!}\Bigr{)} \Big],$
and by Theorem \ref{secondeq},  this equals 
$\ch_a(\pi_+)$, which is independent of the Bott form 
$\omega$ used to construct $\wf^*$.    
Let  $\phi$ be a smooth even function on $\R$, decreasing on $[0,1]$, 
with  $\phi(0) = 1$ and $\phi(x) = 0$ for $|x| \geq 1$, and let  $\omega$ 
be the Bott form  which is a multiple of  $\phi(x_1) ... \phi(x_k) dx_1 
... dx_k$.  For $t > 0$, let $q_t: \R^k \to \R^k$ be the diffeomorphism 
$q_t(x) =  x/t$. 
Denote by  $\omega_t$ the smooth family of Bott forms given by
$
\omega_t =  q_t^* \omega,
$
and denote by $\wf^*_t$ the map constructed using $\omega_t$.   
Then for 
all $t > 0$ and $k \geq 1$, we have 
$$
\Big[ \Tr(\theta^k) \Big] \,\, = \,\, \Big[  \sum_V \int_{V}  \psi_V \, \, i^* \tr(\Theta^k )\Big] \,\, = \,\, 
\Big[  \sum_V \int_{V}  f \circ \psi'_{U'}\, \, i^* 
\tr(\wf^*_t({\Theta'}^k))\Big]  \,\, = \,\,
\Big[  \sum_V  \lim_{t \to 0} \int_{V}  f^* (\psi'_{U'})\, \, i^*  \wf^*_t(\tr {\Theta'}^k)\Big].
$$
We may use the $\omega_t$ to construct the  family of maps 
$f^*_t$ (analogous to the family $\wf^*_t$), defined on the original foliation $F$.  As both $\wf^*_t$ and 
$f^*_t$ are locally constructed, and $\tr {\Theta'}^k$ is $\cG'$ 
invariant, it is clear that 
$$
i^* \wf^*_t(\tr {\Theta'}^k) \,\, = \,\,  f^*_t({i'}^* \tr {\Theta'}^k).
$$   
Thus,
$$
\Big[  \sum_V  \lim_{t \to 0} \int_{V}  f^* (\psi'_{U'})\, \, i^*  \wf^*_t(\tr {\Theta'}^k)\Big]  \,\, = \,\, 
\Big[  \sum_V  \lim_{t \to 0} \int_{V}  f^* (\psi'_{U'})f^*_t({i'}^* \tr {\Theta'}^k)\Big].
$$
It is a classical result that on each plaque in $V$, the compactly 
supported  forms $f^* (\psi'_{U'}) f^*_t({i'}^* \tr {\Theta'}^k)$ are bounded 
independently of $t \in [0,1]$, and converge pointwise to $f^* 
(\psi'_{U'})f^*({i'}^* \tr {\Theta'}^k)  =  f^* (\psi'_{U'} {i'}^* \tr {\Theta'}^k)$.  By the Dominated Convergence 
Theorem, we have
$$
\Big[\Tr(\theta^k)  \Big] \,\, = \,\,  \Big[  \sum_V   \int_{V} \lim_{t \to 0} f^* (\psi'_{U'}) f^*_t( {i'}^* \tr 
{\Theta'}^k)\Big] = 
\Big[  \sum_{U'} \int_{f^{-1}(U')}  f^*( \psi'_{U'} {i'}^* \tr ({\Theta'}^k))\Big]  = 
$$
$$
\Big{[} f^*  \sum_{U'}  \int_{U'}\psi'_{U'}  {i'}^* \tr({\Theta'}^k) \Big{]} = f^* \Big[ \Tr({\theta'}^k) \Big].
$$

As $\dd\ch_a(\pi'_+ ) = \Big[ \Tr\Bigl{(}\pi'_+ + \sum_{k=1}^{[n/2]}\frac {(-1)^k {\theta'}^{k}}{(2i\pi)^k k!}\Bigr{)} \Big]$, to finish the proof that $\ch_a(\pi^f_+) = f^*(\ch_a(\pi'_+))$, we need only show that
$ \Big[ \Tr(\pi^f_+) \Big] =  f^* \Big[ \Tr(\pi'_+) \Big]$.  Just as we did with $\theta'$, 
we may write the Schwartz kernel of $\pi'_+ \,  \, | \, U'$ as
$$
(\pi'_+)_{x'}(y',z') =  \sum_{i,j} b'_{i,j}(\rho'(x')) 
\xi'_i(y') \otimes \xi'_j(z'),
$$
where the $b'_{i,j}$ are functions on $T'$, and the action of $\pi'_+$ on a section $\xi'$ is given by 
$$
\pi'_+(\xi') =\sum_{i,j} b'_{i,j} \,Q'( \xi'_j,  \xi')  \, \xi'_i.
$$

Set $b_{i,j} =  \wf^* b'_{i,j} = b'_{i,j} \circ f \circ \rho$ and $\xi_i = \wf^*(\xi'_i)$, and 
 consider the operator  $\wtit{\pi}^f_+$ on $\cA^{\ell}_{(2)}(F_s,E)$, where
$\wtit{\pi}^f_+  \, | \, V =  \sum_{i,j} b_{i,j} \xi_i \otimes \xi_j$, which acts by 
$$
\wtit{\pi}^f_+ (\xi) \,\, = \,\, \sum_{i,j} b_{i,j} \, Q(\xi_j,  \xi) \xi_i.
$$
Then $\wtit{\pi}^f_+$ 
is a $\cG$ invariant idempotent, has image equal to $\Im(\pi^f_+)$, and has a smooth Schwartz kernel.  In general $\wtit{\pi}^f_+ \neq \pi^f_+$ because forms of the type $\delta_s \beta$, which are in the kernel of $\pi^f_+$, are not necessarily in the kernel of $\wtit{\pi}^f_+$.  However, since $\wtit{\pi}^f_+$ has smooth Schwartz kernel, $\Tr(\wtit{\pi}^f_+)$ is well defined, and its Schwartz kernel is just $\wf^*$ of the Schwartz kernel of $\pi'_+$.  Arguing as we did  for $\theta^k$, we get $\Big[ \Tr(\wtit{\pi}^f_+)\Big] =  f^*\Big[ \Tr(\pi'_+) \Big]$.  

\begin{lemma}
\hspace{1cm} $\Big[\Tr(\pi^f_+)\Big] \,\, = \,\, \Big[\Tr(\wtit{\pi}^f_+)\Big].$
\end{lemma}

\begin{proof}
Since $\Im( \pi^f_+) = \Im(\wtit{\pi}^f_+)$, and both are idempotents, we  need only show that $\wtit{\pi}^f_+$ is transversely smooth, and then apply Lemma \ref{samecc}.

We will be using the notation of Section \ref{lms}.  Suppose the $K'$ is the Schwartz kernel of a $\cG'$ invariant bounded leafwise smoothing operator on $\cA^{\ell}_{(2)}(F'_s,E')$, which is given locally, with respect to a local invariant spanning set $\{\xi'_i\}$ of $\cA^{\ell}_{(2)}(F'_s,E')$, by
$
K' = \sum_{i,j} b'_{i,j} \xi'_i \otimes \xi'_j,
$
with the action given by
$$
K'(\xi') =\sum_{i,j} b'_{i,j} \,Q'( \xi'_j,  \xi')  \, \xi'_i.
$$
Now consider the operators $\wf^* K'$ on $\cA^{\ell}_{(2)}(F_s,E)$ and $\wtit{p}^*_f K'$ on $\cA^{\ell}_{(2)}(F_s \times B^k, p^*_f E')$, with local Schwartz kernels 
$$
\wf^*K' = \sum_{i,j} \wf^* b'_{i,j} \wf^* \xi'_i  \otimes \wf^* \xi'_j,
\quad \text{  and  }  \quad
\wtit{p}^*_f K' = \sum_{i,j} p^*_f b'_{i,j} (p^*_f \xi'_i  \wedge \omega)\otimes (p^*_f \xi'_j  \wedge \omega),
$$
where $\omega$ is a Bott form on $B^k$.  Recall that $\pi_{1,*}$ is integration over the fiber of the projection $\pi_1:\cG \times B^k \to \cG$, and $p_{f,*}$ is integration over the fiber of the submersion $p_f:\cG \times B^k \to \cG'$.
Straight forward computations show that for $\xi \in \cA^{\ell}_{(2)}(F_s,E)$ and $\wtit{\xi} \in \cA^{\ell}_{(2)}(F_s \times B^k, p^*_f E')$,
$$
\wf^*K' (\xi) \,\, = \,\, \pi_{1,*} \Big( \wtit{p}^*_f K' (\pi_1^* \xi) \Big)
\quad \text{  and  }  \quad
\wtit{p}^*_f K' (\wtit{\xi})  \,\, = \,\,    p^*_f \Big( K'( p_{f,*}(\omega \wedge \wtit{\xi})) \Big) \wedge \omega.
$$
The maps  $ \pi_{1,*}$, $\pi_1^*$, $p^*_f$, $p_{f,*}$, and $\wedge \omega$ are all bounded maps, and $K'$ is bounded and leafwise smoothing.  Thus $\wf^*K'$ is a bounded leafwise smoothing operator.  Applying this to $K' = \pi'_+$, we have that  $\wtit{\pi}^f_+$ is a bounded leafwise smoothing operator.

Using Proposition \ref{fdnu}, it is easy to show that 
$
\pa^Y_{\nu}\wtit{\pi}^f_+  =  [A(Y),\wtit{\pi}^f_+]  +  \wf^*(i_{Z'}\pa_{\nu'} \pi'_+),
$
where $Y$ and $Z'$ are as in Lemma \ref{gents}, and  
$A(Y)$ is a leafwise operator whose composition with a bounded leafwise  smoothing operator is again a bounded leafwise  smoothing operator.  Applying the argument above to $i_{Z'}\pa_{\nu'} \pi'_+$, we have that $\pa^Y_{\nu}\wtit{\pi}^f_+$ is also a bounded leafwise smoothing operator.  An obvious induction argument finishes the proof.
\end{proof}

Thus
$\Big[\Tr(\pi^f_+) \Big]   = \Big[ \Tr(\wtit{\pi}^f_+)  \Big]= f^*\Big[ \Tr(\pi'_+) \Big]$, and we are done.
\end{proof}

\section{The twisted leafwise signature operator and the twisted higher Betti classes}

In this section we give some immediate consequences of our results.  In particular, we show that  the twisted higher harmonic signature  equals the 
(graded) Chern-Connes character in Haefliger cohomology of the ``index bundle" of 
the twisted leafwise signature operator, that is the (graded) Chern-Connes character 
$\ch_a(P)$ of the projection $P$ onto all the twisted leafwise harmonic forms.  We conjecture a cohomological formula for this Chern-Connes character, which has already been proven in some cases.  We also indicate how our methods prove that the twisted higher Betti numbers are leafwise homotopy invariants.

Consider the first order leafwise operator $D^E = d_s + \delta_s$, which  is 
formally self adjoint and satisfies $(D^E)^2 = \Delta^E$.   Because of this, the 
kernel of $D^E$ is the same as the kernel of $\Delta^E$.   Recall the $\pm 1$ 
eigenspaces $\cA^*_{\pm}(F_s,E)$ of the involution $\what{\tau}$ of 
$\cA_{(2)}^*(F_{s},E)$,  and that 
$$
D^E \what{\tau} = -\what{\tau} D^E,
$$
so we have the operators $D^{E\pm}:\cA^*_{\pm}(F_{s},E) \to 
\cA^*_{\mp}(F_{s},E)$, and  $D^{E+}$ is designated the twisted  leafwise signature 
operator.

Denote by $P_{\pm}$ the projections onto the $\Ker(D^{E\pm})$.
We assume that the projection $P$ to $\Ker (\Delta^E)$ is transversely smooth, so the 
$P_{\pm}$ are also.  Then the (graded) Chern-Connes character of the index bundle of the 
twisted leafwise signature operator, $\ch_a(P)$, is defined, and is given by 
$$
\ch_a( P)  =   \ch_a( P_+) - \ch_a( P_-) =
$$
$$
\Bigl{[}\ch_a(\sum_{j=0}^{\ell-1} P_j + \tau P_j) 
+ \ch_a(\frac{1}{2}(P_{\ell} + \tau P_{\ell}))\Bigr{]}-
\Bigl{[}\ch_a(\sum_{j=0}^{\ell-1} P_j - \tau P_j) +
\ch_a(\frac{1}{2}(P_{\ell} - \tau P_{\ell}))\Bigr{]}.
$$
As in the case of compact manifolds, we have
\begin{theorem}\label{grdedlhi}  
Suppose that  $M$ is a compact
Riemannian manifold,  with oriented {Riemannian} foliation $F$ of dimension $2\ell$, 
and that $E$ is a leafwise flat complex bundle over $M$ with a (possibly indefinite) non-degenerate Hermitian metric which is preserved by the leafwise flat structure.  Assume that the projection $P$ onto  $\Ker( \Delta^{E} )$ for the associated foliation $F_s$ of the homotopy groupoid of $F$ is transversely smooth.  Then, the (graded) Chern-Connes character $\ch_a( P)$ of the index 
bundle of the twisted  leafwise signature operator equals the twisted higher harmonic 
signature of $F$, that is
$$ \ch_a( P) = \sigma(F,E).$$
\end{theorem}
\begin{proof}  As
$\ch_a$ is linear
and  $\frac{1}{2}(P_{\ell} \pm  \tau P_{\ell}) = \pi_{\pm}$, we need only 
show that 
$$
\ch_a( P_j + \tau P_j)  = \ch_a( P_j - \tau P_j),
$$ 
for $j = 0,...,\ell-1$.    Set $P_t = P_j + t \tau P_j$ where $-1 \leq t 
\leq 1$.  Then $P_t$ is a smooth family of $\cG$ invariant transversely 
smooth idempotents (since $P_j \tau P_j = 0$ for $j = 0,...,\ell-1$) which 
connects $ P_j + \tau P_j$ to $ P_j - \tau P_j$.  It follows from 
Theorem \ref{family} that $\ch_a( P_j + \tau P_j)  = \ch_a( P_j - \tau 
P_j)$.
\end{proof}

\begin{corollary}
Under the hypothesis of Theorem \ref{grdedlhi}, the (graded) Chern-Connes character $\ch_a( P)$ of the index bundle of the leafwise signature operator with coefficients in $E$ is a leafwise homotopy invariant.
\end{corollary}

The operator $D^{E+}$ is elliptic along the leaves of $F_s$, and so  
produces, via a now classical construction due to Connes 
\cite{Connes:1981}, a $K-$theory invariant $\Ind_a(D^{E+})$, the index of the 
operator $D^{E+}$, which has a Chern-Connes character   $\ch_{a}(\Ind_{a}(D^{E+})) \in \oH_c^{*}(M/F)$, \cite{BHI}.
\begin{conjecture} \label{limit-k}\
Under the hypothesis of Theorem \ref{grdedlhi}, 
$$
\ch_{a}(\Ind_{a}(D^{E+})) =  \ch_{a}( P) \quad \in \quad \oH_c^{*}(M/F).
$$
\end{conjecture} 
This conjecture has been proven when the spectrum of $D^{E+}$ is reasonably well behaved, see  \cite{Heitsch:1995, H-L:1999,BHII}, where it is proven for the holonomy groupoid.  The proofs extend immediately to the homotopy groupoid.  It also holds for both groupoids, without any extra assumptions, whenever the projection $P$ belongs to Connes' $C^*$-algebra of the foliation for the groupoid in question.  In particular, it holds for the holonomy groupoid case for any foliation whose leaves are the fibers of a fibration between closed manifolds, provided that $P$ is transversely smooth.   

Recently, Azzali, Goette and Schick have announced, \cite{AGS}, that they have proven it  
for  smooth proper submersions $V \to B$ with the fibrewise action (freely and properly discontinuous) of a discrete group $\Gamma$ such that  the quotient $V/\Gamma \to B$ is a fibration with compact fiber, but only for bundles $E$ which are globally flat.
Conjecture \ref{limit-k} should follow immediately for the homotopy groupoid provided that their result extends to bundles which are only leafwise flat.

Recall, \cite{BHI, GL}, that in Haefliger cohomology,
$$
\ch_{a}(\Ind_{a}(D^{E+})) \,\, = \,\,   \int_F \L(TF)  \ch_2(E),
$$
where $ \L(TF)$ is the characteristic class of $TF$ associated with  the multiplicative sequence   $\prod_j {x_j}/{\tanh(x_j)}$, and $  \ch_2(E) = \sum_k 2^k \ch_k(E)$.

\begin{corollary}
Under the hypothesis of Theorem \ref{grdedlhi}, and assuming Conjecture \ref{limit-k},
 $\dd \int_F \L(TF)  \ch_2(E)$ is a leafwise homotopy invariant.
\end{corollary}

Finally we have the following.
\begin{definition} Assume the hypothesis of Theorem \ref{grdedlhi}, but now $F$ may have arbitrary dimension.     
For  $0\leq j \leq p = \dim(F)$,  define the $j$-th twisted higher Betti class $\beta_j (F,E)$  by
$$
\beta_j(F,E) = \ch_a( P_j)  \quad \in \quad \oH^*_c(M/F).
$$
\end{definition}
It is an interesting exercise to show that, just as in the case of compact fibrations, the bundle defined by the projection onto the leafwise harmonics (in the case $E=M \times \C$) is a flat bundle.  That is, it admits a connection whose curvature is zero, so there are no higher terms in the $\beta_j(F,M \times \C)$.  This is not the case in general.

\begin{theorem}\label{betti}  (Compare \cite{H-L:1991}) Under the hypothesis of Theorem \ref{grdedlhi} with $F$ allowed to have arbitrary dimension,  the twisted higher Betti classes $\beta_j(F,E)$,  are  leafwise homotopy invariants.
\end{theorem}

\begin{proof} We only give a sketch here of the proof of the second statement.   Let $f:(M,F)\to (M',F')$ be a smooth leafwise homotopy equivalence with smooth homotopy inverse $g$. The pull-back bundle $f^*(P'_j)$ is a 
smooth bundle since it can be realized by the transversely smooth 
idempotent $P_j^f =f^* R'^* P'_j g^* R^* P_j$.  It can be endowed with the 
pull-back connection under $f$ of the connection $P'_j {\nabla'}^\nu 
P'_j$, and hence the Chern-Connes character of $f^*(P'_j)$ is 
given by
$$
\ch_a (f^*(P'_j)) = f^* \ch_a (P'_j) = f^* \beta_j (F',E').
$$
As in the proof of our main theorem, one proves that $P_j:f^*(\Ker (\Delta^{E'}_j)) \to \Ker(\Delta^E_j)$ is an isomorphism and that $Q_j^f= P_j P_j^f$ is a smooth idempotent with image 
$\Ker(\Delta^E_j)$, hence its Chern-Connes character coincides with the 
Betti class $\beta_j(F,E)$.   {As $Q_j^f P_j^f = Q_j^f$ and $P_j^f Q_j^f = P_j^f,$ the family
$Q_t = t Q_j^f + (1-t) P_j^f )$ is a smooth homotopy by transversely smooth idempotents from $Q_j^f$ to $P_j^f$.  Therefore, $P_j^f$ and $Q_j^f$ have same Chern-Connes character.}
\end{proof}

\section{Consequences of the Main Theorem}\label{eg}

In this section, we derive some important consequences of Theorem \ref{main}.    In particular, we re-derive some classic results for the Novikov conjecture, and then give some general results for the  Novikov conjecture for groups and for foliations.  

\begin{example}[Lusztig, \cite{Lusztig}] \label{Lusz}
\end{example}
Let $N$ be a compact connected even dimensional Riemannian manifold.  Set 
$W = \oH^1(N;\R / \Z)$, and recall the natural (onto) map 
$h_1: W \to \Hom(\oH_1(N; \Z); \R / \Z)$.  
Choose a base point $x_o \in N$.  Then there is the natural (onto) homomorphism 
$h: W \to \Hom(\pi_1(N, x_o) ; \R / \Z)$ given by composing $h_1$ with the natural map 
$\pi_1(N, x_o)  \to H_1(N,\Z)$.  Thus for each element $w \in W$, we have the homomorphism 
$h(w):\pi_1(N, x_o)  \to \R / \Z$, which we may compose with the map $x \to \exp (2\pi i x)$ to obtain the homomorphism 
$h_{w}:\pi_1(N, x_o)  \to S^1 \subset \C$.    Denote by $\wtit{N}$ the universal covering of $N$.   $\pi_1(N, x_o)$ acts on $\wtit{N}$ in the usual way, and on $\wtit{N} \times W \times \C$ as follows.  Let $\beta \in \pi_1(N, x_o)$, and $(x, w, z) \in  \wtit{N} \times W \times \C$,  and define  
$$
\beta\cdot (x, w, z) =  (\beta x, w, h_{w}(\beta) z). 
$$
Set 
$$
E =  (\wtit{N} \times W \times \C) /  \pi_1(N, x_o),
$$
a complex bundle over $  (\wtit{N} \times W) /  \pi_1(N, x_o) = N \times W$, which is  leafwise flat for the foliation $F$ given by the fibration 
$M \equiv N \times W \to W$.  It is obvious that the usual metric on $\C$ defines  a positive definite metric on $E$ which is preserved by the leafwise flat structure.  
As $ \oH_1(N;\R / \Z)$ is the abelianization of $\pi_1(N, x_o)$, $h$ is onto, and it is natural to call $E$ the universal flat $\C$ bundle for $N$.   Then $M$, $F$, and $E$ satisfy the hypothesis of Theorem \ref{main}, since the preserved metric is positive definite.

Note that if $f:N \to N'$ is a homotopy equivalence, then there is a natural extension of $f$ to $f:M,F \to M',F'$ which is a  leafwise homotopy equivalence, and $f^*E' = E$.   Thus $\sigma(F,E)$ is a homotopy invariant of the manifold $N$.

By \cite{BHI} (and assuming Conjecture \ref{limit-k} if necessary), we  have that
$$
\sigma(F,E) \,\, = \,\, \int_N   \L   (TF)\ch_2 (E)  \quad  \in \quad \oH^*_c(M/F) = \oH^*(\oH^1(N;\R / \Z);\R). 
$$
To relate this to Lusztig's theorem on Novikov conjecture,  suppose  that $\pi_1(N,x_0) = \Z^n$. Denote 
by $g:N \to B\Z^n = \T^n$  the map classifying the universal cover $\wtit{N} \to N$ (as a $\Z^n$ bundle), and  let $\alpha_1, ...,\alpha_{n}$ be the natural basis of $\oH^1 (\T^n;\R)$.  
\begin{proposition} \label{Luszch}
\hspace{2cm}  $\dd \ch_2 (E)  \,\,=\,\, \prod_{i=1}^n (1 + 2g^*(\alpha_i) \otimes \alpha_i).$
\end{proposition}

\begin{theorem}[Lusztig, \cite{Lusztig}]
The Novikov conjecture is true for any compact manifold with fundamental group $\Z^n$.
\end{theorem}

\begin{proof}
$$
\sigma(F,E) \,\, = \,\, \int_N     \L(TF)\ch_2 (E)  = \sum_{i_1 < \cdots < i_k}    2^k  \Bigl{[} \int_N\L(TN)g^*(\alpha_{i_1} \wedge \cdots \wedge \alpha_{i_k} ) \Bigr{]}\alpha_{i_1} \wedge \cdots \wedge \alpha_{i_k} 
$$
is a homotopy invariant, so each of the individual terms $\dd \int_N  \L(TN)g^*(\alpha_{i_1} \wedge \cdots \wedge \alpha_{i_k} )$ is a homotopy invariant.
\end{proof}

\begin{proof} (of Proposition \ref{Luszch})  Since $\pi_1(N,x_0) = \Z^n$,  $W = \oH^1(N;\R / \Z) \simeq \T^n$.    The bundle $E \to N \times T^n$ is the pull back by $g \times I:N \times \T^n \to \T^n \times \T^n$ of the bundle $\what{E}_n \to \T^n \times \T^n$  which is given as follows. 
 Let $\xi \in \Z^n = \pi_1(\T^n)$, and $(x, t, z) \in  \R^n \times \T^n \times \C$,  and define  
$$
\xi \cdot (x, t, z) =  ( x + \xi, w,  (\xi w) \cdot z),
$$
where 
$$
 (\xi w) \cdot z = (\exp(2\pi i \xi_1 w_1) z_1,\ldots,  \exp(2\pi i \xi_n w_n) z_n).
$$
Then 
$$
\what{E}_n =  (\R^n \times \T^n \times \C) / \Z^n.
$$
Note that $\what{E}_n = E_1 \otimes \cdots \otimes E_n$, where 
$E_j$ is the pull back by the projection $\T^n \times \T^n \to \T \times \T$ onto the $j$-th coordinates of the bundle $\what{E}_1$.  As  $\dd \ch_2 (\what{E}_n)  \,\,=\,\, \prod_{j=1}^n  \ch_2 (E_j)$, we need only show that 
$$
\dd \ch_2 (\what{E}_1)  \,\,=\,\, \prod_{i=1}^n (1 +2 \alpha \otimes \alpha),
$$
where $\alpha$ is the natural generator of $\oH^1 (\T;\R)$.  That is, $c_1(\what{E}_1)$ is the natural generator of  $\oH^2 (\T^2;\R)$.
This is a  classical direct computation in the theory of characteristic classes.
\end{proof}

We can extend the previous example to the fundamental group $\Gamma$ of the closed oriented surface $S_g$ of genus $g \geq 2$.  This is a well know theorem which follows from the results of many people, the first probably being Lusztig.
\begin{theorem}
The Novikov conjecture is true for any compact manifold with fundamental group $\Gamma$.
\end{theorem}

\begin{proof} 
The space of equivalence classes of representations of $\Gamma$ in  $U(1)$ is easily seen to be a torus $\T^{2g}$ of dimension $2g$.   Form the 
fiberwise flat line bundle $E$ over the total space of the trivial fibration $\pi^{S_g}_2: S_g\times \T^{2g}\to \T^{2g}$ given by
$$
(x, \theta; u) \sim (x\gamma, \theta; h_\theta (\gamma)(u)), \quad x\in \H^2, \,\,  \theta\in \T^{2g}, \,\, u\in \C, \,\,  \gamma \in \Gamma
$$  
where $h_\theta : \Gamma \to U(1)$ is the corresponding homomorphism as in \ref{Lusz}. Denote by
$\pi^{S_g}_1: S_g\times \T^{2g}\to S_g$  the other projection.
Then for any cohomology class $y\in H^*(\T^{2g};\R)$, the cohomology class in $H^*(S_g;\R) = H^*(B\Gamma;\R)$ given by
$$
x= \pi^{S_g}_{1,*} \left[ (\pi^{S_g}_2)^* y \wedge \ch (E)  \right]
$$
satisfies the Novikov conjecture. 
This can be seen as follows.  Let $N$ be a smooth closed manifold with fundamental group $\Gamma$ and denote by $\varphi:N \to S_g = B\Gamma$ a smooth classifying map.  Notice that the  harmonic signature of the foliated manifold $M= N \times \T^{2g}$ (with foliation given by the fibers of the projection  $ \pi^N_{2}:N \times \T^{2g} \to \T^{2g}$)
twisted by the fiberwise flat bundle $(\varphi  \times id)^*E$, is given in $H^*(\T^{2g};\R)$  by the formula
$$
\sigma (M, F ; (\varphi\times id)^*E) = \pi^N_{2,*} \left[(\pi^N_1)^*\L(TN) \cup (\varphi\times id)^* \ch (E)\right].
$$
Clearly, for any cohomology class $y\in H^*(\T^{2g};\R)$, we get the homotopy invariance of 
$$
\int_{\T^{2g}} y \cup \pi^N_{2,*} \left[(\pi_1^N)^*\L(TN) \cup (\varphi\times id)^* \ch (E)\right] = \int_N \L(TN) \pi^N_{1,*} \left[(\pi^N_2)^*y \wedge (\varphi\times id)^* \ch (E)\right].
$$
But $(\pi^N_2)^*y = (\varphi\times id)^* (\pi^{S_g}_2)^*y$ and therefore 
$$
\pi^N_{1,*} \left[(\pi^N_2)^*y \wedge (\varphi\times id)^* \ch (E)\right] = (\pi^N_{1,*}\circ (\varphi\times id)^*) \left[ (\pi^{S_g}_2)^* y \wedge \ch (E)  \right].
$$
The conclusion follows using that $\pi^N_{1,*}\circ (\varphi\times id)^* = \varphi^* \circ \pi^{S_g}_{1,*}$.

Thus we need only show that every class  $x\in H^*(S_g;\R)$ has the given form.  
We may write $S_g = \sharp_g \T^2$ as the union 
$$
S_g \,\,=\,\,H_1 \cup H_2 \cup  \cdots \cup H_g,
$$
where $H_1$ and $H_g$ are $\T^2$ with a disc removed, and the other $H_j$ are $\T^2$ with two discs removed.   There are natural inclusions $g_j:H_j \to \T^2_j \subset \T^{2g}$.
On $\T^{2g} \times \T^{2g}$ we have the bundle $\what{E}_{2g}$.  Consider the natural map 
$$
h_j = g_j \times I:H_j  \times \T^{2g} \to \T^2_j \times \T^{2g} \subset \T^{2g} \times \T^{2g}.
$$
Then $E \, | \, {H_j  \times \T^{2g}} = h^*_j(\what{E}_{2g})$.  Note also that on a neighborhood of the boundary of $H_j$, the bundle $E$ is trivial, and the trivialization is independent of $j$.  Thus we may construct a connection on $E$ by using a partition of unity and the local connections given on the $H_j  \times \T^{2g}$ by the pull back under $h_j$ of the connection used on $\what{E}_{2g}$, and the local flat connections on the neighborhoods of the boundaries of the $H_j$.
Thus on the complement of a collar neighborhood of the boundary of $H_j  \times \T^{2g}$,  $\ch(E) = h^*_j(\ch(\what{E}_{2g}))$, and on a neighborhood of the boundary, 
$\ch(E) = 0$.  Now on $\T^{2g} \times \T^{2g}$, we have the one dimensional cohomology classes $[dx^1_j]$, $[dx^2_j]$, $[dw^1_j]$ and  $[dw^2_j]$ which are dual 
to the natural generators of $\oH_1(\T^2_j;\R)$.  The $[dx^k_j]$  live on the first factor of  $\T^{2g} \times \T^{2g}$, and the  $[dw^k_j]$ on the second.
In addition, 
$$
\ch(\what{E}_{2g}) = \prod_{i=1}^{2g} (1 + [dx^1_i] \wedge [dw^1_i])(1 + [dx^2_i] \wedge [dw^2_i]).
$$   
Set $y_j = \prod_{i \neq j}  [dw^1_i] \wedge [dw^2_i]$.  Denote by $\gamma^1_j$ and $\gamma^2_j$ the elements of $\oH_1(S_g;\R)$ corresponding to the natural generators of $\oH_1(\T^2_j;\R)$.
Then
$$
\Bigl{(}\pi_{1,*} \left[ \pi_2^* y_j \wedge [dw^k_j] \wedge \ch (E)  \right] \, | \, {H_j  \times \T^{2g}} \Bigr{)} (\gamma^{m}_j) = h^*_j([dx^k_j]) (\gamma^{m}_j) = \delta^k_{m},
$$
while for $i \neq j$,
$$
\Bigl{(}\pi_{1,*} \left[ \pi_2^* y_j \wedge [dw^k_j] \wedge \ch (E)  \right] \, | \, {H_i  \times \T^{2g}} \Bigr{)} (\gamma^{m}_i) = h^*_i([dx^k_j]) (\gamma^{m}_i) = 0,
$$
as $h^*_i([dx^k_j])= 0$.

Thus each element of $\oH^1(S_g;\R)$ has the required form.  It is not difficult to see that 
$\pi_{1,*} \left[ \pi_2^* y_j  \wedge \ch (E)  \right]$ gives a non-zero two dimensional class of the required form, so we have the theorem.

\end{proof}

Here is another version of Lusztig's construction,  see \cite{Lusztig} and  \cite{Gromov}.
Let $E$ be a flat $U(p,q)$ bundle over $N$ (that is a flat bundle given by a map $\rho:\pi_1(N) \to U(p,q)$).   Then  $E$ is a leafwise flat complex bundle over $N$ with an indefinite non-degenerate Hermitian metric which is preserved by the leafwise flat structure.    Write $E = E^+ \oplus E^-$, where the indefinite metric is positive $\pm$ on $E^{\pm}$.
\begin{theorem}\label{homotopy}
$$
\int_N     \L(TN)(\ch_2 (E^+) - \ch_2 (E^-))
$$
is a homotopy invariant of $N$.
\end{theorem}

\begin{proof}
If $N$ is odd dimensional, this is zero, so assume that $N$ is even dimensional.
Let $F$ be the foliation of $N$ with one leaf, namely $N$.  The holonomy groupoid of $F$ is just $\cG = N \times N$, and  the projection onto the leafwise harmonic forms is the same on each $N$.   Thus the hypothesis of Theorem \ref{main} are satisfied and Conjecture \ref{limit-k} holds, giving the result.
\end{proof}

This may be recast as follows.  Let $\rho:\Gamma \to U(p,q)$ be a homomorphism of a finitely presented group.  Given any manifold $N$ and homomorphism $\psi:\pi_1(N) \to \Gamma$, we may construct the bundle 
$E = E^+ \oplus E^- \to N$.   This construction is natural under pull-back maps, i.e., given any map 
$f:N' \to N$ we can form the bundle $E'  = E'^+ \oplus E'^-\to N'$ using the homomorphism $\rho \circ \psi\circ f_*$ where $f_*:\pi_1(N') \to \pi_1(N)$ is the induced map.   Then $E'^{\pm} = f^*(E^{\pm})$, and so this construction determines two universal bundles $E^+_{\rho}$ and $E^-_{\rho}$ over $B\Gamma$.  

\begin{theorem}
 Let $\rho:\Gamma \to U(p,q)$ be a homomorphism of a finitely presented group.  Then
$$
\ch(E^+_{\rho}) - \ch(E^-_{\rho}) \quad \in \quad \oH^*(B\Gamma;\R)
$$
satisfies the Novikov conjecture.
\end{theorem}

Note that the universal $\C^{p+q}$ bundle $EU(p,q) \times_{U(p,q)} \C^{p+q} \to BU(p,q)$ splits as  $EU(p,q) \times_{U(p,q)} \C^{p+q} = E^+_{p,q} \oplus E^-_{p,q}$, and for any map $f:N \to BU(p,q)$ classifying a bundle $E$ with splitting $E = E^+ \oplus E^-$, $f^*(E^{\pm}_{p,q}) = E^{\pm}$.  The map  $\rho:\Gamma \to U(p,q)$ induces $B\rho:B\Gamma \to BU(p,q)$, and 
$\ch(E^{\pm}_{\rho}) = B\rho^*(\ch(E^{\pm}_{p,q}))$.
Now $U(p) \times U(q)$ is a maximal compact subgroup of $U(p,q)$, so the inclusion $i:BU(p) \times BU(q) \to BU(p,q)$ induces an isomorphism in cohomology. 
 That is 
$$
\oH^*(BU(p,q);\R) \,\, = \,\,  \oH^*(BU(p);\R) \otimes \oH^*(BU(q);\R).
$$
It is not difficult to see that under this isomorphism 
$$
\ch (E^+_{p,q}) \,\,=\,\, \ch (E_p) \quad \text{and} \quad \ch (E^-_{p,q}) \,\,= \,\, \ch (E_q),  
$$
where  $E_p \to BU(p)$ and $E_q \to BU(q)$ are the universal bundles.  Thus we have

\begin{theorem}
 Let $\rho:\Gamma \to U(p,q)$ be a homomorphism of a finitely presented group.  Then
$$
(B\rho)^*( i^*)^{-1} \Bigl{(} \ch(E_p) - \ch(E_q) \Bigr{)} \quad \in \quad \oH^*(B\Gamma;\R)
$$
satisfies the Novikov conjecture.
\end{theorem}

Of course, this follows immediately from the well known fact that the Novikov conjecture is true for subgroups of Lie groups. The main input here is the possibility to use (complementary) families of representations giving rise to interesting foliations.   To this end, we have the following generalization of the Lusztig construction.  It would be a very interesting application to use this construction to shed more light on the series of some discrete groups sitting in $U(p,q)$.  Note that, for a given Lie group $H$, the space
$\Hom(\Gamma, H)$ is well understood for abelian groups $\Gamma$ and 
has been intensively studied  when $\Gamma$ is a higher genus surface group and $H$ is $PSL(2, \R)$ or $PU(1,2)$, see \cite{Goldman}. Other examples of $\Gamma$ and $H$ 
have also been studied by other authors and they all fit into the case of $H=U(p,q)$, see \cite{Gromov} for a survey.

\begin{example} [Foliation Lusztig Example]\label{LuszFol}
\end{example} 
Let $K$ be a compact Riemannian manifold without boundary, and $g:\pi_1(N) \to \Iso (K)$ a homomorphism to the isometries of $K$.  
Denote by $\Hom_c(\pi_1(N), U(p,q))$ the set of homomorphisms from $\pi_1(N)$ to $U(p,q)$ which have image contained in a compact subgroup.
Let
$$
h:K \to \Hom_c(\pi_1(N), U(p,q))
$$
be a 
weakly uniformly continuous smooth $g$-cocycle.   Smoothness of $h$ means that for any $\gamma \in \pi_1(N)$, $w \to h_{w}(\gamma)$ is a smooth function from $K$ to $U(p,q)$.
Weak uniform continuity of $h$ means the following.  Denote the norm on $U(p,q)$ by $|| \cdot ||$.   Given  $w_1, w_2 \in K$, define
$$
d_W( w_1, w_2) \,\, = \,\, \max_{A_1}\bigl{[} \min_{A_2} ||A_1 - A_2 ||\bigr{]},
$$
where $A_i \in \overline{ h_{w_i}( \pi_1(N))}$, the closure of the image of $ \pi_1(N)$ under $h_{w_i}$.
Then, $h$ is weakly uniformly continuous if $d_W( w_1, w_2) \to 0$ as $w_1 \to w_2$.

That $h$  is a $g$-cocycle means that for  $\gamma_1, \gamma_2 \in \pi_1(N)$ and $w \in K$,
$$
h_{g_{\gamma_2}(w)}(\gamma_1)h_w(\gamma_2) \,\, = \,\, h_w(\gamma_1 \gamma_2).
$$
Then we may form 
$$
E \,\, = \,\, \wN \times K \times \C^{p+q} / \pi_1(N),
$$ 
where the action of $\gamma \in \pi_1(N)$ on $(x,w,z) \in  \wN \times K \times \C^{p+q}$ is given by
$$
\gamma(x,w,z) \,\, = \,\, (\gamma(x), g_{\gamma}(w), h_w(\gamma)z).
$$
Then $E$ is a $\C^{p+q}$ bundle over $\wN \times_{\pi_1(N)} K$.

Now, we have the Riemannian foliation $F$ of the flat fiber bundle $\wN \times_{\pi_1(N)} K \to N$, whose leaves consist of the images of the $\wN \times \{w\}$.  The bundle  $E$
is leafwise flat, and the (indefinite) inner product is preserved by the flat structure.  Again  write $E = E^+ \oplus E^-$, where the indefinite metric is $\pm$ definite on $E^{\pm}$.  The parallel translation along the leaves of $F$ is bounded since  the closure of the union of all the images, $\overline{\bigcup_K h_w( \pi_1(N))}$ is a compact subset of $U(p,q)$.  This follows easily from the facts that $K$ is compact, each  $\overline{ h_{w}( \pi_1(N))}$ is compact, and $h$ is weakly uniformly continuous.   (We conjecture that continuity of $h$ and compactness of $K$ imply compactness of  $\overline{\bigcup_K h_w( \pi_1(N))}$.)  As above,  the hypothesis of Theorem \ref{main} are satisfied, and we may  apply Conjecture \ref{limit-k} to get
\begin{theorem}\label{homotopy2}
For every $g$, $h$ and $K$ as above,
$$
\int_F     \L(TF)(\ch_2 (E^+) - \ch_2 (E^-))
$$
is a  homotopy invariant of $N$.
\end{theorem}
Note that we may view this Haefliger form as living on a single fiber $K$ of the bundle 
$\wN \times_{\pi_1(N)} K \to N$.  This is because we may take fundamental domains of $N$ in the various leaves to integrate over (when we do integration over the fiber to get to Haefliger cohomology), and these fundamental domains are indexed by any fiber $K$.   Thus we may integrate over $K$ to obtain 
\begin{corollary} For every $g$, $h$ and $K$ as above, the real number 
$$
\int_K \int_F     \L(TF)(\ch_2 (E^+) - \ch_2 (E^-))  
$$
is a  homotopy invariant of $N$.
\end{corollary}

As above, we may recast this result in terms of the Novikov conjecture.
Let $\Gamma = \pi_1(N)$ and let $g$, $h$ and $K$ be as in Example \ref{LuszFol}.  The construction of the bundle $E \to \wN \times_{\Gamma} K$ and its splitting $E = E^+ \oplus E^-$ are natural with respect to pull-back maps, so this construction defines the universal bundle 
$$
E_B \,\, = \,\, E\Gamma \times K \times \C^{p+q} / \Gamma,
$$ 
where the action of $\gamma \in \Gamma$ on $E\Gamma \times K \times \C^{p+q}$ is given as above  by
$\gamma(x,w,z) \,\, = \,\, (\gamma(x), g_{\gamma}(w), h_w(\gamma)z)$.
Then $E_B$ is a $\C^{p+q}$ bundle over $E\Gamma \times_\Gamma K$, and it splits as 
$E_B = E^+_B  \oplus E^-_B$.  If $\varphi:N \to B\Gamma$ classifies the universal cover $\wN
\to N$, with induced map $\wtit{\varphi}:\wN \to E\Gamma$, 
then $\wtit{\varphi} \times id_K:\wN \times K \to E\Gamma \times K$ descends to the map     
$\wtit{\varphi} \times_{\Gamma} id_K:\wN \times_{\Gamma} K \to  E\Gamma \times_\Gamma K$, and
$(\wtit{\varphi} \times_{\Gamma} id_K)^*(E^{\pm}_B) = E^{\pm}$.  
\begin{proposition}\label{Nov} Denote by $\pi_1^\Gamma: E\Gamma\times_\Gamma K \to B\Gamma$ the projection.  Then
$$
\pi_{1,*}^\Gamma (\ch ([E^+] - [E^-]))
$$
satisfies the Novikov conjecture.
\end{proposition}

\begin{proof}  This follows immediatelly since a direct inspection shows that in the cohomology of $N$ 
$$
\pi^N_{1, *} \circ (\wtit{\varphi} \times_{\Gamma} id_K)^*  = \varphi^* \circ \pi^{\Gamma}_{1,*}.
$$

\end{proof}

\begin{remark}
Example \ref{LuszFol} can be easily generalized to the following situation. Let $E_0$ be a complex vector bundle over $K$ which is endowed with a (possibly indefinite) non-degenerate metric $\{\cdot, \cdot\}$. Assume that the vector bundle $E_0$ is a $\Gamma$-equivariant vector bundle and that the action of $\Gamma$ preserves $\{\cdot, \cdot\}$. Then the vector bundle
$$
E := {\tilde N} \times_\Gamma E_0 \rightarrow {\tilde N} \times_\Gamma K,
$$
is easily seen to be a complex bundle with a well defined (possibly non-degenerate) metric, which admits a leafwise flat connection preserving that metric. Hence (assuming Conjecture \ref{limit-k}), we get in this way more general cohomology classes which satisfy the Novikov conjecture. 
\end{remark}

\noindent
{\bf Applications to the BC Novikov conjecture.}
We now explain how Theorem \ref{main} can be used to investigate the Baum-Connes Novikov conjecture, that is the Novikov conjecture for foliations.  We do this by generalizing the construction in Example \ref{LuszFol}.  Choose a complete smooth transversal $T$ to the foliation $(M, F)$ and denote by $B\cG_T^T$  the classifying space of the groupoid $\cG_T^T$ which is the reduced (to $T$) homotopy groupoid.  $\cG_T^T$ consists of elements of $\cG$ which start and end on $T$.
It is well known that $B\cG_T^T$ classifies  free and proper actions of $\cG_T^T$, so that the principal $\cG_T^T$ bundle $\cG_T$ (which consists of elements of $\cG$ which start on $T$) over $M$ is the pull-back, by a (up to homotopy well defined) map
$\varphi:M\to B\cG_T^T$, of a universal $\cG_T^T$ bundle 
$E\cG_T^T$ over $B\cG_T^T$.  More precisely, we have an action of $\cG_T^T$ on $E\cG_T^T$ on the right $E\cG_T^T \times_{s_B} \cG_T^T \to E\cG_T^T$, denoted $x\gamma$ for $(x, \gamma)\in E\cG_T^T \times_{s_B} \cG_T^T$, 
where 
$$
E\cG_T^T \times_{s_B} \cG_T^T := \{(x, \gamma)\in E\cG_T^T\times \cG_T^T, s_B(x)=r(\gamma)\},
$$
and $s_B:E\cG_T^T \to T$,  $r_B: E\cG_T^T\to B\cG_T^T$ satisfy
$$
s_B \circ \tvarphi = s, \quad  s_B (x\gamma) = s(\gamma) \quad  \text{ and }  \quad  r_B\circ \tvarphi = \varphi \circ r.
$$
where $s: \cG_T \to T$ and $r: \cG_T \to M$ are the source and range maps, and $\tvarphi:\cG_T\to E\cG_T^T$ is the $\cG_T^T$-equivariant classifying map which covers $\varphi$. 
So, we have the picture
$$
T \stackrel{s_B}{\longleftarrow} E\cG_T^T \stackrel{r_B}{\longrightarrow} B\cG_T^T.
$$
The fibers of the submersion $s_B$ are contractible and this identifies the universal principal bundle $E\cG_T^T$, see \cite{ConnesBook}, 
pages 126-127. 

\begin{definition}
A $\cG_T^T$-equivariant Hermitian bundle $(E_0, \{\cdot, \cdot\})$ is a complex vector bundle $\pi_0: E_0\to T$ endowed with a (possibly indefinite) non-degenerate metric $\{\cdot, \cdot\}$ together with an action of $\cG_T^T$ which preserves the metric.
\end{definition}
So if we set  
$$
\cG_T^T \times_T E_0 \,\, := \,\, \{(\alpha, u)\in \cG_T^T\times E_0, s(\alpha) = \pi_0(u)\}   \,\, = \,\,(s\,|\, \cG^T_T)^*E_0,
$$
then there is a smooth map $h:\cG_T^T \times_T E_0 \to E_0$ such that $\pi_0\circ h (\alpha, u) = r(\alpha)$ and for any $\alpha\in \cG_T^T$, the map $h_\alpha (u):= h(\alpha, u)$ is a linear map from $E_{0, s(\alpha)}$ to $E_{0, r(\alpha)}$ which preserves the metric $\{\cdot, \cdot\}$. It is understood that $h$ is an action in the sense that 
$$
h_{\alpha\beta} = h_\alpha\circ h_\beta, \quad \text{ if }r(\beta)=s(\alpha).
$$
Given a $\cG_T^T$ Hermitian bundle $(E_0, \{\cdot, \cdot\})$, we define a Hermitian bundle over the classifying space $B\cG_T^T$ whose total space is
$$
E=E\cG_T^T\times_{\cG_T^T} E_0.
$$
Here $E$ is the quotient manifold where we have identified $(x, u)$ with $(x\alpha, h(\alpha^{-1}, u))$, for any $\alpha \in \cG_T^T$ such that 
$$
s(\alpha)=\pi_0(u) \text{ and } r(\alpha)= s_B(x).
$$

Note that Example \ref{LuszFol} falls into this class where we take $T = K$,  a single fiber of $\wN \times_{\pi_1(N)} K$  and where the Hermitian bundle $E_0$ is trivial and equivariant through the cocycle $h$. Finally, for general Riemannian foliations, the holonomy action of $\cG_T^T$ 
on the transverse  bundle to the foliation, and on all functorially defined bundles obtained from it, gives an example of a $\cG_T^T$-equivariant Hermitian bundle.  

\begin{definition}
For any $\cG_T^T$-equivariant Hermitian bundle $(E_0, \{\cdot, \cdot\})$, the vector bundle $E$ over the classifying space $B\cG_T^T$ will be called a 
Hermitian leafwise flat bundle.
\end{definition}

This terminology is explained by the following.  Recall that $\varphi: M \to B\cG_T^T$ is a classifying map for the foliation $F$.

\begin{lemma}
  The complex vector bundle $\varphi^*E$ over $M$ admits a leafwise flat structure which preserves the induced (possibly indefinite) metric.
\end{lemma}

\begin{proof}
We may assume that the vector bundle $\varphi^*E$ is smooth and is isomorphic to $\cG_T \times_{\cG_T^T} E_0$. Since the action of $\cG_T^T$ preserves the metric $\{\cdot, \cdot\}$, there is a well defined metric on $E\to M$ which is induced from $\{\cdot, \cdot\}$. The usual proof, using for instance properness of the action of $\cG_T^T$ on $\cG_T$, allows the construction of a connection on $E$ which is leafwise flat and which preserves the (possibly indefinite) non-degenerate metric on $E$.    
\end{proof}
As usual, the complex bundle $E$ splits into a direct sum of unitary vector bundles $E=E^+ \oplus E^-$ which are not leafwise flat in general.  We say that the leafwise flat bundle $E$  is bounded if the leafwise parallel translation along the leafwise flat connection of $E$ is a bounded map. 
\begin{theorem}\label{BCNov}
Assume that the foliation $(M, F)$ is Riemannian, oriented, and transversely oriented. Then for any Hermitian bounded leafwise flat   bundle $E$ over $B\cG_T^T$, the Chern character $\ch (E^+) - \ch(E^-) \in H^*(B\cG_T^T;\R)$ satisfies the BC Novikov conjecture.
\end{theorem}
\begin{proof}
The bundle $\varphi^* E$ is a leafwise Hermitian flat  bundle for the smooth foliation $(M,F)$, and by our assumption of boundedness, the parallel translation along the leaves is bounded, so the projection onto the twisted leafwise harmonics is transversely smooth.  Let $f: (M', F') \to (M, F)$ be a leafwise oriented, leafwise homotopy equivalence (which also preserves the transverse orientations).  Then $f^*(\varphi^* E) = (\varphi \circ f)^* E$ is also  a bounded leafwise Hermitian flat  bundle, so projection onto the twisted leafwise harmonics for $(M', F')$ is also transversely smooth. 
Applying Theorem \ref{main}, we get
$$
\sigma(M', F'; (\varphi \circ f)^*([E^+] - [E^-])) =  f^*\sigma (M, F; \varphi^*([E^+] - [E^-])) \quad \in  \,\, H_c^*(M'/F').
$$
Since the foliation is transversely oriented, there is a well defined transverse fundamental class, namely the holonomy invariant closed current $[M'/F']$ 
which is given by  integration over the transversals of $(M', F')$. Applying $[M'/F']$ to the above equality and using the fact that $[M'/F'] \circ f^* = [M/F]$ (since $f$ preserves the  transverse orientations) we  get
$$
\left<[M'/F'], \sigma(M', F'; (\varphi\circ f)^*([E^+] - [E^-]))\right> = \left< [M/F],  \sigma (M, F; \varphi^*([E^+] - [E^-]))\right>.
$$
But Conjecture  \ref{limit-k} gives 
$$
\sigma (M, F; \varphi^*([E^+] - [E^-])) = \int_F \L (TF) \wedge \varphi^* \ch ([E^+] - [E^-]),
$$
and
$$
\sigma(M', F'; (\varphi\circ f)^*([E^+] - [E^-])) = \int_{F'}\L (TF') \wedge (\varphi\circ f)^* \ch ([E^+] - [E^-]).
$$
Since $\dd [M/F] \circ \int_F = \int_M$ and $\dd [M'/F'] \circ \int_{F'} = \int_{M'}$, the conclusion follows, namely
$$
\int_M  \L (TF) \wedge \varphi^* \ch ([E^+] - [E^-]) = \int_{M'}
\L (TF') \wedge (\varphi\circ f)^* \ch ([E^+] - [E^-]).
$$
\end{proof}

\end{document}